\documentclass[english]{amsart}

\usepackage[latin1]{inputenc} 
\usepackage{setspace} 
\usepackage{amssymb}
\usepackage[]{epsf,epsfig,amsmath,amssymb,amsfonts,latexsym}
\usepackage{tikz}

\usepackage{amsthm}
\usepackage{mathtools}
\usepackage{lmodern}
\usepackage{mathabx}
\usepackage{stmaryrd}
\usepackage{mathrsfs}
\usepackage{wasysym}
\usepackage{bbm}		%
\usepackage[inline,shortlabels]{enumitem}
\usepackage{nccmath}	%
\usepackage{braket}		%
\usepackage[normalem]{ulem}		%
\usepackage{nicefrac}
\usepackage{cancel}
\usepackage{xcolor}
\usepackage[a4paper, left=2.7cm, right=2.7cm, top=2.7cm, bottom=2cm]{geometry}
\usepackage{hyperref}
\usepackage{doi}
\usepackage[numbers]{natbib}
\usepackage{cleveref}
\usepackage{leftidx} %
\usetikzlibrary{shapes,arrows,automata,matrix,fit,patterns}
\usepackage{tikz-cd}  %
\usepackage{todonotes}

\theoremstyle{plain}
\newtheorem{theorem}{Theorem}[section]
\newtheorem{lemma}[theorem]{Lemma}
\newtheorem{corollary}[theorem]{Corollary}
\newtheorem{proposition}[theorem]{Proposition}
\newtheorem{observation}[theorem]{Observation}

\newtheorem{definition}[theorem]{Definition}

\theoremstyle{definition}
\newtheorem{remark}[theorem]{Remark}
\newtheorem{example}[theorem]{Example}

\newtheorem{maintheorem}{Theorem}
\newtheorem{maincorollary}{Corollary}
\newtheorem{mainconjecture}{Conjecture}
\newtheorem{mainquestion}{Question}

\newtheorem*{THEC}{Theorem~\ref{maintheorem:ddim-complexity-is-FminusS}}

\newcommand{\ZZ}{\mathbb{Z}}			%
\newcommand{\ZZd}{{\mathbb{Z}^d}}		%
\newcommand{\NN}{\mathbb{N}}			%
\newcommand{\RR}{\mathbb{R}}			%
\newcommand{\QQ}{\mathbb{Q}}			%

\newcommand{\symb}[1]{\mathtt{#1}}		%
\newcommand{\isdef}{\coloneqq}			%
\DeclarePairedDelimiter\norm{\lVert}{\rVert}	%

\DeclareMathOperator{\Ker}{Ker}

\newcommand{\indicator}[1]{\mathbbm{1}_{#1}}					%
\newcommand{\A}{\mathcal{A}}

\newcommand{\Lcal}{\mathcal{L}}
\newcommand{\Pcal}{\mathcal{P}}
\newcommand{\Xcal}{\mathcal{X}}

\newcommand{\balpha}{\alpha}
\newcommand{\bk}{k}
\newcommand{\bn}{n}
\newcommand{\bm}{m}
\newcommand{\be}{e}
\newcommand{\bg}{g}
\newcommand{\bu}{u}
\newcommand{\bv}{v}
\newcommand{\bzero}{0}
\newcommand{\occ}{\mathsf{occ}}
\newcommand{\Orb}{\mathrm{Orb}}

\newcommand{\alfa}[1]{\{\symb{0},\symb{1},\dots,#1\}}
\newcommand{\alfad}{\{\symb{0},\symb{1},\dots,d\}}
\newcommand{\dynsys}[3]{#2\overset{#3}{\curvearrowright}#1}

\newcommand{\xConfig}[2]{%
	\begin{tikzpicture}[
		baseline=-\the\dimexpr\fontdimen22\textfont2\relax,ampersand replacement=\&]
		\matrix[
			matrix of math nodes,
			nodes={
				minimum size=1.2ex,text width=1.2ex,
				text height=1.2ex,inner sep=3pt,draw={gray!20},align=center,
				anchor=base
			}, row sep=1pt,column sep=1pt
		] (config) {#1};
		\node[draw,rectangle,dashed,help lines,fit=(config), inner sep=0.5ex] {};
		#2
	\end{tikzpicture}
}

\newcommand{\define}[1]{\textbf{#1}}

\title[Indistinguishable asymptotic pairs over $\mathbb{Z}^d$]
      {Indistinguishable asymptotic pairs and\\multidimensional Sturmian configurations}

\author[S.~Barbieri]{Sebasti\'an Barbieri}
\address[S.~Barbieri]{Departamento de Matem\'atica y Ciencia de la
Computaci\'on, Universidad de Santiago de Chile, Las Sophoras 173. Estaci\'on
Central. Santiago. Chile.}
\email{sebastian.barbieri@usach.cl}

\author[S.~Labb\'e]{S\'ebastien Labb\'e}
\address[S.~Labb\'e]{Univ. Bordeaux, CNRS,  Bordeaux INP, LaBRI, UMR 5800, F-33400, Talence, France}
\email{sebastien.labbe@labri.fr}

\makeatletter
\@namedef{subjclassname@2020}{\textup{2020} Mathematics Subject Classification}
\makeatother



\date{}

\begin{document}

\begin{abstract}
Two asymptotic configurations on a full $\mathbb{Z}^d$-shift are indistinguishable if for every finite pattern the associated sets of occurrences in each configuration coincide up to a finitely supported permutation of $\mathbb{Z}^d$. We prove that indistinguishable asymptotic pairs satisfying a ``flip condition'' are characterized by their pattern complexity on finite connected supports. Furthermore, we prove that uniformly recurrent indistinguishable asymptotic pairs satisfying the flip condition are described by codimension-one (dimension of the internal space) cut and project schemes, which symbolically correspond to multidimensional Sturmian configurations. Together the two results provide a generalization to $\mathbb{Z}^d$ of the characterization of Sturmian sequences by their factor complexity $n+1$. Many open questions are raised by the current work and are listed in the introduction.
\end{abstract}
	
    \maketitle
	
    \setcounter{tocdepth}{1}
	\tableofcontents

	\section{Introduction}

	Asymptotic pairs, also known as homoclinic pairs, are pairs of points in a dynamical system whose orbits coalesce. These were first studied by Poincar\'e~\cite{Andersson1994} in the context of the three body problem and used to model chaotic behavior. Namely, two orbits which remain arbitrarily close outside a finite window of time may be used to represent pairs of trajectories that despite having similar behavior for an arbitrarily long time, present abrupt local differences.
	
	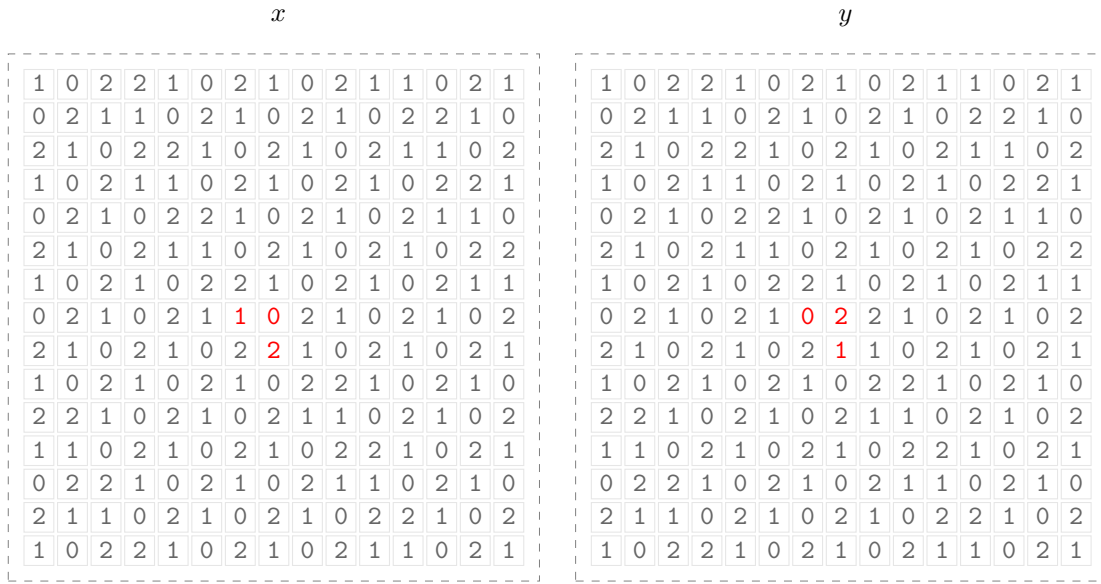
\begin{figure}[ht]
		\begin{center}
			\begin{tabular}{cc}
                $x$  & $y$\\[3mm]
				\begin{tikzpicture}
[baseline=-\the\dimexpr\fontdimen22\textfont2\relax,ampersand replacement=\&]
  \matrix[matrix of math nodes,nodes={
       minimum size=1.2ex,text width=1.2ex,
       text height=1.2ex,inner sep=3pt,draw={gray!20},align=center,
       anchor=base
     }, row sep=1pt,column sep=1pt
  ] (config) {
{\color{black!60}\symb{1}}\&{\color{black!60}\symb{0}}\&{\color{black!60}\symb{2}}\&{\color{black!60}\symb{2}}\&{\color{black!60}\symb{1}}\&{\color{black!60}\symb{0}}\&{\color{black!60}\symb{2}}\&{\color{black!60}\symb{1}}\&{\color{black!60}\symb{0}}\&{\color{black!60}\symb{2}}\&{\color{black!60}\symb{1}}\&{\color{black!60}\symb{1}}\&{\color{black!60}\symb{0}}\&{\color{black!60}\symb{2}}\&{\color{black!60}\symb{1}}\\
{\color{black!60}\symb{0}}\&{\color{black!60}\symb{2}}\&{\color{black!60}\symb{1}}\&{\color{black!60}\symb{1}}\&{\color{black!60}\symb{0}}\&{\color{black!60}\symb{2}}\&{\color{black!60}\symb{1}}\&{\color{black!60}\symb{0}}\&{\color{black!60}\symb{2}}\&{\color{black!60}\symb{1}}\&{\color{black!60}\symb{0}}\&{\color{black!60}\symb{2}}\&{\color{black!60}\symb{2}}\&{\color{black!60}\symb{1}}\&{\color{black!60}\symb{0}}\\
{\color{black!60}\symb{2}}\&{\color{black!60}\symb{1}}\&{\color{black!60}\symb{0}}\&{\color{black!60}\symb{2}}\&{\color{black!60}\symb{2}}\&{\color{black!60}\symb{1}}\&{\color{black!60}\symb{0}}\&{\color{black!60}\symb{2}}\&{\color{black!60}\symb{1}}\&{\color{black!60}\symb{0}}\&{\color{black!60}\symb{2}}\&{\color{black!60}\symb{1}}\&{\color{black!60}\symb{1}}\&{\color{black!60}\symb{0}}\&{\color{black!60}\symb{2}}\\
{\color{black!60}\symb{1}}\&{\color{black!60}\symb{0}}\&{\color{black!60}\symb{2}}\&{\color{black!60}\symb{1}}\&{\color{black!60}\symb{1}}\&{\color{black!60}\symb{0}}\&{\color{black!60}\symb{2}}\&{\color{black!60}\symb{1}}\&{\color{black!60}\symb{0}}\&{\color{black!60}\symb{2}}\&{\color{black!60}\symb{1}}\&{\color{black!60}\symb{0}}\&{\color{black!60}\symb{2}}\&{\color{black!60}\symb{2}}\&{\color{black!60}\symb{1}}\\
{\color{black!60}\symb{0}}\&{\color{black!60}\symb{2}}\&{\color{black!60}\symb{1}}\&{\color{black!60}\symb{0}}\&{\color{black!60}\symb{2}}\&{\color{black!60}\symb{2}}\&{\color{black!60}\symb{1}}\&{\color{black!60}\symb{0}}\&{\color{black!60}\symb{2}}\&{\color{black!60}\symb{1}}\&{\color{black!60}\symb{0}}\&{\color{black!60}\symb{2}}\&{\color{black!60}\symb{1}}\&{\color{black!60}\symb{1}}\&{\color{black!60}\symb{0}}\\
{\color{black!60}\symb{2}}\&{\color{black!60}\symb{1}}\&{\color{black!60}\symb{0}}\&{\color{black!60}\symb{2}}\&{\color{black!60}\symb{1}}\&{\color{black!60}\symb{1}}\&{\color{black!60}\symb{0}}\&{\color{black!60}\symb{2}}\&{\color{black!60}\symb{1}}\&{\color{black!60}\symb{0}}\&{\color{black!60}\symb{2}}\&{\color{black!60}\symb{1}}\&{\color{black!60}\symb{0}}\&{\color{black!60}\symb{2}}\&{\color{black!60}\symb{2}}\\
{\color{black!60}\symb{1}}\&{\color{black!60}\symb{0}}\&{\color{black!60}\symb{2}}\&{\color{black!60}\symb{1}}\&{\color{black!60}\symb{0}}\&{\color{black!60}\symb{2}}\&{\color{black!60}\symb{2}}\&{\color{black!60}\symb{1}}\&{\color{black!60}\symb{0}}\&{\color{black!60}\symb{2}}\&{\color{black!60}\symb{1}}\&{\color{black!60}\symb{0}}\&{\color{black!60}\symb{2}}\&{\color{black!60}\symb{1}}\&{\color{black!60}\symb{1}}\\
{\color{black!60}\symb{0}}\&{\color{black!60}\symb{2}}\&{\color{black!60}\symb{1}}\&{\color{black!60}\symb{0}}\&{\color{black!60}\symb{2}}\&{\color{black!60}\symb{1}}\&{\color{red}\symb{1}}\&{\color{red}\symb{0}}\&{\color{black!60}\symb{2}}\&{\color{black!60}\symb{1}}\&{\color{black!60}\symb{0}}\&{\color{black!60}\symb{2}}\&{\color{black!60}\symb{1}}\&{\color{black!60}\symb{0}}\&{\color{black!60}\symb{2}}\\
{\color{black!60}\symb{2}}\&{\color{black!60}\symb{1}}\&{\color{black!60}\symb{0}}\&{\color{black!60}\symb{2}}\&{\color{black!60}\symb{1}}\&{\color{black!60}\symb{0}}\&{\color{black!60}\symb{2}}\&{\color{red}\symb{2}}\&{\color{black!60}\symb{1}}\&{\color{black!60}\symb{0}}\&{\color{black!60}\symb{2}}\&{\color{black!60}\symb{1}}\&{\color{black!60}\symb{0}}\&{\color{black!60}\symb{2}}\&{\color{black!60}\symb{1}}\\
{\color{black!60}\symb{1}}\&{\color{black!60}\symb{0}}\&{\color{black!60}\symb{2}}\&{\color{black!60}\symb{1}}\&{\color{black!60}\symb{0}}\&{\color{black!60}\symb{2}}\&{\color{black!60}\symb{1}}\&{\color{black!60}\symb{0}}\&{\color{black!60}\symb{2}}\&{\color{black!60}\symb{2}}\&{\color{black!60}\symb{1}}\&{\color{black!60}\symb{0}}\&{\color{black!60}\symb{2}}\&{\color{black!60}\symb{1}}\&{\color{black!60}\symb{0}}\\
{\color{black!60}\symb{2}}\&{\color{black!60}\symb{2}}\&{\color{black!60}\symb{1}}\&{\color{black!60}\symb{0}}\&{\color{black!60}\symb{2}}\&{\color{black!60}\symb{1}}\&{\color{black!60}\symb{0}}\&{\color{black!60}\symb{2}}\&{\color{black!60}\symb{1}}\&{\color{black!60}\symb{1}}\&{\color{black!60}\symb{0}}\&{\color{black!60}\symb{2}}\&{\color{black!60}\symb{1}}\&{\color{black!60}\symb{0}}\&{\color{black!60}\symb{2}}\\
{\color{black!60}\symb{1}}\&{\color{black!60}\symb{1}}\&{\color{black!60}\symb{0}}\&{\color{black!60}\symb{2}}\&{\color{black!60}\symb{1}}\&{\color{black!60}\symb{0}}\&{\color{black!60}\symb{2}}\&{\color{black!60}\symb{1}}\&{\color{black!60}\symb{0}}\&{\color{black!60}\symb{2}}\&{\color{black!60}\symb{2}}\&{\color{black!60}\symb{1}}\&{\color{black!60}\symb{0}}\&{\color{black!60}\symb{2}}\&{\color{black!60}\symb{1}}\\
{\color{black!60}\symb{0}}\&{\color{black!60}\symb{2}}\&{\color{black!60}\symb{2}}\&{\color{black!60}\symb{1}}\&{\color{black!60}\symb{0}}\&{\color{black!60}\symb{2}}\&{\color{black!60}\symb{1}}\&{\color{black!60}\symb{0}}\&{\color{black!60}\symb{2}}\&{\color{black!60}\symb{1}}\&{\color{black!60}\symb{1}}\&{\color{black!60}\symb{0}}\&{\color{black!60}\symb{2}}\&{\color{black!60}\symb{1}}\&{\color{black!60}\symb{0}}\\
{\color{black!60}\symb{2}}\&{\color{black!60}\symb{1}}\&{\color{black!60}\symb{1}}\&{\color{black!60}\symb{0}}\&{\color{black!60}\symb{2}}\&{\color{black!60}\symb{1}}\&{\color{black!60}\symb{0}}\&{\color{black!60}\symb{2}}\&{\color{black!60}\symb{1}}\&{\color{black!60}\symb{0}}\&{\color{black!60}\symb{2}}\&{\color{black!60}\symb{2}}\&{\color{black!60}\symb{1}}\&{\color{black!60}\symb{0}}\&{\color{black!60}\symb{2}}\\
{\color{black!60}\symb{1}}\&{\color{black!60}\symb{0}}\&{\color{black!60}\symb{2}}\&{\color{black!60}\symb{2}}\&{\color{black!60}\symb{1}}\&{\color{black!60}\symb{0}}\&{\color{black!60}\symb{2}}\&{\color{black!60}\symb{1}}\&{\color{black!60}\symb{0}}\&{\color{black!60}\symb{2}}\&{\color{black!60}\symb{1}}\&{\color{black!60}\symb{1}}\&{\color{black!60}\symb{0}}\&{\color{black!60}\symb{2}}\&{\color{black!60}\symb{1}}\\
};
\node[draw,rectangle,dashed,help lines,fit=(config), inner sep=0.5ex] {};
\end{tikzpicture} &
                \begin{tikzpicture}
[baseline=-\the\dimexpr\fontdimen22\textfont2\relax,ampersand replacement=\&]
  \matrix[matrix of math nodes,nodes={
       minimum size=1.2ex,text width=1.2ex,
       text height=1.2ex,inner sep=3pt,draw={gray!20},align=center,
       anchor=base
     }, row sep=1pt,column sep=1pt
  ] (config) {
{\color{black!60}\symb{1}}\&{\color{black!60}\symb{0}}\&{\color{black!60}\symb{2}}\&{\color{black!60}\symb{2}}\&{\color{black!60}\symb{1}}\&{\color{black!60}\symb{0}}\&{\color{black!60}\symb{2}}\&{\color{black!60}\symb{1}}\&{\color{black!60}\symb{0}}\&{\color{black!60}\symb{2}}\&{\color{black!60}\symb{1}}\&{\color{black!60}\symb{1}}\&{\color{black!60}\symb{0}}\&{\color{black!60}\symb{2}}\&{\color{black!60}\symb{1}}\\
{\color{black!60}\symb{0}}\&{\color{black!60}\symb{2}}\&{\color{black!60}\symb{1}}\&{\color{black!60}\symb{1}}\&{\color{black!60}\symb{0}}\&{\color{black!60}\symb{2}}\&{\color{black!60}\symb{1}}\&{\color{black!60}\symb{0}}\&{\color{black!60}\symb{2}}\&{\color{black!60}\symb{1}}\&{\color{black!60}\symb{0}}\&{\color{black!60}\symb{2}}\&{\color{black!60}\symb{2}}\&{\color{black!60}\symb{1}}\&{\color{black!60}\symb{0}}\\
{\color{black!60}\symb{2}}\&{\color{black!60}\symb{1}}\&{\color{black!60}\symb{0}}\&{\color{black!60}\symb{2}}\&{\color{black!60}\symb{2}}\&{\color{black!60}\symb{1}}\&{\color{black!60}\symb{0}}\&{\color{black!60}\symb{2}}\&{\color{black!60}\symb{1}}\&{\color{black!60}\symb{0}}\&{\color{black!60}\symb{2}}\&{\color{black!60}\symb{1}}\&{\color{black!60}\symb{1}}\&{\color{black!60}\symb{0}}\&{\color{black!60}\symb{2}}\\
{\color{black!60}\symb{1}}\&{\color{black!60}\symb{0}}\&{\color{black!60}\symb{2}}\&{\color{black!60}\symb{1}}\&{\color{black!60}\symb{1}}\&{\color{black!60}\symb{0}}\&{\color{black!60}\symb{2}}\&{\color{black!60}\symb{1}}\&{\color{black!60}\symb{0}}\&{\color{black!60}\symb{2}}\&{\color{black!60}\symb{1}}\&{\color{black!60}\symb{0}}\&{\color{black!60}\symb{2}}\&{\color{black!60}\symb{2}}\&{\color{black!60}\symb{1}}\\
{\color{black!60}\symb{0}}\&{\color{black!60}\symb{2}}\&{\color{black!60}\symb{1}}\&{\color{black!60}\symb{0}}\&{\color{black!60}\symb{2}}\&{\color{black!60}\symb{2}}\&{\color{black!60}\symb{1}}\&{\color{black!60}\symb{0}}\&{\color{black!60}\symb{2}}\&{\color{black!60}\symb{1}}\&{\color{black!60}\symb{0}}\&{\color{black!60}\symb{2}}\&{\color{black!60}\symb{1}}\&{\color{black!60}\symb{1}}\&{\color{black!60}\symb{0}}\\
{\color{black!60}\symb{2}}\&{\color{black!60}\symb{1}}\&{\color{black!60}\symb{0}}\&{\color{black!60}\symb{2}}\&{\color{black!60}\symb{1}}\&{\color{black!60}\symb{1}}\&{\color{black!60}\symb{0}}\&{\color{black!60}\symb{2}}\&{\color{black!60}\symb{1}}\&{\color{black!60}\symb{0}}\&{\color{black!60}\symb{2}}\&{\color{black!60}\symb{1}}\&{\color{black!60}\symb{0}}\&{\color{black!60}\symb{2}}\&{\color{black!60}\symb{2}}\\
{\color{black!60}\symb{1}}\&{\color{black!60}\symb{0}}\&{\color{black!60}\symb{2}}\&{\color{black!60}\symb{1}}\&{\color{black!60}\symb{0}}\&{\color{black!60}\symb{2}}\&{\color{black!60}\symb{2}}\&{\color{black!60}\symb{1}}\&{\color{black!60}\symb{0}}\&{\color{black!60}\symb{2}}\&{\color{black!60}\symb{1}}\&{\color{black!60}\symb{0}}\&{\color{black!60}\symb{2}}\&{\color{black!60}\symb{1}}\&{\color{black!60}\symb{1}}\\
{\color{black!60}\symb{0}}\&{\color{black!60}\symb{2}}\&{\color{black!60}\symb{1}}\&{\color{black!60}\symb{0}}\&{\color{black!60}\symb{2}}\&{\color{black!60}\symb{1}}\&{\color{red}\symb{0}}\&{\color{red}\symb{2}}\&{\color{black!60}\symb{2}}\&{\color{black!60}\symb{1}}\&{\color{black!60}\symb{0}}\&{\color{black!60}\symb{2}}\&{\color{black!60}\symb{1}}\&{\color{black!60}\symb{0}}\&{\color{black!60}\symb{2}}\\
{\color{black!60}\symb{2}}\&{\color{black!60}\symb{1}}\&{\color{black!60}\symb{0}}\&{\color{black!60}\symb{2}}\&{\color{black!60}\symb{1}}\&{\color{black!60}\symb{0}}\&{\color{black!60}\symb{2}}\&{\color{red}\symb{1}}\&{\color{black!60}\symb{1}}\&{\color{black!60}\symb{0}}\&{\color{black!60}\symb{2}}\&{\color{black!60}\symb{1}}\&{\color{black!60}\symb{0}}\&{\color{black!60}\symb{2}}\&{\color{black!60}\symb{1}}\\
{\color{black!60}\symb{1}}\&{\color{black!60}\symb{0}}\&{\color{black!60}\symb{2}}\&{\color{black!60}\symb{1}}\&{\color{black!60}\symb{0}}\&{\color{black!60}\symb{2}}\&{\color{black!60}\symb{1}}\&{\color{black!60}\symb{0}}\&{\color{black!60}\symb{2}}\&{\color{black!60}\symb{2}}\&{\color{black!60}\symb{1}}\&{\color{black!60}\symb{0}}\&{\color{black!60}\symb{2}}\&{\color{black!60}\symb{1}}\&{\color{black!60}\symb{0}}\\
{\color{black!60}\symb{2}}\&{\color{black!60}\symb{2}}\&{\color{black!60}\symb{1}}\&{\color{black!60}\symb{0}}\&{\color{black!60}\symb{2}}\&{\color{black!60}\symb{1}}\&{\color{black!60}\symb{0}}\&{\color{black!60}\symb{2}}\&{\color{black!60}\symb{1}}\&{\color{black!60}\symb{1}}\&{\color{black!60}\symb{0}}\&{\color{black!60}\symb{2}}\&{\color{black!60}\symb{1}}\&{\color{black!60}\symb{0}}\&{\color{black!60}\symb{2}}\\
{\color{black!60}\symb{1}}\&{\color{black!60}\symb{1}}\&{\color{black!60}\symb{0}}\&{\color{black!60}\symb{2}}\&{\color{black!60}\symb{1}}\&{\color{black!60}\symb{0}}\&{\color{black!60}\symb{2}}\&{\color{black!60}\symb{1}}\&{\color{black!60}\symb{0}}\&{\color{black!60}\symb{2}}\&{\color{black!60}\symb{2}}\&{\color{black!60}\symb{1}}\&{\color{black!60}\symb{0}}\&{\color{black!60}\symb{2}}\&{\color{black!60}\symb{1}}\\
{\color{black!60}\symb{0}}\&{\color{black!60}\symb{2}}\&{\color{black!60}\symb{2}}\&{\color{black!60}\symb{1}}\&{\color{black!60}\symb{0}}\&{\color{black!60}\symb{2}}\&{\color{black!60}\symb{1}}\&{\color{black!60}\symb{0}}\&{\color{black!60}\symb{2}}\&{\color{black!60}\symb{1}}\&{\color{black!60}\symb{1}}\&{\color{black!60}\symb{0}}\&{\color{black!60}\symb{2}}\&{\color{black!60}\symb{1}}\&{\color{black!60}\symb{0}}\\
{\color{black!60}\symb{2}}\&{\color{black!60}\symb{1}}\&{\color{black!60}\symb{1}}\&{\color{black!60}\symb{0}}\&{\color{black!60}\symb{2}}\&{\color{black!60}\symb{1}}\&{\color{black!60}\symb{0}}\&{\color{black!60}\symb{2}}\&{\color{black!60}\symb{1}}\&{\color{black!60}\symb{0}}\&{\color{black!60}\symb{2}}\&{\color{black!60}\symb{2}}\&{\color{black!60}\symb{1}}\&{\color{black!60}\symb{0}}\&{\color{black!60}\symb{2}}\\
{\color{black!60}\symb{1}}\&{\color{black!60}\symb{0}}\&{\color{black!60}\symb{2}}\&{\color{black!60}\symb{2}}\&{\color{black!60}\symb{1}}\&{\color{black!60}\symb{0}}\&{\color{black!60}\symb{2}}\&{\color{black!60}\symb{1}}\&{\color{black!60}\symb{0}}\&{\color{black!60}\symb{2}}\&{\color{black!60}\symb{1}}\&{\color{black!60}\symb{1}}\&{\color{black!60}\symb{0}}\&{\color{black!60}\symb{2}}\&{\color{black!60}\symb{1}}\\
};
\node[draw,rectangle,dashed,help lines,fit=(config), inner sep=0.5ex] {};
\end{tikzpicture} %
			\end{tabular}
		\end{center}
        \caption{The indistinguishable asymptotic configurations
        $x,y\in\{\symb{0},\symb{1},\symb{2}\}^{\ZZ^2}$ are shown on the support 
            $\llbracket -7,7\rrbracket \times \llbracket -7,7\rrbracket$. 
            The two configurations are equal except on their difference set
            $F=\{0,-\be_1,-\be_2\}$ shown in red.}
		\label{fig:intro-sturmian-config-pair}
	\end{figure}

	In this work we consider asymptotic pairs of zero-dimensional expansive actions of $\ZZd$. Concretely, given a finite set $\Sigma$, we consider the space of configurations $\Sigma^{\ZZd} = \{ x \colon \ZZd \to \Sigma\}$ endowed with the prodiscrete topology and the shift action $\ZZd \overset{\sigma}{\curvearrowright} \Sigma^{\ZZd}$. In this setting, two configurations $x,y \in \Sigma^{\ZZd}$ are \define{asymptotic} if $x$ and $y$ differ in finitely many sites of $\ZZd$. The finite set $F = \{ \bv \in \ZZd : x_\bv \neq y_\bv\}$ is called the \define{difference set} of $(x,y)$. An example of an asymptotic pair when $d=2$ is shown in \Cref{fig:intro-sturmian-config-pair}.

    Given two asymptotic configurations $x,y \in \Sigma^{\ZZd}$, we want to compare the
    number of occurrences of patterns.  A pattern is a function
    $p \colon S \to \Sigma$ where $S$, called the \define{support} of $p$, is a finite subset of $\ZZd$.
    The occurrences of a pattern $p\in\Sigma^S$ in a configuration
    $x\in\Sigma^\ZZd$ is the set $\occ_p(x)\isdef \{n\in\ZZd\colon \sigma^n(x)|_S = p\}$.
    The \define{language} of a configuration
    $x\in\Sigma^\ZZd$ over a finite support $S\subset\ZZd$ is
    $\Lcal_S(x)=\{p\in\Sigma^S\colon \occ_p(x)\neq\varnothing\}$.
    When $x,y \in \Sigma^{\ZZd}$ are asymptotic configurations,
    the difference $\occ_p(x)\setminus \occ_p(y)$ is finite
    because the occurrences of $p$ are the same far from the difference set of $x$ and $y$.
    We say that $(x,y)$ is
    an \define{indistinguishable asymptotic pair} 
    if $(x,y)$ is asymptotic and 
    the following equality holds
    \begin{equation}\label{eq:indistinguishable-introduction}
        \#\left(\occ_p(x)\setminus \occ_p(y)\right) = 
        \#\left(\occ_p(y)\setminus \occ_p(x)\right)
    \end{equation}
    for every pattern $p$ of finite support. 

    In other words, an asymptotic pair $(x,y)$ is indistinguishable if every
    pattern appears the same number of times in $x$ and in $y$ while
    overlapping the difference set. 
    The pair of configurations $x$ and $y$ shown
    in \Cref{fig:intro-sturmian-config-pair} is an example of an
    indistinguishable asymptotic pair:
    we may check by hand that \Cref{eq:indistinguishable-introduction} holds
    for patterns with small supports such as symbols (patterns of shape
    $\{0\}$), dominoes (patterns of shape $\{0,\be_1\}$ and $\{0,\be_2\}$),
    etc.
    For instance, the configurations $x$ and $y$ in
    \Cref{fig:intro-sturmian-config-pair} contain eight different patterns with support $\{0,\be_1,2\be_1,\be_2\}$, each occurring exactly once while overlapping the difference
    set, see \Cref{fig:8-L-shaped-patterns-in-x-y}.

\begin{figure}[ht]
\begin{center}
    \input{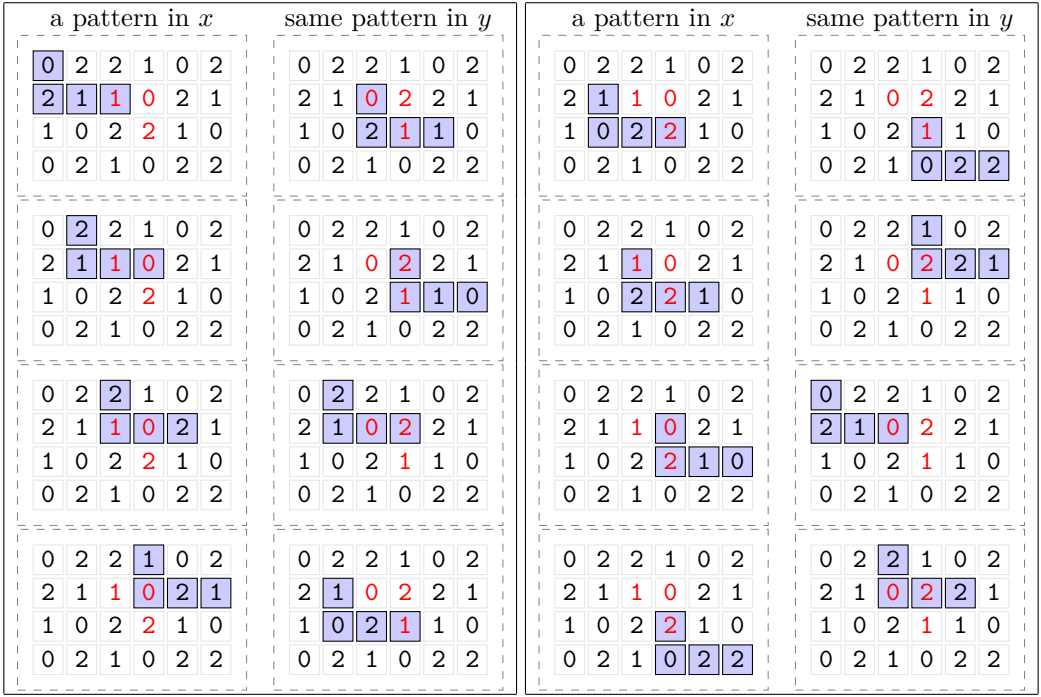}
\end{center}
    \caption{The 8 patterns of shape $\{0,\be_1,2\be_1,\be_2\}$ appearing in the 
    configurations $x$ and $y$. All of them appear intersecting the difference
    set in $x$ and $y$.}
    \label{fig:8-L-shaped-patterns-in-x-y}
\end{figure}

	The notion of indistinguishable asymptotic pairs appears naturally in Gibbs theory. This theory studies measures on symbolic dynamical systems which are at equilibrium in the sense that the conditional pressure for every finite region of
	the lattice is maximized, so that every finite region is in equilibrium with its surrounding. See~\cite{Geo88,LanRue69,Rue04,BarGomMarTaa20} for further background. An important component of Gibbs measures, the \define{specification}, can be formalized by means of a shift-invariant cocycle in the equivalence relation of asymptotic pairs, see~\cite{ChaMey16,Barbieri_Gomez_Marcus_Meyerovitch_Taati_2020}. With an appropriate norm, the space of continuous shift invariant cocycles on the asymptotic relation becomes a Banach space, and every asymptotic pair induces a continuous linear functional through the canonical evaluation map.
	
	The set of indistinguishable asymptotic pairs are precisely those which induce the trivial linear functional and thus a natural question is if there is an underlying dynamical structure behind this property. We shall not speak any further of Gibbs theory in this work and study indistinguishable asymptotic pairs without further reference to their origin in Gibbs theory. An interested reader can find out more about the role of indistinguishable asymptotic pairs in the aforementioned setting by reading sections 2 and 3 of~\cite{Barbieri_Gomez_Marcus_Meyerovitch_Taati_2020}.

	In the case of dimension $d=1$, it was shown that for the difference set $F= \{-1,0\}\subset \ZZ$, indistinguishable asymptotic pairs are precisely the \'etale limits of characteristic bi-infinite Sturmian sequences (\cite[Theorem B]{BarLabSta2021}). In the case where one of the configurations in the indistinguishable pair is recurrent, the asymptotic pair can only be a pair of characteristic bi-infinite Sturmian sequences associated to a fixed irrational value (\cite[Theorem A]{BarLabSta2021}). Furthermore, it was shown that any indistinguishable asymptotic pair in $\Sigma^{\ZZ}$ can be obtained from these base cases through the use of a substitution and the shift map (\cite[Theorem C]{BarLabSta2021}), thus providing a full characterization of indistinguishable asymptotic pairs in~$\ZZ$.

    \subsection*{Main results}
	
    In this article, we extend \cite[Theorem A]{BarLabSta2021} to the
    multidimensional setting. 
    It is based on the following additional condition made on the difference set.
    Let $\{\be_1,\ldots, \be_d\}$ denote the canonical basis of $\ZZd$. We say
    that two indistinguishable asymptotic configurations $x,y \in
    \alfad^{\ZZd}$ satisfy the \define{flip condition} if their difference
    set is $F = \{ \bzero, -\be_1,\dots,-\be_d   \}$, 
    every symbol in
    $\alfad$ occurs in $x$ and $y$ at the support $F$, and
    the map defined by $x_\bn\mapsto y_\bn$ for every $\bn\in F$
    is a cyclic permutation on the alphabet $\alfad$.
    Without lost of generality, we assume that $x_{\bzero} = \symb{0}$ and
    $y_\bn = x_\bn - \symb{1} \bmod (d+\symb{1})$ for every $\bn\in F$.
    For example, the indistinguishable asymptotic pair $(x,y)$ illustrated
    in \Cref{fig:intro-sturmian-config-pair} satisfies the flip condition
    with $(x_0,x_{-\be_1},x_{-\be_2})=(\symb{0},\symb{1},\symb{2})$
    and $(y_0,y_{-\be_1},y_{-\be_2})=(\symb{2},\symb{0},\symb{1})$.

    It is a well known fact that Sturmian configurations in dimension one can be characterized by their complexity~\cite{MR0000745,MR0322838}, that is, they are exactly the bi-infinite recurrent words in which exactly $n+1$ subwords of length $n$ occur for every $n \in \NN$. We first prove the following result providing a similar characterization of 
    indistinguishable
    asymptotic pairs satisfying the flip condition
    by their pattern complexity which does not require uniform recurrence, or even recurrence, as an hypothesis.

\newcommand\MainTheoremComplexityStatement{
    Let $d\geq1$ and $x,y \in\alfad^{\ZZd}$ be an asymptotic pair
    satisfying the flip condition with difference set 
    $F = \{ \bzero, -\be_1,\dots,-\be_d\}$. The following are equivalent:
    \begin{enumerate}[(i)]
        \item For every nonempty finite connected subset $S\subset\ZZd$ and $p
            \in \Lcal_S(x)\cup\Lcal_S(y)$, we have
            \[
                \#\left(\occ_p(x)\setminus \occ_p(y)\right)
                = 1 =
                \#\left(\occ_p(y)\setminus \occ_p(x)\right).
            \]
		\item The asymptotic pair $(x,y)$ is indistinguishable.
        \item For every nonempty finite connected subset $S\subset\ZZd$, the
            pattern complexity of $x$ and $y$ is \[\#\Lcal_S(x)=\#\Lcal_S(y) = \#(F-S).\]
    \end{enumerate}
}

\begin{maintheorem} \label{maintheorem:ddim-complexity-is-FminusS}
    \MainTheoremComplexityStatement
\end{maintheorem}

The proof of \Cref{maintheorem:ddim-complexity-is-FminusS} relies on an extension of
the notion of bispecial factor to the setting of multidimensional
configurations. Given a language, a bispecial factor is a  word that can be extended in more than one way
to the left and to the right. The bilateral
multiplicity of bispecial factors in a one-dimensional language is closely related to
the complexity of that language, see~\cite{MR2759107}. Here, for a
connected support $S\subset\ZZd$ and two distinct positions $a,b\in\ZZd\setminus S$ such that
$S\cup\{a\}$, $S\cup\{b\}$ and $S \cup \{a,b\}$ are connected, we say that a pattern
$w\colon S\to\A$ is bispecial if it can be extended in more than one way at
position $a$ and at position $b$. The description of the bispecial patterns of indistinguishable asymptotic pairs and their multiplicities,
provides us a tool for bounding their pattern complexity.
Reciprocally, the rigid pattern complexity given in~\Cref{maintheorem:ddim-complexity-is-FminusS} 
forces the extension graphs associated to the bispecial patterns to have no cycle, which in turn provides us a way to show that the configurations are indistinguishable.
In one dimension, sequences such as the extension graphs of bispecial factors are trees
are known as dendric words~\cite{MR3845381} and thus we may think of our construction as multidimensional analogues of those.

When $S$ is a $d$-dimensional rectangular block, the number $\#(F-S)$ from~\Cref{maintheorem:ddim-complexity-is-FminusS} admits a nice form.
When $d=1$, we compute $\#(F-S)=\#(\{0,-1\}-\{0,1,\dots,n-1\})=n+1$
which is the factor complexity function for
$1$-dimensional Sturmian words.
When $d=2$, $\#(F-S)=\#(\{(0,0),(-1,0),(0,-1)\}-\{(i,j)\colon 0\leq i<n, 0\leq j<m\})=mn+m+n$
is the rectangular pattern complexity of a 
discrete plane with totally irrational (irrational and rationally independent) slope, see~\cite{MR1782038} for further references. 
With our result above, we can provide an explicit formula for every dimension.

\begin{maincorollary}\label{maincorollary:ddim-complexity}
    Let $d\geq1$ and $(m_1,\dots,m_d)\in \NN^d$.
	The rectangular pattern complexity 
    of an indistinguishable asymptotic pair $x,y \in\alfad^{\ZZd}$ satisfying the flip condition
    is
	\begin{equation*}\label{eq:ddim-complexity}
		\#\Lcal_{(m_1,\dots,m_d)}(x) =\#\Lcal_{(m_1,\dots,m_d)}(y)= m_1\cdots m_d \left(1+\frac{1}{m_1}+\dots+\frac{1}{m_d}\right).
	\end{equation*}
\end{maincorollary}

    Our main result provides a beautiful connection between indistinguishable
    asymptotic pairs satisfying the flip condition and
    codimension-one (dimension of the internal space) cut and project schemes,
    see~\cite{MR3480345} for further background,
    and more precisely with multidimensional Sturmian configurations.
    The definition of
    multidimensional Sturmian configurations from
    codimension-one cut and project schemes is fully described in
    \Cref{sec:alternate-def-cut-and-project}.
    A quick and easy definition of multidimensional Sturmian configurations
    can be given with the following formulas.
	Given a totally irrational vector $\balpha=(\alpha_1,\dots,\alpha_d)\in[0,1)^d$, the \define{lower} and \define{upper characteristic $d$-dimensional Sturmian configurations} with slope $\alpha$ are given by
    \begin{equation}\label{eq:characteristic-sturmian-formula-in-intro}
        \begin{array}{rccl}
        c_{\balpha}:&\ZZd & \to & \alfad\\
        &\bn & \mapsto &  \sum\limits_{i=1}^d \left(\lfloor\alpha_i+\bn\cdot\balpha\rfloor
        -\lfloor\bn\cdot\balpha\rfloor\right)
        \end{array}
        \text{ and }
        \begin{array}{rccl}
        c'_{\balpha}:&\ZZd & \to & \alfad\\
        &\bn & \mapsto & \sum\limits_{i=1}^d \left(\lceil\alpha_i+\bn\cdot\balpha\rceil
        -\lceil\bn\cdot\balpha\rceil\right).
        \end{array}
    \end{equation}
	It turns out that these configurations are examples of indistinguishable asymptotic pairs which satisfy the flip condition. In fact, we show that a pair of uniformly recurrent asymptotic configurations is indistinguishable and satisfies the flip condition if and only if it is a pair of characteristic $d$-dimensional Sturmian configurations for some totally irrational slope.

	\begin{maintheorem}\label{thm:multidim_sturmian_characterization}
        Let $d\geq1$ and $x,y\in\alfad^\ZZd$ such that $x$ is uniformly recurrent.
        The pair $(x,y)$ is an indistinguishable asymptotic pair satisfying the flip condition 
        if and only if 
        there exists a totally irrational vector $\alpha \in [0,1)^d$ such that $x=c_{\alpha}$
        and $y=c'_{\alpha}$ are the lower and upper characteristic
        $d$-dimensional Sturmian configurations with slope $\alpha$.
	\end{maintheorem}

    The indistinguishable asymptotic pair 
    shown in
	\Cref{fig:intro-sturmian-config-pair}
    is an example as such, where
    $x=c_{\balpha}$ and
    $y=c'_{\balpha}$ with $\balpha=(\alpha_1,\alpha_2)=(\sqrt{2}/2,\sqrt{19}-4)$.
    Notice that $c_\alpha$ and $c'_\alpha$ are uniformly
    recurrent when $\alpha$ contains at least an irrational coordinate (see
    \Cref{lem:sturmian_is_unirec}), so that hypothesis is really only used in
    one direction of the theorem. 
    Note that a version of \Cref{thm:multidim_sturmian_characterization} for
    rational vector $\alpha\in\QQ^d$ was considered in \cite{MR3351761} with an
    infinite difference set of the form $F+K$ where $K\subset\ZZd$ is some
    lattice.

    The link with codimension-one cut and project schemes can be illustrated as
    follows.
    The configurations 
    $x=c_{\balpha}$ and
    $y=c'_{\balpha}$
    encode the rhombi obtained as the projection
    of the cube faces in a
    discrete plane of normal vector $(1-\alpha_1,\alpha_1-\alpha_2,\alpha_2)$,
    see \Cref{fig:intro-sturmian-config-pair-discrete-plane}.
    This three symbol coding of a discrete plane was proposed in \cite{Jamet_2005_coding},
    see also \cite{MR2742675}.

    \begin{figure}[ht]
		\begin{center}
			\begin{tabular}{cc}
                $x$  & $y$\\%[3mm]
                \includegraphics[width=.4\linewidth]{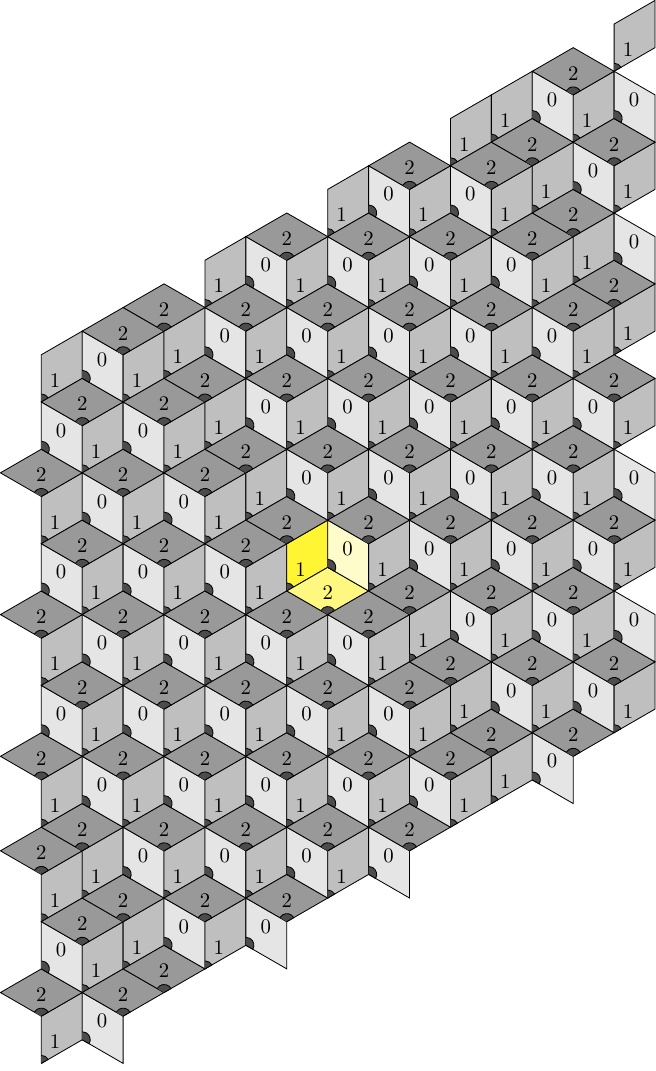} &
				\includegraphics[width=.4\linewidth]{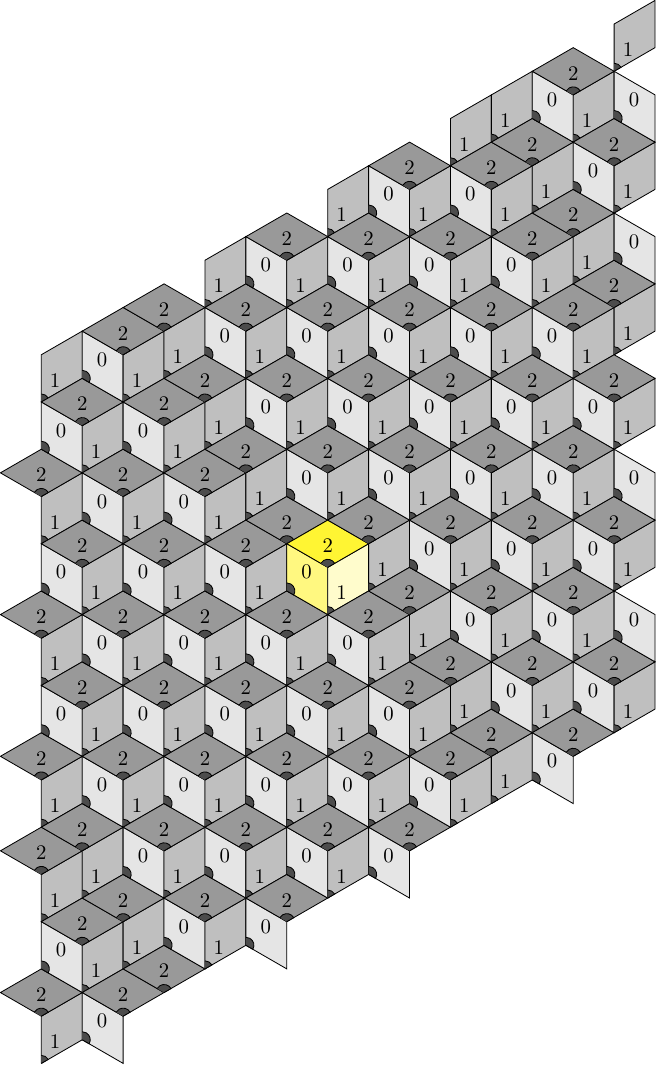}
			\end{tabular}
		\end{center}
        \caption{The configurations $x$ and $y$ from
            \Cref{fig:intro-sturmian-config-pair} are encoding a tiling of the
            plane \cite{MR2074952}
			by three types of pointed rhombus drawn using Jolivet's notation
            \cite[p.~112]{jolivet_phd_2013}.
            The tilings shown above correspond to the projection of the surface of a
            discrete plane of normal vector 
            $(1-\alpha_1,\alpha_1-\alpha_2,\alpha_2)
            \approx(0.293, 0.348, 0.359)$,
            with $\balpha=(\alpha_1,\alpha_2)=(\sqrt{2}/2,\sqrt{19}-4)$,
            in 3 dimensional space, and their difference can be interpreted as
            the flip of a unit cube shown in yellow.}
		\label{fig:intro-sturmian-config-pair-discrete-plane}
	\end{figure}

    We also prove a slightly more general version
    of~\Cref{thm:multidim_sturmian_characterization}.
    We say that two indistinguishable asymptotic configurations $x,y \in
    \Sigma^{\ZZd}$ satisfy the \define{affine flip condition} if their
    difference set $F$ has cardinality  $\#F = d+1$, there is $m \in F$ such
    that $(F-m)\setminus\{0\}$ is a base of $\ZZd$,
    the restriction $x|_F$ is a bijection $F\to\Sigma$
    and
    the map defined by
    $x_\bn\mapsto y_\bn$ for every $\bn\in F$
    is a cyclic permutation on $\Sigma$.

	\begin{maincorollary}\label{maincorollary:affine_multidim_sturmian_characterization}
        Let $d\geq1$ and $x,y\in\Sigma^\ZZd$ such that $x$ is uniformly recurrent.
        The pair $(x,y)$ is an indistinguishable asymptotic pair satisfying the
        affine flip condition if and only if 
        there exist a bijection $\tau\colon \alfad\to\Sigma$,
        an invertible affine transformation $A\in \operatorname{Aff}(\ZZd)$
        and
        a totally irrational vector $\alpha\in[0,1)^d$
        such that $x=\tau\circ c_{\alpha}\circ A$ and $y=\tau\circ c'_{\alpha}\circ A$.
	\end{maincorollary}

If we further assume that the configurations in the asymptotic pair are uniformly recurrent, we can put together~\Cref{maintheorem:ddim-complexity-is-FminusS} and~\Cref{thm:multidim_sturmian_characterization} and obtain the following characterization of uniformly recurrent multidimensional Sturmian configurations in terms of their pattern complexity. This generalizes the well-known theorem of
Morse-Hedlund-Coven to higher dimensions \cite{MR0000745,MR0322838}.

\begin{maincorollary}\label{corollary:language-vs-existence-alpha}
	Let $d\geq1$ and $x,y \in\alfad^{\ZZd}$ be an asymptotic pair such that $x$ is uniformly recurrent and which satisfies the flip condition with difference set 
	$F = \{ \bzero, -\be_1,\dots,-\be_d\}$. The following are equivalent:
	\begin{enumerate}[(i)]
        \item For every nonempty finite connected subset $S\subset\ZZd$ and $p
            \in \Lcal_S(x,y)$, we have
            \[
                \#\left(\occ_p(x)\setminus \occ_p(y)\right)
                = 1 =
                \#\left(\occ_p(y)\setminus \occ_p(x)\right).
            \]
		\item The asymptotic pair $(x,y)$ is indistinguishable.
        \item For every nonempty finite connected subset $S\subset\ZZd$, we
            have \[\#\Lcal_S(x)=\#\Lcal_S(y) = \#(F-S).\]
        \item There exists a totally irrational vector $\alpha \in [0,1)^d$
            such that $x = c_{\alpha}$ and $y=c'_{\alpha}$.
	\end{enumerate}
\end{maincorollary}

\subsection*{Open questions}

To fully generalize the theorem of Morse-Hedlund-Coven, we would hope the
equivalence holds for single configurations and not only for pairs of asymptotic
configurations satisfying the flip condition.
More precisely, in the case of uniformly recurrent configurations, we believe
that the pattern complexity characterizes multidimensional Sturmian
configurations. The formula defining 
$s_{\alpha,\rho}$ and
$s'_{\alpha,\rho}$ slightly extends 
\Cref{eq:characteristic-sturmian-formula-in-intro}
and can be found in \Cref{lem:sturmian-formual-from-cut-and-project}.

\begin{mainquestion}
	Let $d\geq1$ and $x \in\alfad^{\ZZd}$ be uniformly recurrent configuration.
	Let $F = \{ \bzero, -\be_1,\dots,-\be_d\}$. 
    Consider the following two statements:
	\begin{enumerate}[(i)]
        \item for every nonempty finite connected subset $S\subset\ZZd$, we
            have $\#\Lcal_S(x)= \#(F-S)$.
        \item there exists a totally irrational vector $\alpha \in [0,1)^d$ and
            a intercept $\rho\in[0,1)$ such that $x = s_{\alpha,\rho}$ or
            $x=s'_{\alpha,\rho}$.
	\end{enumerate}
    Since $s_{\alpha,\rho}$, $s'_{\alpha,\rho}$ and $c_\alpha$ have the same language 
    when $\alpha$ is totally irrational,
    we can deduce from \Cref{corollary:language-vs-existence-alpha} that (ii) implies (i). 
    Is it true that (i) and (ii) are equivalent?
\end{mainquestion}

Consider a sequence of totally irrational slopes $(\alpha_n)_{n \in \NN}$ for which both $c_{\alpha_n}$ and $c'_{\alpha_n}$ converge in the prodiscrete topology. Then $(c_{\alpha_n},c'_{\alpha_n})_{n \in \NN}$ converges in the asymptotic relation to an \'etale limit $(c,c')$,  see~\Cref{def:etale}. It turns out that \'etale limits preserve both the flip condition and indistinguishability, and will thus satisfy all of the equivalences stated in~\Cref{maintheorem:ddim-complexity-is-FminusS}. An example of such a limit is illustrated in~\Cref{fig:etale}.

\begin{figure}[!ht]
	\begin{center}
		\begin{tabular}{cc}
			$c$  & $c'$\\[3mm]
			\begin{tikzpicture}
[baseline=-\the\dimexpr\fontdimen22\textfont2\relax,ampersand replacement=\&]
  \matrix[matrix of math nodes,nodes={
       minimum size=1.2ex,text width=1.2ex,
       text height=1.2ex,inner sep=3pt,draw={gray!20},align=center,
       anchor=base
     }, row sep=1pt,column sep=1pt
  ] (config) {
{\color{black!60}\symb{0}}\&{\color{black!60}\symb{0}}\&{\color{black!60}\symb{0}}\&{\color{black!60}\symb{0}}\&{\color{black!60}\symb{0}}\&{\color{black!60}\symb{0}}\&{\color{black!60}\symb{0}}\&{\color{black!60}\symb{0}}\&{\color{black!60}\symb{0}}\&{\color{black!60}\symb{0}}\&{\color{black!60}\symb{0}}\&{\color{black!60}\symb{0}}\&{\color{black!60}\symb{0}}\&{\color{black!60}\symb{0}}\&{\color{black!60}\symb{0}}\\
{\color{black!60}\symb{0}}\&{\color{black!60}\symb{0}}\&{\color{black!60}\symb{0}}\&{\color{black!60}\symb{0}}\&{\color{black!60}\symb{0}}\&{\color{black!60}\symb{0}}\&{\color{black!60}\symb{0}}\&{\color{black!60}\symb{0}}\&{\color{black!60}\symb{0}}\&{\color{black!60}\symb{0}}\&{\color{black!60}\symb{0}}\&{\color{black!60}\symb{0}}\&{\color{black!60}\symb{0}}\&{\color{black!60}\symb{0}}\&{\color{black!60}\symb{0}}\\
{\color{black!60}\symb{0}}\&{\color{black!60}\symb{0}}\&{\color{black!60}\symb{0}}\&{\color{black!60}\symb{0}}\&{\color{black!60}\symb{0}}\&{\color{black!60}\symb{0}}\&{\color{black!60}\symb{0}}\&{\color{black!60}\symb{0}}\&{\color{black!60}\symb{0}}\&{\color{black!60}\symb{0}}\&{\color{black!60}\symb{0}}\&{\color{black!60}\symb{0}}\&{\color{black!60}\symb{0}}\&{\color{black!60}\symb{0}}\&{\color{black!60}\symb{0}}\\
{\color{black!60}\symb{0}}\&{\color{black!60}\symb{0}}\&{\color{black!60}\symb{0}}\&{\color{black!60}\symb{0}}\&{\color{black!60}\symb{0}}\&{\color{black!60}\symb{0}}\&{\color{black!60}\symb{0}}\&{\color{black!60}\symb{0}}\&{\color{black!60}\symb{0}}\&{\color{black!60}\symb{0}}\&{\color{black!60}\symb{0}}\&{\color{black!60}\symb{0}}\&{\color{black!60}\symb{0}}\&{\color{black!60}\symb{0}}\&{\color{black!60}\symb{0}}\\
{\color{black!60}\symb{1}}\&{\color{black!60}\symb{2}}\&{\color{black!60}\symb{0}}\&{\color{black!60}\symb{0}}\&{\color{black!60}\symb{0}}\&{\color{black!60}\symb{0}}\&{\color{black!60}\symb{0}}\&{\color{black!60}\symb{0}}\&{\color{black!60}\symb{0}}\&{\color{black!60}\symb{0}}\&{\color{black!60}\symb{0}}\&{\color{black!60}\symb{0}}\&{\color{black!60}\symb{0}}\&{\color{black!60}\symb{0}}\&{\color{black!60}\symb{0}}\\
{\color{black!60}\symb{0}}\&{\color{black!60}\symb{0}}\&{\color{black!60}\symb{1}}\&{\color{black!60}\symb{2}}\&{\color{black!60}\symb{0}}\&{\color{black!60}\symb{0}}\&{\color{black!60}\symb{0}}\&{\color{black!60}\symb{0}}\&{\color{black!60}\symb{0}}\&{\color{black!60}\symb{0}}\&{\color{black!60}\symb{0}}\&{\color{black!60}\symb{0}}\&{\color{black!60}\symb{0}}\&{\color{black!60}\symb{0}}\&{\color{black!60}\symb{0}}\\
{\color{black!60}\symb{0}}\&{\color{black!60}\symb{0}}\&{\color{black!60}\symb{0}}\&{\color{black!60}\symb{0}}\&{\color{black!60}\symb{1}}\&{\color{black!60}\symb{2}}\&{\color{black!60}\symb{0}}\&{\color{black!60}\symb{0}}\&{\color{black!60}\symb{0}}\&{\color{black!60}\symb{0}}\&{\color{black!60}\symb{0}}\&{\color{black!60}\symb{0}}\&{\color{black!60}\symb{0}}\&{\color{black!60}\symb{0}}\&{\color{black!60}\symb{0}}\\
{\color{black!60}\symb{0}}\&{\color{black!60}\symb{0}}\&{\color{black!60}\symb{0}}\&{\color{black!60}\symb{0}}\&{\color{black!60}\symb{0}}\&{\color{black!60}\symb{0}}\&{\color{red}\symb{2}}\&{\color{red}\symb{0}}\&{\color{black!60}\symb{0}}\&{\color{black!60}\symb{0}}\&{\color{black!60}\symb{0}}\&{\color{black!60}\symb{0}}\&{\color{black!60}\symb{0}}\&{\color{black!60}\symb{0}}\&{\color{black!60}\symb{0}}\\
{\color{black!60}\symb{0}}\&{\color{black!60}\symb{0}}\&{\color{black!60}\symb{0}}\&{\color{black!60}\symb{0}}\&{\color{black!60}\symb{0}}\&{\color{black!60}\symb{0}}\&{\color{black!60}\symb{0}}\&{\color{red}\symb{1}}\&{\color{black!60}\symb{2}}\&{\color{black!60}\symb{0}}\&{\color{black!60}\symb{0}}\&{\color{black!60}\symb{0}}\&{\color{black!60}\symb{0}}\&{\color{black!60}\symb{0}}\&{\color{black!60}\symb{0}}\\
{\color{black!60}\symb{0}}\&{\color{black!60}\symb{0}}\&{\color{black!60}\symb{0}}\&{\color{black!60}\symb{0}}\&{\color{black!60}\symb{0}}\&{\color{black!60}\symb{0}}\&{\color{black!60}\symb{0}}\&{\color{black!60}\symb{0}}\&{\color{black!60}\symb{0}}\&{\color{black!60}\symb{1}}\&{\color{black!60}\symb{2}}\&{\color{black!60}\symb{0}}\&{\color{black!60}\symb{0}}\&{\color{black!60}\symb{0}}\&{\color{black!60}\symb{0}}\\
{\color{black!60}\symb{0}}\&{\color{black!60}\symb{0}}\&{\color{black!60}\symb{0}}\&{\color{black!60}\symb{0}}\&{\color{black!60}\symb{0}}\&{\color{black!60}\symb{0}}\&{\color{black!60}\symb{0}}\&{\color{black!60}\symb{0}}\&{\color{black!60}\symb{0}}\&{\color{black!60}\symb{0}}\&{\color{black!60}\symb{0}}\&{\color{black!60}\symb{1}}\&{\color{black!60}\symb{2}}\&{\color{black!60}\symb{0}}\&{\color{black!60}\symb{0}}\\
{\color{black!60}\symb{0}}\&{\color{black!60}\symb{0}}\&{\color{black!60}\symb{0}}\&{\color{black!60}\symb{0}}\&{\color{black!60}\symb{0}}\&{\color{black!60}\symb{0}}\&{\color{black!60}\symb{0}}\&{\color{black!60}\symb{0}}\&{\color{black!60}\symb{0}}\&{\color{black!60}\symb{0}}\&{\color{black!60}\symb{0}}\&{\color{black!60}\symb{0}}\&{\color{black!60}\symb{0}}\&{\color{black!60}\symb{1}}\&{\color{black!60}\symb{2}}\\
{\color{black!60}\symb{0}}\&{\color{black!60}\symb{0}}\&{\color{black!60}\symb{0}}\&{\color{black!60}\symb{0}}\&{\color{black!60}\symb{0}}\&{\color{black!60}\symb{0}}\&{\color{black!60}\symb{0}}\&{\color{black!60}\symb{0}}\&{\color{black!60}\symb{0}}\&{\color{black!60}\symb{0}}\&{\color{black!60}\symb{0}}\&{\color{black!60}\symb{0}}\&{\color{black!60}\symb{0}}\&{\color{black!60}\symb{0}}\&{\color{black!60}\symb{0}}\\
{\color{black!60}\symb{0}}\&{\color{black!60}\symb{0}}\&{\color{black!60}\symb{0}}\&{\color{black!60}\symb{0}}\&{\color{black!60}\symb{0}}\&{\color{black!60}\symb{0}}\&{\color{black!60}\symb{0}}\&{\color{black!60}\symb{0}}\&{\color{black!60}\symb{0}}\&{\color{black!60}\symb{0}}\&{\color{black!60}\symb{0}}\&{\color{black!60}\symb{0}}\&{\color{black!60}\symb{0}}\&{\color{black!60}\symb{0}}\&{\color{black!60}\symb{0}}\\
{\color{black!60}\symb{0}}\&{\color{black!60}\symb{0}}\&{\color{black!60}\symb{0}}\&{\color{black!60}\symb{0}}\&{\color{black!60}\symb{0}}\&{\color{black!60}\symb{0}}\&{\color{black!60}\symb{0}}\&{\color{black!60}\symb{0}}\&{\color{black!60}\symb{0}}\&{\color{black!60}\symb{0}}\&{\color{black!60}\symb{0}}\&{\color{black!60}\symb{0}}\&{\color{black!60}\symb{0}}\&{\color{black!60}\symb{0}}\&{\color{black!60}\symb{0}}\\
};
\node[draw,rectangle,dashed,help lines,fit=(config), inner sep=0.5ex] {};
\end{tikzpicture} &
			\begin{tikzpicture}
[baseline=-\the\dimexpr\fontdimen22\textfont2\relax,ampersand replacement=\&]
  \matrix[matrix of math nodes,nodes={
       minimum size=1.2ex,text width=1.2ex,
       text height=1.2ex,inner sep=3pt,draw={gray!20},align=center,
       anchor=base
     }, row sep=1pt,column sep=1pt
  ] (config) {
{\color{black!60}\symb{0}}\&{\color{black!60}\symb{0}}\&{\color{black!60}\symb{0}}\&{\color{black!60}\symb{0}}\&{\color{black!60}\symb{0}}\&{\color{black!60}\symb{0}}\&{\color{black!60}\symb{0}}\&{\color{black!60}\symb{0}}\&{\color{black!60}\symb{0}}\&{\color{black!60}\symb{0}}\&{\color{black!60}\symb{0}}\&{\color{black!60}\symb{0}}\&{\color{black!60}\symb{0}}\&{\color{black!60}\symb{0}}\&{\color{black!60}\symb{0}}\\
{\color{black!60}\symb{0}}\&{\color{black!60}\symb{0}}\&{\color{black!60}\symb{0}}\&{\color{black!60}\symb{0}}\&{\color{black!60}\symb{0}}\&{\color{black!60}\symb{0}}\&{\color{black!60}\symb{0}}\&{\color{black!60}\symb{0}}\&{\color{black!60}\symb{0}}\&{\color{black!60}\symb{0}}\&{\color{black!60}\symb{0}}\&{\color{black!60}\symb{0}}\&{\color{black!60}\symb{0}}\&{\color{black!60}\symb{0}}\&{\color{black!60}\symb{0}}\\
{\color{black!60}\symb{0}}\&{\color{black!60}\symb{0}}\&{\color{black!60}\symb{0}}\&{\color{black!60}\symb{0}}\&{\color{black!60}\symb{0}}\&{\color{black!60}\symb{0}}\&{\color{black!60}\symb{0}}\&{\color{black!60}\symb{0}}\&{\color{black!60}\symb{0}}\&{\color{black!60}\symb{0}}\&{\color{black!60}\symb{0}}\&{\color{black!60}\symb{0}}\&{\color{black!60}\symb{0}}\&{\color{black!60}\symb{0}}\&{\color{black!60}\symb{0}}\\
{\color{black!60}\symb{0}}\&{\color{black!60}\symb{0}}\&{\color{black!60}\symb{0}}\&{\color{black!60}\symb{0}}\&{\color{black!60}\symb{0}}\&{\color{black!60}\symb{0}}\&{\color{black!60}\symb{0}}\&{\color{black!60}\symb{0}}\&{\color{black!60}\symb{0}}\&{\color{black!60}\symb{0}}\&{\color{black!60}\symb{0}}\&{\color{black!60}\symb{0}}\&{\color{black!60}\symb{0}}\&{\color{black!60}\symb{0}}\&{\color{black!60}\symb{0}}\\
{\color{black!60}\symb{1}}\&{\color{black!60}\symb{2}}\&{\color{black!60}\symb{0}}\&{\color{black!60}\symb{0}}\&{\color{black!60}\symb{0}}\&{\color{black!60}\symb{0}}\&{\color{black!60}\symb{0}}\&{\color{black!60}\symb{0}}\&{\color{black!60}\symb{0}}\&{\color{black!60}\symb{0}}\&{\color{black!60}\symb{0}}\&{\color{black!60}\symb{0}}\&{\color{black!60}\symb{0}}\&{\color{black!60}\symb{0}}\&{\color{black!60}\symb{0}}\\
{\color{black!60}\symb{0}}\&{\color{black!60}\symb{0}}\&{\color{black!60}\symb{1}}\&{\color{black!60}\symb{2}}\&{\color{black!60}\symb{0}}\&{\color{black!60}\symb{0}}\&{\color{black!60}\symb{0}}\&{\color{black!60}\symb{0}}\&{\color{black!60}\symb{0}}\&{\color{black!60}\symb{0}}\&{\color{black!60}\symb{0}}\&{\color{black!60}\symb{0}}\&{\color{black!60}\symb{0}}\&{\color{black!60}\symb{0}}\&{\color{black!60}\symb{0}}\\
{\color{black!60}\symb{0}}\&{\color{black!60}\symb{0}}\&{\color{black!60}\symb{0}}\&{\color{black!60}\symb{0}}\&{\color{black!60}\symb{1}}\&{\color{black!60}\symb{2}}\&{\color{black!60}\symb{0}}\&{\color{black!60}\symb{0}}\&{\color{black!60}\symb{0}}\&{\color{black!60}\symb{0}}\&{\color{black!60}\symb{0}}\&{\color{black!60}\symb{0}}\&{\color{black!60}\symb{0}}\&{\color{black!60}\symb{0}}\&{\color{black!60}\symb{0}}\\
{\color{black!60}\symb{0}}\&{\color{black!60}\symb{0}}\&{\color{black!60}\symb{0}}\&{\color{black!60}\symb{0}}\&{\color{black!60}\symb{0}}\&{\color{black!60}\symb{0}}\&{\color{red}\symb{1}}\&{\color{red}\symb{2}}\&{\color{black!60}\symb{0}}\&{\color{black!60}\symb{0}}\&{\color{black!60}\symb{0}}\&{\color{black!60}\symb{0}}\&{\color{black!60}\symb{0}}\&{\color{black!60}\symb{0}}\&{\color{black!60}\symb{0}}\\
{\color{black!60}\symb{0}}\&{\color{black!60}\symb{0}}\&{\color{black!60}\symb{0}}\&{\color{black!60}\symb{0}}\&{\color{black!60}\symb{0}}\&{\color{black!60}\symb{0}}\&{\color{black!60}\symb{0}}\&{\color{red}\symb{0}}\&{\color{black!60}\symb{2}}\&{\color{black!60}\symb{0}}\&{\color{black!60}\symb{0}}\&{\color{black!60}\symb{0}}\&{\color{black!60}\symb{0}}\&{\color{black!60}\symb{0}}\&{\color{black!60}\symb{0}}\\
{\color{black!60}\symb{0}}\&{\color{black!60}\symb{0}}\&{\color{black!60}\symb{0}}\&{\color{black!60}\symb{0}}\&{\color{black!60}\symb{0}}\&{\color{black!60}\symb{0}}\&{\color{black!60}\symb{0}}\&{\color{black!60}\symb{0}}\&{\color{black!60}\symb{0}}\&{\color{black!60}\symb{1}}\&{\color{black!60}\symb{2}}\&{\color{black!60}\symb{0}}\&{\color{black!60}\symb{0}}\&{\color{black!60}\symb{0}}\&{\color{black!60}\symb{0}}\\
{\color{black!60}\symb{0}}\&{\color{black!60}\symb{0}}\&{\color{black!60}\symb{0}}\&{\color{black!60}\symb{0}}\&{\color{black!60}\symb{0}}\&{\color{black!60}\symb{0}}\&{\color{black!60}\symb{0}}\&{\color{black!60}\symb{0}}\&{\color{black!60}\symb{0}}\&{\color{black!60}\symb{0}}\&{\color{black!60}\symb{0}}\&{\color{black!60}\symb{1}}\&{\color{black!60}\symb{2}}\&{\color{black!60}\symb{0}}\&{\color{black!60}\symb{0}}\\
{\color{black!60}\symb{0}}\&{\color{black!60}\symb{0}}\&{\color{black!60}\symb{0}}\&{\color{black!60}\symb{0}}\&{\color{black!60}\symb{0}}\&{\color{black!60}\symb{0}}\&{\color{black!60}\symb{0}}\&{\color{black!60}\symb{0}}\&{\color{black!60}\symb{0}}\&{\color{black!60}\symb{0}}\&{\color{black!60}\symb{0}}\&{\color{black!60}\symb{0}}\&{\color{black!60}\symb{0}}\&{\color{black!60}\symb{1}}\&{\color{black!60}\symb{2}}\\
{\color{black!60}\symb{0}}\&{\color{black!60}\symb{0}}\&{\color{black!60}\symb{0}}\&{\color{black!60}\symb{0}}\&{\color{black!60}\symb{0}}\&{\color{black!60}\symb{0}}\&{\color{black!60}\symb{0}}\&{\color{black!60}\symb{0}}\&{\color{black!60}\symb{0}}\&{\color{black!60}\symb{0}}\&{\color{black!60}\symb{0}}\&{\color{black!60}\symb{0}}\&{\color{black!60}\symb{0}}\&{\color{black!60}\symb{0}}\&{\color{black!60}\symb{0}}\\
{\color{black!60}\symb{0}}\&{\color{black!60}\symb{0}}\&{\color{black!60}\symb{0}}\&{\color{black!60}\symb{0}}\&{\color{black!60}\symb{0}}\&{\color{black!60}\symb{0}}\&{\color{black!60}\symb{0}}\&{\color{black!60}\symb{0}}\&{\color{black!60}\symb{0}}\&{\color{black!60}\symb{0}}\&{\color{black!60}\symb{0}}\&{\color{black!60}\symb{0}}\&{\color{black!60}\symb{0}}\&{\color{black!60}\symb{0}}\&{\color{black!60}\symb{0}}\\
{\color{black!60}\symb{0}}\&{\color{black!60}\symb{0}}\&{\color{black!60}\symb{0}}\&{\color{black!60}\symb{0}}\&{\color{black!60}\symb{0}}\&{\color{black!60}\symb{0}}\&{\color{black!60}\symb{0}}\&{\color{black!60}\symb{0}}\&{\color{black!60}\symb{0}}\&{\color{black!60}\symb{0}}\&{\color{black!60}\symb{0}}\&{\color{black!60}\symb{0}}\&{\color{black!60}\symb{0}}\&{\color{black!60}\symb{0}}\&{\color{black!60}\symb{0}}\\
};
\node[draw,rectangle,dashed,help lines,fit=(config), inner sep=0.5ex] {};
\end{tikzpicture} \\ [37mm]
		\end{tabular}
	\end{center}
    \caption{An indistinguishable asymptotic pair $(c,c')$ which satisfies the flip condition obtained by taking the limit of the Sturmian configurations given by $\alpha_n = (\frac{1}{n}(\sqrt{2}-1),\frac{1}{n}(\sqrt{3}-1))$.}
	\label{fig:etale}
\end{figure}
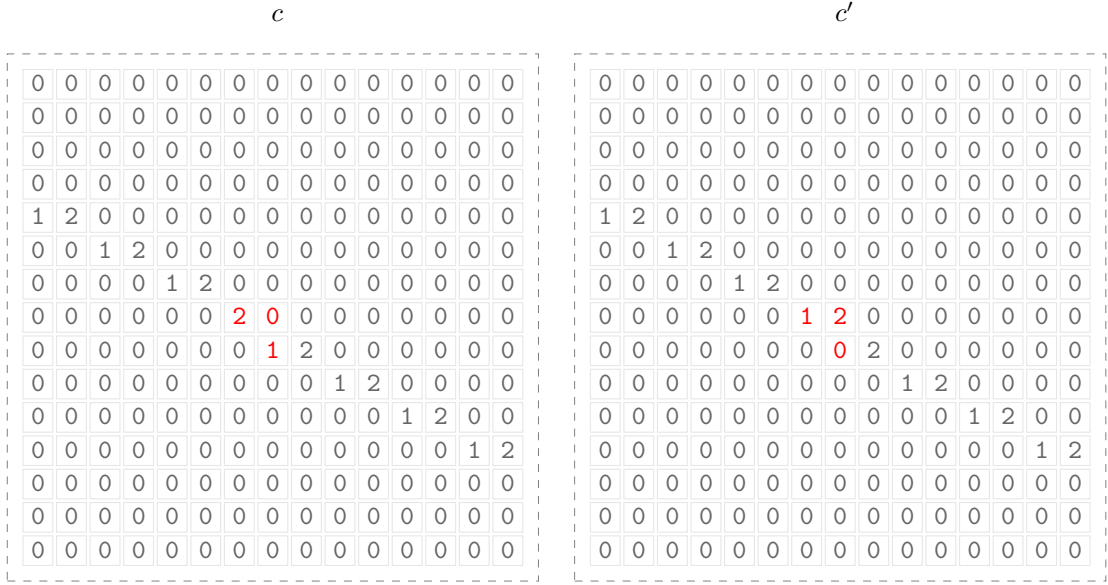

We believe that in fact every indistinguishable asymptotic pair on $\ZZd$ which satisfies the flip condition can be obtained through an \'etale limit as above.

\begin{mainconjecture}\label{conj:etale-limite}
	Let $d\geq1$ and $x,y \in\alfad^{\ZZd}$ be an indistinguishable asymptotic pair which satisfies the flip condition. Then there exists a sequence of totally irrational vectors $(\alpha_n)_{n \in \NN}$ such that $(x,y)$ is the \'etale limit of the sequence of asymptotic pairs $(c_{\alpha_n},c'_{\alpha_n})_{n \in \NN}$.
\end{mainconjecture}

It was proved that \Cref{conj:etale-limite} holds when $d=1$, see
\cite[Theorem B]{BarLabSta2021}. Proving it for $d>1$ is harder due to the
various ways a sequence $(\alpha_n)_{n\in\NN}$ can converge to some vector
$\alpha\in[0,1)^d$ leading to infinitely many \'etale limits associated to a single vector.
When $d=1$, there are only two such ways: from above or from below. Describing
combinatorially what happens in these two cases was sufficient in
\cite{BarLabSta2021} to prove the result. An analogue combinatorial description of all different behaviors when $d>1$ is still open.

In \cite[Theorem C]{BarLabSta2021}, indistinguishable asymptotic pairs were
totally described when $d=1$ by the image under substitutions of characteristic
Sturmian sequences. Describing indistinguishable asymptotic pairs in general
when $d>1$ (other than those satisfying the flip condition or some affine version of it) remains an open
question.

\begin{mainquestion}\label{Q:derived-from}
	Let $d\geq1$ and $x,y \in\alfad^{\ZZd}$ be an indistinguishable asymptotic pair. 
    Does there exists a sequence of totally irrational vectors $(\alpha_n)_{n
    \in \NN}$ such that $(x,y)$ can be derived from the \'etale limit of the sequence of asymptotic pairs
    $(c_{\alpha_n},c'_{\alpha_n})_{n \in \NN}$?
\end{mainquestion}

Our current work also leads to another interesting question.
In dimension 1, it is known at least since~\cite{MR1507944}
that a sequence of factor complexity less than or equal to $n$ is eventually periodic.
In two dimensions, it is still an open problem 
\cite{MR3356945,MR4073398}
known as Nivat's conjecture
\cite{1997-nivat-talk}
whether a configuration $x$
for which there are $n,m \in \NN$ with $\#\Lcal_{(n,m)}(x)\leq nm$
is periodic or not.
Another question which seems to have been overlooked due to the difficulty of
settling Nivat's conjecture is to describe the minimal complexity of an aperiodic
configuration (trivial stabilizer under the shift map, that is $\sigma^n(x)=x$ only holds for $n =0$) which admits a totally irrational vector of symbols frequencies.
When $d=1$, we know that such sequences have complexity at least $n+1$ and are realized by Sturmian configurations.
However, when $d=2$, configurations with rectangular pattern complexity $mn+1$ 
are not uniformly recurrent and
do not have a totally irrational vector of symbol frequencies~\cite{MR1719387}. As the symbol frequencies of the multidimensional Sturmian configurations $c_\alpha$ and $c'_\alpha$ is $\alpha$, it follows by \Cref{thm:multidim_sturmian_characterization} and
\Cref{maintheorem:ddim-complexity-is-FminusS} that they provide an upper bound for this problem, namely, that these sequences can be realized with 
complexity $\#(F-S)$ for every pattern of connected support $S$.
According to Cassaigne and Moutot (personal communication, January 2023),
there exist $2$-dimensional configurations with totally irrational vector
of symbol frequencies with pattern complexity strictly less than $\#(F-S)$ for infinitely many connected supports $S$.
Therefore, we ask the following question.

\begin{mainquestion}
    Let $d\geq1$. Let $x\in\alfad^{\ZZd}$ be a configuration with trivial
    stabilizer and assume that the frequencies of symbols in $x$ 
    exist and form a totally irrational vector.
    Let $S\subset\ZZd$ be a nonempty connected finite support.
    What is the greatest lower bound for the pattern complexity $\#\Lcal_S(x)$?
\end{mainquestion}

It is known that bispecial factors within the language of a Sturmian sequence
of slope $\alpha\in[0,1)$ are related to the convergents of the continued
fraction expansion of $\alpha$ \cite{MR1468450}. Since our work extends the
notion of bispecial factors to the setup of multidimensional Sturmian
configurations (see \Cref{fig:bispecial-pattern}), it is natural to ask the
following question about simultaneous Diophantine approximation
\cite{MR568710}.

\begin{figure}[ht]
\begin{center}
    \newcommand{\symbb}[1]{$\mathtt{#1}$}		%

\begin{tabular}{cccc}

\begin{tikzpicture}
[baseline=-\the\dimexpr\fontdimen22\textfont2\relax,
    fonce/.style={fill=blue!20,draw=black}, 
    normal/.style={draw=gray!20}, 
    vide/.style={draw=none}, 
    ampersand replacement=\&]
  \matrix[nodes={
       minimum size=1.2ex,text width=1.2ex,
       text height=1.2ex,inner sep=3pt,draw={gray!20},align=center,
       anchor=base,
   }, row sep=1pt,column sep=1pt,
  ] (config) {
    \& \& \& \& \\
    \& \& \& \& \node [normal] {\symbb{2}};\\
    \& \& \& \& \node [normal] {\symbb{1}};\\
    \& \& \& \& \node [normal] {\symbb{0}};\\
    \& \& \& \& \node [normal] {\symbb{2}};\\
    \&
    \node [normal] {\symbb{1}};\&
    \node [normal] {\symbb{0}};\&
    \node [normal] {\symbb{2}};\&
    \node [normal] {\symbb{1}};\\
};
\node[draw,rectangle,dashed,help lines,fit=(config), inner sep=0.5ex] {};
\end{tikzpicture}

    &

\begin{tikzpicture}
[baseline=-\the\dimexpr\fontdimen22\textfont2\relax,
    fonce/.style={fill=blue!20,draw=black}, 
    normal/.style={draw=gray!20}, 
    vide/.style={draw=none}, 
    ampersand replacement=\&]
  \matrix[nodes={
       minimum size=1.2ex,text width=1.2ex,
       text height=1.2ex,inner sep=3pt,draw={gray!20},align=center,
       anchor=base,
   }, row sep=1pt,column sep=1pt,
  ] (config) {
    \& \& \& \& \node [fonce] {\symbb{0}};\\
    \& \& \& \& \node [normal] {\symbb{2}};\\
    \& \& \& \& \node [normal] {\symbb{1}};\\
    \& \& \& \& \node [normal] {\symbb{0}};\\
    \& \& \& \& \node [normal] {\symbb{2}};\\
    \node [fonce] {\symbb{1}};\&
    \node [normal] {\symbb{1}};\&
    \node [normal] {\symbb{0}};\&
    \node [normal] {\symbb{2}};\&
    \node [normal] {\symbb{1}};\\
};
\node[draw,rectangle,dashed,help lines,fit=(config), inner sep=0.5ex] {};
\end{tikzpicture}

    &

\begin{tikzpicture}
[baseline=-\the\dimexpr\fontdimen22\textfont2\relax,
    fonce/.style={fill=blue!20,draw=black}, 
    normal/.style={draw=gray!20}, 
    vide/.style={draw=none}, 
    ampersand replacement=\&]
  \matrix[nodes={
       minimum size=1.2ex,text width=1.2ex,
       text height=1.2ex,inner sep=3pt,draw={gray!20},align=center,
       anchor=base,
   }, row sep=1pt,column sep=1pt,
  ] (config) {
    \& \& \& \& \node [fonce] {\symbb{0}};\\
    \& \& \& \& \node [normal] {\symbb{2}};\\
    \& \& \& \& \node [normal] {\symbb{1}};\\
    \& \& \& \& \node [normal] {\symbb{0}};\\
    \& \& \& \& \node [normal] {\symbb{2}};\\
    \node [fonce] {\symbb{2}};\&
    \node [normal] {\symbb{1}};\&
    \node [normal] {\symbb{0}};\&
    \node [normal] {\symbb{2}};\&
    \node [normal] {\symbb{1}};\\
};
\node[draw,rectangle,dashed,help lines,fit=(config), inner sep=0.5ex] {};
\end{tikzpicture}

    &

\begin{tikzpicture}
[baseline=-\the\dimexpr\fontdimen22\textfont2\relax,
    fonce/.style={fill=blue!20,draw=black}, 
    normal/.style={draw=gray!20}, 
    vide/.style={draw=none}, 
    ampersand replacement=\&]
  \matrix[nodes={
       minimum size=1.2ex,text width=1.2ex,
       text height=1.2ex,inner sep=3pt,draw={gray!20},align=center,
       anchor=base,
   }, row sep=1pt,column sep=1pt,
  ] (config) {
    \& \& \& \& \node [fonce] {\symbb{1}};\\
    \& \& \& \& \node [normal] {\symbb{2}};\\
    \& \& \& \& \node [normal] {\symbb{1}};\\
    \& \& \& \& \node [normal] {\symbb{0}};\\
    \& \& \& \& \node [normal] {\symbb{2}};\\
    \node [fonce] {\symbb{2}};\&
    \node [normal] {\symbb{1}};\&
    \node [normal] {\symbb{0}};\&
    \node [normal] {\symbb{2}};\&
    \node [normal] {\symbb{1}};\\
};
\node[draw,rectangle,dashed,help lines,fit=(config), inner sep=0.5ex] {};
\end{tikzpicture}

\end{tabular}
\end{center}
    \caption{On the left, an L-shape pattern of support
    $\{(1,0),(2,0),(3,0),(4,0),(4,1),(4,2),(4,3),(4,4)\}$ is shown.
    It is bispecial at positions $a=(0,0)$ and $b=(4,5)$ because
    it can be extended in more than one way at these positions
    within the language of the configurations $x$ and $y$ 
    shown in \Cref{fig:intro-sturmian-config-pair}.
    Thus $b-a=(4,5)\in V_\alpha$
    when $\balpha=(\sqrt{2}/2,\sqrt{19}-4)$.}
    \label{fig:bispecial-pattern}
\end{figure}
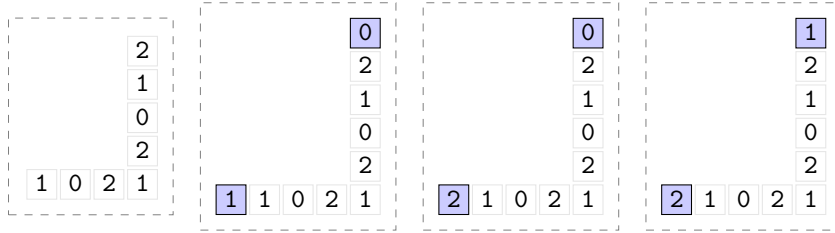

\begin{mainquestion}
    Let $d\geq1$ and $\alpha \in [0,1)^d$ be a totally irrational vector.
    What is the relation between the set
    \[
        V_\alpha=\{b-a\colon \text{ there exists } w\in\Lcal_S(c_\alpha)
                    \text{ which is bispecial at positions } a,b\in\ZZd\}
    \]
    and simultaneous Diophantine approximations 
    of the vector $\alpha$?
\end{mainquestion}

	\textbf{Structure of the article}.
    Preliminary properties of indistinguishable asymptotic pairs are presented
    in \Cref{sec:preliminaries}.
    In \Cref{sec:proof_of_thm_B}, we study the pattern complexity
    of multidimensional indistinguishable asymptotic pairs satisfying the flip condition
    and we prove \Cref{maintheorem:ddim-complexity-is-FminusS}.
    In \Cref{sec:multidim_sturmian}, we define characteristic Sturmian configurations in $\ZZd$
    from codimension-one cut and project schemes.
    We prove that they are indistinguishable asymptotic pairs satisfying the flip
    condition.
    In \Cref{sec:proof_of_thm_A}, we complete the proof of
    \Cref{thm:multidim_sturmian_characterization}, more precisely that
    uniformly recurrent indistinguishable asymptotic pairs satisfying the flip
    condition are multidimensional Sturmian configurations.
    In the appendix (\Cref{sec:appendix}), we provide an analogous notion of
    indistinguishable pairs for pairs of asymptotic configurations on a countable group and provide proofs of their basic properties for further reference.

	\textbf{Acknowledgments}. 
    We are very grateful to an anonymous reviewer who suggested several improvements. The two authors were supported by the Agence Nationale de la Recherche through the
    projects CODYS (ANR-18-CE40-0007), CoCoGro (ANR-16-CE40-0005) and
    IZES (ANR-22-CE40-0011). S. Barbieri was also supported by the FONDECYT grants 11200037 and 1240085. 
    This work originated from a previous work with \v{S}. Starosta 
    and a visit of the two authors in
    October 2019 supported by PHC Barrande, a France-Czech Republic bilateral
    funding and grant no. 7AMB18FR048 of MEYS of Czech Republic.

	\section{Preliminaries}\label{sec:preliminaries}
	
	We denote by $\NN$ the set of non-negative integers. Intervals consisting of integers are written using the notation $\llbracket n,m \rrbracket = \{ k \in \ZZ : n \leq k \leq m  \}$, for $n,m \in \ZZ$ with $n \leq m$. A finite subset $S\subset\ZZd$ is \define{connected} if the subgraph induced by the vertices $S$ within the graph with vertices $V = \ZZd$ and edges
		$E=\{(u,u+e_i)\colon u\in\ZZd,1\leq i\leq d\}$ is connected, where $e_i$ is the canonical vector with $1$ on position $i$ and $0$ elsewhere.

	Let $\Sigma$ be a finite set which we call \define{alphabet} and $d$ a positive integer. An element $x \in \Sigma^{\ZZd} = \{x \colon \ZZd \to \Sigma  \}$ is called a \define{configuration}. For $u \in \ZZd$ we denote the value $x(u)$ by $x_u$. We endow set of all configurations $\Sigma^{\ZZd}$ with the prodiscrete topology.
	The \define{shift action} $\ZZd \overset{\sigma}{\curvearrowright} \Sigma^{\ZZd}$ is given by the map $\sigma\colon \ZZd \times \Sigma^{\ZZd}\to \Sigma^{\ZZd}$ where
	\[ \sigma^u(x)_v \isdef \sigma(u,x)_v =  x_{u+v} \quad \mbox{ for every } u,v \in \ZZd, x \in \Sigma^{\ZZd}.  \]

	The orbit of $x \in \Sigma^{\ZZd}$ is the set $\Orb(x) = \{\sigma^{v}(x) : v \in \ZZd \}$. For a finite subset $S \subset \ZZd$, a function $p \colon S \to \Sigma$ 
    is called a \define{pattern} and the set $S$ is its \define{support}. 
    We denote it as $p \in \Sigma^S$.
    Given a pattern $p \in \Sigma^S$, the \define{cylinder} centered at $p$ is $[p] = \{ x \in \Sigma^{\ZZd} \colon x|_S = p \}$. 
	For finite subset $S \subset \ZZd$, the \define{language with support $S$} of a configuration $x$ is the set of patterns \[\Lcal_{S}(x) = \{ p \in \Sigma^S : \mbox{ there is } u \in \ZZd \mbox{ such that } \sigma^u(x) \in [p]\}. \]
	The \define{language of $x$} is the union $\Lcal(x)$ of the sets $\Lcal_S(x)$ for every finite $S \subset \ZZd$. 
    We say a pattern $p$ \define{appears} in $x \in \Sigma^{\ZZd}$ if there exists $u \in \ZZd$ such that $\sigma^u(x) \in [p]$. Let us also denote by $\occ_p(x) = \{u \in \ZZd \colon \sigma^u(x) \in [p]\}$ the set of \define{occurrences} of a pattern $p$ in the configuration $x \in \Sigma^{\ZZd}$.
	
    \begin{definition}%
		We say that two configurations $x,y$ are \define{asymptotic}, or that $(x,y)$ is an asymptotic pair, if the set $F = \{ u \in \ZZd \colon x_u \neq y_u\}$ is finite. The set $F$ is called the \define{difference set} of $(x,y)$. If $x=y$ we say that the asymptotic pair is \define{trivial}.
	\end{definition}

    Observe that when $x,y \in \Sigma^{\ZZd}$ are asymptotic sequences,
    the difference $\occ_p(x)\setminus \occ_p(y)$ is finite
    because the occurrences of $p$ are the same far from the difference set.
    More precisely, 
    let $F$ denote the difference set of an asymptotic pair $x,y$.
    Let $S$ denote the support of a pattern $p$, 
    then for every $u \in \ZZd \setminus (F-S)$ and every
    $s \in S$, we have $s+u \notin F$ and thus \[\sigma^u(x)_s = x_{u+s} =
    y_{u+s} = \sigma^u(y)_s,\] which implies 
    in turn that
    $u\in \occ_p(x)$ if and only if $u\in\occ_p(y)$
    for all $u\in\ZZd\setminus(F-S)$.
    Therefore,
    \[
        \occ_p(x)\setminus \occ_p(y) \subseteq F-S
        = \{ g-s \colon g \in F, s \in S \}.
    \] 
    In particular, the set $\occ_p(x)\setminus \occ_p(y)$ is finite. Moreover, since $F$
    is the difference set of $x$ and $y$, we have
    \[
        \occ_p(x)\setminus \occ_p(y)
        = 
        \occ_p(x)\cap (F-S).
    \] 

    \begin{definition}\label{def:indistinguishable}
        We say that two asymptotic configurations $x,y \in \Sigma^{\ZZd}$ are
        \define{indistinguishable} if 
        for every pattern $p$ of finite support, we have
    \[
        \#\left(\occ_p(x)\setminus \occ_p(y)\right) = 
        \#\left(\occ_p(y)\setminus \occ_p(x)\right).
    \] 
	\end{definition}

    Notice that \Cref{def:indistinguishable} holds only for asymptotic pairs. A
    more general notion, known as \define{local indistinguishability} exists in
    the context of tilings of $\RR^d$, see \cite[\S
    5.1.1]{baake_aperiodic_2013}.  In terms of subshifts, two configurations
    $x,y\in\Sigma^\ZZd$ are \define{locally indistinguishable}, or LI for short,
    if they have the same language, i.e., $\Lcal(x)=\Lcal(y)$. In this work, we
    always write ``indistinguishable asymptotic pair'' to emphasize the context
    in which \Cref{def:indistinguishable} holds.

	An example of an indistinguishable asymptotic pair over $\ZZ^2$ 
    is shown on~\Cref{fig_same_orbit_example}, see also \Cref{fig:intro-sturmian-config-pair}
    in the introduction.
	\begin{figure}[ht!]
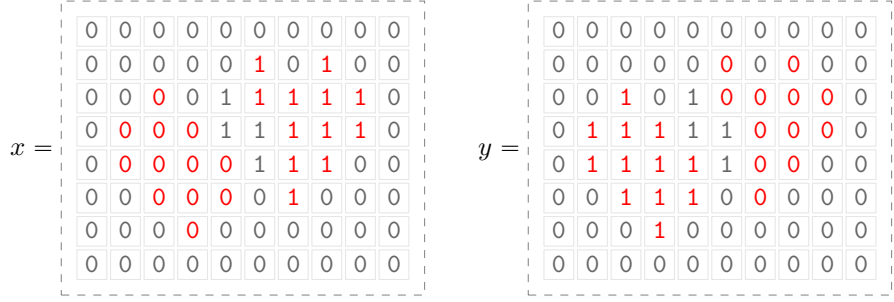

		\begin{align*}
		x = \xConfig{
			{\color{black!60}\symb{0}}\& {\color{black!60}\symb{0}}\& {\color{black!60}\symb{0}}\& {\color{black!60}\symb{0}}\& {\color{black!60}\symb{0}}\& {\color{black!60}\symb{0}}\& {\color{black!60}\symb{0}}\& {\color{black!60}\symb{0}}\& {\color{black!60}\symb{0}}\& {\color{black!60}\symb{0}}\\
			{\color{black!60}\symb{0}}\& {\color{black!60}\symb{0}}\& {\color{black!60}\symb{0}}\& {\color{black!60}\symb{0}}\& {\color{black!60}\symb{0}}\& {\color{red}\symb{1}} \& {\color{black!60}\symb{0}}\& {\color{red}\symb{1}} \& {\color{black!60}\symb{0}}\& {\color{black!60}\symb{0}}\\
			{\color{black!60}\symb{0}}\& {\color{black!60}\symb{0}}\& {\color{red}\symb{0}} \& {\color{black!60}\symb{0}}\& {\color{black!60}\symb{1}} \& {\color{red}\symb{1}} \& {\color{red}\symb{1}} \& {\color{red}\symb{1}} \& {\color{red}\symb{1}} \& {\color{black!60}\symb{0}}\\
			{\color{black!60}\symb{0}}\& {\color{red}\symb{0}} \& {\color{red}\symb{0}} \& {\color{red}\symb{0}} \& {\color{black!60}\symb{1}} \& {\color{black!60}\symb{1}} \& {\color{red}\symb{1}} \& {\color{red}\symb{1}} \& {\color{red}\symb{1}} \& {\color{black!60}\symb{0}}\\
			{\color{black!60}\symb{0}}\& {\color{red}\symb{0}} \& {\color{red}\symb{0}} \& {\color{red}\symb{0}} \& {\color{red}\symb{0}} \& {\color{black!60}\symb{1}} \& {\color{red}\symb{1}} \& {\color{red}\symb{1}} \& {\color{black!60}\symb{0}}\& {\color{black!60}\symb{0}}\\
			{\color{black!60}\symb{0}}\& {\color{black!60}\symb{0}}\& {\color{red}\symb{0}} \& {\color{red}\symb{0}} \& {\color{red}\symb{0}} \& {\color{black!60}\symb{0}}\& {\color{red}\symb{1}} \& {\color{black!60}\symb{0}}\& {\color{black!60}\symb{0}}\& {\color{black!60}\symb{0}}\\
			{\color{black!60}\symb{0}}\& {\color{black!60}\symb{0}}\& {\color{black!60}\symb{0}}\& {\color{red}\symb{0}} \& {\color{black!60}\symb{0}}\& {\color{black!60}\symb{0}}\& {\color{black!60}\symb{0}}\& {\color{black!60}\symb{0}}\& {\color{black!60}\symb{0}}\& {\color{black!60}\symb{0}}\\
			{\color{black!60}\symb{0}}\& {\color{black!60}\symb{0}}\& {\color{black!60}\symb{0}}\& {\color{black!60}\symb{0}}\& {\color{black!60}\symb{0}}\& {\color{black!60}\symb{0}}\& {\color{black!60}\symb{0}}\& {\color{black!60}\symb{0}}\& {\color{black!60}\symb{0}}\& {\color{black!60}\symb{0}}\\
		}
		\qquad \qquad
		y = \xConfig{
			{\color{black!60}\symb{0}}\& {\color{black!60}\symb{0}}\& {\color{black!60}\symb{0}}\& {\color{black!60}\symb{0}}\& {\color{black!60}\symb{0}}\& {\color{black!60}\symb{0}}\& {\color{black!60}\symb{0}}\& {\color{black!60}\symb{0}}\& {\color{black!60}\symb{0}}\& {\color{black!60}\symb{0}}\\
			{\color{black!60}\symb{0}}\& {\color{black!60}\symb{0}}\& {\color{black!60}\symb{0}}\& {\color{black!60}\symb{0}}\& {\color{black!60}\symb{0}}\& {\color{red}\symb{0}} \& {\color{black!60}\symb{0}}\& {\color{red}\symb{0}} \& {\color{black!60}\symb{0}}\& {\color{black!60}\symb{0}}\\
			{\color{black!60}\symb{0}}\& {\color{black!60}\symb{0}}\& {\color{red}\symb{1}} \& {\color{black!60}\symb{0}}\& {\color{black!60}\symb{1}} \& {\color{red}\symb{0}} \& {\color{red}\symb{0}} \& {\color{red}\symb{0}} \& {\color{red}\symb{0}} \& {\color{black!60}\symb{0}}\\
			{\color{black!60}\symb{0}}\& {\color{red}\symb{1}} \& {\color{red}\symb{1}} \& {\color{red}\symb{1}} \& {\color{black!60}\symb{1}} \& {\color{black!60}\symb{1}} \& {\color{red}\symb{0}} \& {\color{red}\symb{0}} \& {\color{red}\symb{0}} \& {\color{black!60}\symb{0}}\\
			{\color{black!60}\symb{0}}\& {\color{red}\symb{1}} \& {\color{red}\symb{1}} \& {\color{red}\symb{1}} \& {\color{red}\symb{1}} \& {\color{black!60}\symb{1}} \& {\color{red}\symb{0}} \& {\color{red}\symb{0}} \& {\color{black!60}\symb{0}}\& {\color{black!60}\symb{0}}\\
			{\color{black!60}\symb{0}}\& {\color{black!60}\symb{0}}\& {\color{red}\symb{1}} \& {\color{red}\symb{1}} \& {\color{red}\symb{1}} \& {\color{black!60}\symb{0}}\& {\color{red}\symb{0}} \& {\color{black!60}\symb{0}}\& {\color{black!60}\symb{0}}\& {\color{black!60}\symb{0}}\\
			{\color{black!60}\symb{0}}\& {\color{black!60}\symb{0}}\& {\color{black!60}\symb{0}}\& {\color{red}\symb{1}} \& {\color{black!60}\symb{0}}\& {\color{black!60}\symb{0}}\& {\color{black!60}\symb{0}}\& {\color{black!60}\symb{0}}\& {\color{black!60}\symb{0}}\& {\color{black!60}\symb{0}}\\
			{\color{black!60}\symb{0}}\& {\color{black!60}\symb{0}}\& {\color{black!60}\symb{0}}\& {\color{black!60}\symb{0}}\& {\color{black!60}\symb{0}}\& {\color{black!60}\symb{0}}\& {\color{black!60}\symb{0}}\& {\color{black!60}\symb{0}}\& {\color{black!60}\symb{0}}\& {\color{black!60}\symb{0}}\\
		} \;
		\end{align*}
		\caption{A non-trivial indistinguishable asymptotic pair for $\Sigma = \{\symb{0},\symb{1}\}$ and $d=2$ where $y = \sigma^{(3,1)}(x)$. The difference set is highlighted in red and the portions of the configurations which are not shown consist only of the symbol $\symb{0}$.}
		\label{fig_same_orbit_example}
	\end{figure}

	Next we state equivalent conditions for indistinguishability which we will use interchangeably in the proofs that follow. We use the symbol $\indicator{A}$ to indicate the characteristic function of a set $A$.

	\begin{remark}
        The following conditions are equivalent:
        \begin{enumerate}
        \item $x$ and $y$ are indistinguishable asymptotic configurations with difference set $F$,
        \item for every pattern $p$ with finite support $S \subset \ZZd$, we have 
            \[
                \#\left(\occ_p(x)\cap (F-S)\right) = 
                \#\left(\occ_p(y)\cap (F-S)\right),
            \] 
        \item for every pattern $p$ with finite support $S \subset \ZZd$, we have 
            \[ 
            \Delta_p (x,y) 
            \isdef \sum_{u \in F-S } \indicator{[p]}(\sigma^u(y))-\indicator{[p]}(\sigma^u(x))
            = 0.
            \]
        \end{enumerate}
	\end{remark}
	
	\subsection{Properties of indistinguishable asymptotic pairs}

		\begin{proposition}\label{prop:trivialite}
			Let $S_1 \subset S_2$ be finite subsets of $\ZZd$, and let $p \in \Sigma^{S_1}$. 
			We have \[ \Delta_p(x,y) = \sum_{q \in \Sigma^{S_2}, [q] \subset [p] } \Delta_q(x,y).\]
		\end{proposition}
		
		\begin{proof}
			Notice that $[p]$ is the disjoint union of all $[q]$ where $q \in \Sigma^{S_2}$ and $[q] \subset [p]$.  It follows that for any $z \in \Sigma^{\ZZd}$ we have $ \indicator{[p]}(z) =  1$ if and only if there is a unique $q \in \Sigma^{S_2}$ such that $[q] \subset [p]$ and $\indicator{[q]}(z)=1$. Letting $F$ be the difference set of $x,y$ we obtain,
			
			\begin{align*}
			\Delta_p(x,y) & = \sum_{u \in F-S_1}\indicator{[p]}(\sigma^u(y))-\indicator{[p]}(\sigma^u(x)) \\
			& = \sum_{v \in F-S_2}\indicator{[p]}(\sigma^v(y))-\indicator{[p]}(\sigma^v(x)) \\
			& = \sum_{v \in F-S_2}  \sum_{\substack{q \in \Sigma^{S_2} \\ [q] \subset [p]} }\indicator{[q]}(\sigma^v(y))-\indicator{[q]}(\sigma^v(x)).
			\end{align*}
			Exchanging the order of the sums yields the result.\end{proof}
		
		In particular, to prove that two asymptotic configurations are indistinguishable, it suffices to verify the condition $\Delta_p(x,y)=0$ on patterns $p$ whose supports form a collection of finite subsets of $\ZZd$ with the property that every finite subset of $\ZZd$ is contained in some set in the collection. In particular, we may consider the collection of all rectangles (products of bounded integer intervals) or the collection of all connected finite subsets of $\ZZd$.
		
		The affine group $\operatorname{Aff}(\ZZd)$ of $\ZZd$ is the group of all invertible affine transformations from $\ZZd$ into itself. We can represent it as the semidirect product $\operatorname{Aff}(\ZZd) = \ZZd \rtimes \operatorname{GL}_d(\ZZ)$, where $\operatorname{GL}_d(\ZZ)$ is the group of all invertible $d \times d$ matrices with integer entries, which represents the automorphisms of $\ZZd$, and the factor $\ZZd$ on the left represents translations.

	\begin{proposition}\label{prop:shifted_SI}
		Let $(x,y)$ be an indistinguishable asymptotic pair, then
		\begin{enumerate}
			\item $(\sigma^u(x),\sigma^u(y))$ is an indistinguishable asymptotic pair for every $u \in \ZZd$.
			\item $(x\circ A, y \circ A)$ is an indistinguishable asymptotic pair for every $A \in \operatorname{GL}_d(\ZZ)$.
		\end{enumerate}
	In particular, the set of indistinguishable asymptotic pairs is invariant under the action of $\operatorname{Aff}(\ZZd)$.
	\end{proposition}
	
	\begin{proof}
		Let $F$ be the difference set of $(x,y)$. A straightforward computation shows that the difference set of $(\sigma^u(x),\sigma^u(y))$ is $F_1 = F-u$ and the difference set of $(x\circ A, x \circ A)$ is $F_2 = A^{-1}(F)$.
		
		 Let $S \subset \ZZd$ be a finite set and $p \in \Sigma^S$. For the first claim we have
		
		\begin{align*}
\Delta_p(\sigma^u(x),\sigma^u(y)) & = \sum_{v \in F_1-S}\indicator{[p]}(\sigma^{v}(\sigma^u(y)))-\indicator{[p]}(\sigma^v( \sigma^u(y)  ))\\
		& = \sum_{v \in (F-u)-S}\indicator{[p]}(\sigma^{v+u}(y))-\indicator{[p]}(\sigma^{v+u}(y))\\
		& = \sum_{t \in F-S}\indicator{[p]}(\sigma^{t}(y))-\indicator{[p]}(\sigma^{t}(y)) = \Delta_p(x,y) = 0.
		\end{align*}
		Thus $ (\sigma^u(x),\sigma^u(y))$ is an indistinguishable asymptotic pair.
		
		For the second claim, let $q \in \Sigma^{A(S)}$ be the pattern given by $q_{As}=p_s$ for every $s \in S$. We note that for any $v \in \ZZd$, $\sigma^v(x)\in [q]$ if and only if $\sigma^{A^{-1}v}(x \circ A) \in [p]$. This means that $v \in \occ_q(x)$ if and only if $A^{-1}(v) \in \occ_p(x \circ A)$.
		
		As $(x,y)$ is an indistinguishable asymptotic pair, there is a finitely supported permutation $\pi$ of $\ZZd$ so that $\occ_q(x) = \pi(\occ_q(y))$. Then $\pi' = A \circ \pi \circ A^{-1}$ is a finitely supported permutation of $\ZZd$ so that $\occ_p(x\circ A)= \pi'(\occ_p(y\circ A))$. We conclude that $\Delta_p(x\circ A,y\circ A)=0$ and thus they are indistinguishable.\end{proof}

	Let $\Sigma_1,\Sigma_2$ be alphabets. A map $\phi \colon \Sigma_1^{\ZZd} \to \Sigma_2^{\ZZd}$ is a \define{sliding block code} if there exists a finite set $D \subset \ZZd$ and map $\Phi \colon \Sigma_1^D \to \Sigma_2$ called the \define{block code} such that
	
	\[ \phi(x)_{u} = \Phi( \sigma^{u}(x)|_{D})  \mbox{ for every } u \in \ZZd, x \in \Sigma_1^{\ZZd}.\]
	
	Notice that sliding block codes are continuous maps which commute with the shift action, that is, $\sigma^{u}(\phi(x)) = \phi(\sigma^{u}(x))$ for every $u \in \ZZd$ and $x \in \Sigma_1^{\ZZd}$.

	\begin{proposition}\label{prop:invariance_sliding_block_code}
		Let $x,y \in \Sigma_1^{\ZZd}$ be an indistinguishable asymptotic pair and $\phi \colon \Sigma_1^{\ZZd} \to \Sigma_2^{\ZZd}$ a sliding block code. The pair $\phi(x),\phi(y) \in \Sigma_2^{\ZZd}$ is an indistinguishable asymptotic pair.
	\end{proposition}
	
	\begin{proof}
		Let $F$ be the difference set of $x,y$ and $D \subset \ZZd$, $\Phi \colon \Sigma_1^D \to \Sigma_2$ be respectively the set and block code which define $\phi$. If $u \notin F-D$, then $\sigma^{u}(x)|_D = \sigma^{u}(y)|_D$ and thus $\phi(x)_{u} = \phi(y)_{u}$. As $F-D$ is finite, it follows that $\phi(x),\phi(y)$ are asymptotic.
		
		Let $S\subset \ZZd$ be finite and $p\colon S\to\Sigma_2$ be a pattern.
		Let $\phi^{-1}(p) \subset (\Sigma_1)^{D+S}$ be the set of patterns $q$ with support $D+S$ so that for every $s \in S$, $\Phi((q_{d+s})_{d \in D}) = p_s$. It follows that $\phi^{-1}([p]) = \bigcup_{q \in \phi^{-1}(p)}[q]$.
		
		Let $W \subset \ZZd$ be a finite set which is large enough such that $W \supseteq F\cup (D+F)$. We have,
		\begin{align*}
		\#\{u \in W-S \mid \sigma^u(\phi(x))\in[p]\}
		&= \sum_{q \in \phi^{-1}(p)}\#\{u\in W-S \mid \sigma^u(x)\in[q]\}\\
		&= \sum_{q \in \phi^{-1}(p)}\#\{u\in W-S \mid \sigma^u(y)\in[q]\}\\
		&= \#\{u\in W-S \mid \sigma^u(\phi(y))\in[p]\}.
		\end{align*}
		As $W \supseteq F$, we conclude that $\Delta_p(\phi(x), \phi(y))=0$ and therefore $(\phi(x), \phi(y))$
		is an indistinguishable asymptotic pair.
	\end{proof}

\begin{remark}
	The property of being an indistinguishable asymptotic pair is also preserved by $d$-dimensional substitutions and the proof is essentially the same as~\cite[Lemma 5.2]{BarLabSta2021}. We will not make use of this fact anywhere in the article. However, we remark that substitutions might be helpful in order to answer~\Cref{Q:derived-from},
	since they were the tool that provides the characterization of indistinguishable asymptotic pairs 
	for $d=1$, see~\cite[Theorem C]{BarLabSta2021}.
\end{remark}

	Let us recall that a sequence $(x_n)_{n \in \NN}$ of configurations in $\Sigma^{\ZZd}$ converges to $x \in \Sigma^{\ZZd}$ if for every $u \in \ZZd$ we have that $(x_n)_{u} = x_{u}$ for all large enough $n \in \NN$. In what follows we use a notion of convergence for asymptotic pairs which is stronger than the convergence in the prodiscrete topology in order to ensure that limits of asymptotic pairs are themselves asymptotic. This notion comes from interpreting the equivalence relation of asymptotic pairs as an ``\'etale equivalence relation''. For more information on \'etale equivalence relations and their role in the theory of topological orbit equivalence of Cantor minimal systems the reader can refer to~\cite{Put18}. 
	
	\begin{definition}\label{def:etale}
		Let $(x_n,y_n)_{n \in \NN}$ be a sequence of asymptotic pairs. We say that $(x_n,y_n)_{n \in \NN}$ \define{converges in the asymptotic relation} to a pair $(x,y)$ if $(x_n)_{n \in \NN}$ converges to $x$, $(y_n)_{n \in \NN}$ converges to $y$, and there exists a finite set $F \subset \ZZd$ so that $x_n|_{\ZZd \setminus F} = y_n|_{\ZZd \setminus F}$ for all large enough $n \in \NN$.
	\end{definition}

	If $(x_n,y_n)_{n \in \NN}$ converges in the asymptotic relation to a pair $(x,y)$, then the pair $(x,y)$ is necessarily asymptotic. We call $(x,y)$ the \define{\'etale limit} of $(x_n,y_n)_{n \in \NN}$. In the next proposition, we see that indistinguishability is preserved by \'etale limits.
	 
    \begin{proposition}\label{prop:limit-of-indist-is-indist}
		Let $(x_n,y_n)_{n \in \NN}$ be a sequence of asymptotic pairs in $\Sigma^{\ZZd}$ which converges in the asymptotic relation to $(x,y)$. If for every $n \in \NN$ we have that $(x_n,y_n)$ is indistinguishable, then $(x,y)$ is indistinguishable.
	\end{proposition}
	
	\begin{proof}
		Let $p \in \Sigma^S$ be a pattern. As $(x_n,y_n)_{n \in \NN}$ converges in the asymptotic relation to $(x,y)$, there exists a finite set $F \subset \ZZd$ and $N_1 \in \NN$ so that $x_n|_{\ZZd \setminus F} = y_n|_{\ZZd \setminus F}$ for every $n \geq N_1$. In particular we have that the difference sets of $(x,y)$ and $(x_n,y_n)$ for $n \geq N_1$ are contained in $F$. It suffices thus to show that \[ \# \{ \occ_p(x) \cap (F-S)\} = \#\{ \occ_p(y) \cap (F-S)\}.   \]
		As $(x_n)_{n\in \NN}$ converges to $x$ and $(y_n)_{n \in \NN}$ converges to $y$, there exists $N_2 \in \NN$ so that $x_n|_{F-S+S} = x|_{F-S+S}$ and $y_n|_{F-S+S} = y|_{F-S+S}$ for all $n \geq N_2$. Thus for $n \geq N_2$ and every $v\in F-S$ we have that $\sigma^{v}(x)|_S = \sigma^{v}(x_n)|_S$ and $\sigma^{v}(y)|_S = \sigma^{v}(y_n)|_S$. From this we obtain that $\occ_p(x) \cap (F-S) = \occ_p(x_n) \cap (F-S)$ and $\occ_p(y) \cap (F-S) = \occ_p(y_n) \cap (F-S)$ for every $n \geq N_2$.
		
		Let $N = \max\{N_1,N_2\}$ and let $n \geq N$. As $n \geq N_1$, we have that $(x_n,y_n)$ is an indistinguishable asymptotic pair whose difference set is contained in $F$, it follows that $\# \{ \occ_p(x_n) \cap (F-S)\} = \#\{ \occ_p(y_n) \cap (F-S)\}.$ As $n \geq N_2$, we obtain $\# \{ \occ_p(x) \cap (F-S)\} = \#\{ \occ_p(y) \cap (F-S)\}.$ As this argument holds for every pattern $p$, we conclude that $(x,y)$ is indistinguishable.
	\end{proof}
	
	\begin{definition}
		Let $x \in \Sigma^{\ZZd}$ be a configuration. 
		\begin{enumerate}
			\item $x$ is \define{recurrent} if for every $p \in \Lcal(x)$ the set $\occ_p(x)$ is infinite.   
			\item $x$ is \define{uniformly recurrent} if every $p \in \Lcal(x)$ appears with bounded gaps, that is, for every $p \in \Lcal(x)$ there exists a finite $K \subset \ZZd$ such that for every $u \in \ZZd$ there is $k \in K$ such that $\sigma^{u+k}(x) \in [p]$.
		\end{enumerate}
	\end{definition}

	Clearly both recurrence and uniform recurrence are properties that are satisfied either by both configurations in an indistinguishable asymptotic pair simultaneously, or by none of them. Furthermore, it can be easily verified that both of these properties are preserved under the action of $\operatorname{Aff}(\ZZd)$, just as in~\Cref{prop:shifted_SI}.
	
	\begin{proposition}\label{prop:recurrence}
		Let $x,y \in \Sigma^{\ZZd}$ be an indistinguishable asymptotic pair. If $x$ is not recurrent, then $x,y$ lie in the same orbit.
	\end{proposition}
	
	\begin{proof}
		If $x$ is not recurrent, there is a finite $S \subset \ZZd$ and $p \in \Lcal_S(x)$ such that $\occ_p(x)$ is finite. As $\Delta_p(x,y)=0$, it follows that $\occ_p(y)$ is also finite.
		
		Let $(S_n)_{n \in \NN}$ be an increasing sequence of finite subsets of $\ZZd$ such that $S_0 = S$ and $\bigcup_{n \in \NN}S_n = \ZZd$ and let $q_n = x|_{S_n}$. As $x \in [q_n]$ and $\Delta_{q_n}(x,y)=0$, there exists $u_n \in \ZZd$ so that $\sigma^{u_n}(y) \in [q_n]$. Furthermore, as $q_n|_S = p$, it follows that $\sigma^{u_n}(y) \in [p]$ and thus $u_n \in \occ_p(y)$. As $\occ_p(y)$ is finite, there exists $v \in \occ_p(y)$ and a subsequence such that $u_{n(k)} = v$ and thus $\sigma^{v}(y) \in [q_{n(k)}]$ for every $k \in \NN$. As  $\bigcap_{n \in \NN}[q_n] = \bigcap_{k \in \NN}[q_{n(k)}] = \{x\}$ we deduce that $\sigma^{v}(y) = x$.
	\end{proof}

\begin{remark}
	All of the definitions and results stated so far in this section are valid in the more general context where $\ZZd$ is replaced by a countable group $\Gamma$. In~\Cref{sec:appendix} we provide definitions and proofs in this more general setting with the hope that it might be useful for further research.
\end{remark}
	
\subsection{Known results on dimension $1$}\label{subsec:dim1}

When considering $d=1$, two phenomena, stated in the lemmas below, simplify the study of indistinguishable asymptotic pairs: every word in the language can be read from the difference set, and recurrent configurations which are part of an indistinguishable asymptotic pair are in fact uniformly recurrent.

	\begin{lemma}\label{lem:words_appear_in_diff_set} [Lemma 2.8 of~\cite{BarLabSta2021}]
		Let $x,y \in \Sigma^{\ZZ}$ be a non-trivial indistinguishable asymptotic pair with difference set $F$. For every finite $S \subset \ZZ$ and $w \in \Lcal_S(x)$ there is $u \in F-S$ such that $\sigma^u(x) \in [w]$.
	\end{lemma}

	\begin{lemma}\label{lem:rec_implies_urec} [Lemma 2.12 of~\cite{BarLabSta2021}]
		Let $x,y\in \Sigma^{\ZZ}$ be a non-trivial indistinguishable asymptotic pair. If $x$ is recurrent, then $x$ is uniformly recurrent.
	\end{lemma}

	Gathering~\Cref{prop:recurrence} and~\Cref{lem:rec_implies_urec} we obtain the following beautiful dichotomy.
	
	\begin{corollary} [Corollary 2.13 of~\cite{BarLabSta2021}] 
		Let $x,y\in \Sigma^{\ZZ}$ be a non-trivial asymptotic indistinguishable pair. Then exactly one of the following statements holds 
        \begin{enumerate}
			\item $x = \sigma^n(y)$ for some nonzero $n \in \ZZ$,
			\item $x$ and $y$ are uniformly recurrent.
		\end{enumerate}
	\end{corollary}

	This dichotomy was the starting point which lead to our characterization of Sturmian configurations through indistinguishable asymptotic pairs in $\ZZ$.
	
	\begin{theorem}\label{thm:sturmian_characterization}
		[Theorem A of~\cite{BarLabSta2021}]
		Let $x,y\in\{\symb{0},\symb{1}\}^\ZZ$ and assume that $x$ is recurrent.
		The pair $(x,y)$ is an indistinguishable asymptotic pair
		with difference set $F=\{-1,0\}$ such that $x_{-1}x_{0}=\symb{10}$
		and $y_{-1}y_{0}=\symb{01}$ if and only if
		there exists $\alpha\in[0,1]\setminus\QQ$ such that
		$x={c}_{\alpha}$ and $y={c}'_{\alpha}$ are the lower and upper characteristic
		Sturmian sequences of slope $\alpha$.
	\end{theorem}

	When $d\geq 2$ there exist non-trivial indistinguishable asymptotic pairs where both of the above lemmas fail.
	
	\begin{example}\label{ex:rec_not_urec}
		Let $u,v \in \{\symb{0},\symb{1}\}^{\ZZ}$ be any indistinguishable asymptotic pair. Consider the configurations $x,y \in \{\symb{0},\symb{1},\symb{2}\}^{\ZZ^2}$ given by \[
		x(i,j) = \begin{cases}
			u(i) & \mbox{ if } j = 0\\
			\symb{2} & \mbox{ if } j \neq 0\\
		\end{cases} \ \mbox{ and } \ 
		y(i,j) = \begin{cases}
			v(i) & \mbox{ if } j = 0\\
			\symb{2} & \mbox{ if } j \neq 0\\
		\end{cases} \mbox{ for every } i,j \in \ZZ.	\]
		The words $x,y$ form an indistinguishable asymptotic pair which does not satisfy~\Cref{lem:words_appear_in_diff_set} (the symbol $2$ does not occur in the difference set) nor~\Cref{lem:rec_implies_urec} (it is recurrent but not uniformly recurrent). See~\Cref{fig_configuraciones_rec_no_urec}.
	\end{example}
	
	\begin{figure}[ht!]
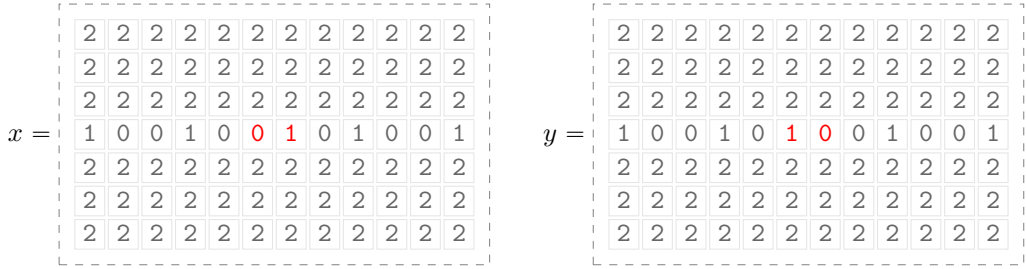

		\begin{align*}
			x = \xConfig{
				{\color{black!60}\symb{2}} \& {\color{black!60}\symb{2}} \& {\color{black!60}\symb{2}} \& {\color{black!60}\symb{2}} \& {\color{black!60}\symb{2}} \& {\color{black!60}\symb{2}} \& {\color{black!60}\symb{2}} \& {\color{black!60}\symb{2}} \& {\color{black!60}\symb{2}} \& {\color{black!60}\symb{2}} \& {\color{black!60}\symb{2}} \& {\color{black!60}\symb{2}} \\
				{\color{black!60}\symb{2}} \& {\color{black!60}\symb{2}} \& {\color{black!60}\symb{2}} \& {\color{black!60}\symb{2}} \& {\color{black!60}\symb{2}} \& {\color{black!60}\symb{2}} \& {\color{black!60}\symb{2}} \& {\color{black!60}\symb{2}} \& {\color{black!60}\symb{2}} \& {\color{black!60}\symb{2}} \& {\color{black!60}\symb{2}} \& {\color{black!60}\symb{2}} \\
				{\color{black!60}\symb{2}} \& {\color{black!60}\symb{2}} \& {\color{black!60}\symb{2}} \& {\color{black!60}\symb{2}} \& {\color{black!60}\symb{2}} \& {\color{black!60}\symb{2}} \& {\color{black!60}\symb{2}} \& {\color{black!60}\symb{2}} \& {\color{black!60}\symb{2}} \& {\color{black!60}\symb{2}} \& {\color{black!60}\symb{2}} \& {\color{black!60}\symb{2}} \\
				{\color{black!60}\symb{1}}\& {\color{black!60}\symb{0}}\& {\color{black!60}\symb{0}}\& {\color{black!60}\symb{1}}\& {\color{black!60}\symb{0}}\& \color{red}{\symb{0}} \& 
				\color{red}{\symb{1}} \& {\color{black!60}\symb{0}}\& {\color{black!60}\symb{1}}\& {\color{black!60}\symb{0}}\& {\color{black!60}\symb{0}}\& {\color{black!60}\symb{1}}\\
				{\color{black!60}\symb{2}} \& {\color{black!60}\symb{2}} \& {\color{black!60}\symb{2}} \& {\color{black!60}\symb{2}} \& {\color{black!60}\symb{2}} \& {\color{black!60}\symb{2}} \& {\color{black!60}\symb{2}} \& {\color{black!60}\symb{2}} \& {\color{black!60}\symb{2}} \& {\color{black!60}\symb{2}} \& {\color{black!60}\symb{2}} \& {\color{black!60}\symb{2}} \\
				{\color{black!60}\symb{2}} \& {\color{black!60}\symb{2}} \& {\color{black!60}\symb{2}} \& {\color{black!60}\symb{2}} \& {\color{black!60}\symb{2}} \& {\color{black!60}\symb{2}} \& {\color{black!60}\symb{2}} \& {\color{black!60}\symb{2}} \& {\color{black!60}\symb{2}} \& {\color{black!60}\symb{2}} \& {\color{black!60}\symb{2}} \& {\color{black!60}\symb{2}} \\
				{\color{black!60}\symb{2}} \& {\color{black!60}\symb{2}} \& {\color{black!60}\symb{2}} \& {\color{black!60}\symb{2}} \& {\color{black!60}\symb{2}} \& {\color{black!60}\symb{2}} \& {\color{black!60}\symb{2}} \& {\color{black!60}\symb{2}} \& {\color{black!60}\symb{2}} \& {\color{black!60}\symb{2}} \& {\color{black!60}\symb{2}} \& {\color{black!60}\symb{2}} \\
			}
			\qquad \qquad
			y = \xConfig{
				{\color{black!60}\symb{2}} \& {\color{black!60}\symb{2}} \& {\color{black!60}\symb{2}} \& {\color{black!60}\symb{2}} \& {\color{black!60}\symb{2}} \& {\color{black!60}\symb{2}} \& {\color{black!60}\symb{2}} \& {\color{black!60}\symb{2}} \& {\color{black!60}\symb{2}} \& {\color{black!60}\symb{2}} \& {\color{black!60}\symb{2}} \& {\color{black!60}\symb{2}} \\
				{\color{black!60}\symb{2}} \& {\color{black!60}\symb{2}} \& {\color{black!60}\symb{2}} \& {\color{black!60}\symb{2}} \& {\color{black!60}\symb{2}} \& {\color{black!60}\symb{2}} \& {\color{black!60}\symb{2}} \& {\color{black!60}\symb{2}} \& {\color{black!60}\symb{2}} \& {\color{black!60}\symb{2}} \& {\color{black!60}\symb{2}} \& {\color{black!60}\symb{2}} \\
				{\color{black!60}\symb{2}} \& {\color{black!60}\symb{2}} \& {\color{black!60}\symb{2}} \& {\color{black!60}\symb{2}} \& {\color{black!60}\symb{2}} \& {\color{black!60}\symb{2}} \& {\color{black!60}\symb{2}} \& {\color{black!60}\symb{2}} \& {\color{black!60}\symb{2}} \& {\color{black!60}\symb{2}} \& {\color{black!60}\symb{2}} \& {\color{black!60}\symb{2}} \\
				{\color{black!60}\symb{1}}\& {\color{black!60}\symb{0}}\& {\color{black!60}\symb{0}}\& {\color{black!60}\symb{1}}\& {\color{black!60}\symb{0}}\& \color{red}{\symb{1}} \& 
				\color{red}{\symb{0}} \&{\color{black!60}\symb{0}}\& {\color{black!60}\symb{1}}\& {\color{black!60}\symb{0}}\& {\color{black!60}\symb{0}}\& {\color{black!60}\symb{1}}\\
				{\color{black!60}\symb{2}} \& {\color{black!60}\symb{2}} \& {\color{black!60}\symb{2}} \& {\color{black!60}\symb{2}} \& {\color{black!60}\symb{2}} \& {\color{black!60}\symb{2}} \& {\color{black!60}\symb{2}} \& {\color{black!60}\symb{2}} \& {\color{black!60}\symb{2}} \& {\color{black!60}\symb{2}} \& {\color{black!60}\symb{2}} \& {\color{black!60}\symb{2}} \\
				{\color{black!60}\symb{2}} \& {\color{black!60}\symb{2}} \& {\color{black!60}\symb{2}} \& {\color{black!60}\symb{2}} \& {\color{black!60}\symb{2}} \& {\color{black!60}\symb{2}} \& {\color{black!60}\symb{2}} \& {\color{black!60}\symb{2}} \& {\color{black!60}\symb{2}} \& {\color{black!60}\symb{2}} \& {\color{black!60}\symb{2}} \& {\color{black!60}\symb{2}} \\
				{\color{black!60}\symb{2}} \& {\color{black!60}\symb{2}} \& {\color{black!60}\symb{2}} \& {\color{black!60}\symb{2}} \& {\color{black!60}\symb{2}} \& {\color{black!60}\symb{2}} \& {\color{black!60}\symb{2}} \& {\color{black!60}\symb{2}} \& {\color{black!60}\symb{2}} \& {\color{black!60}\symb{2}} \& {\color{black!60}\symb{2}} \& {\color{black!60}\symb{2}} \\
			} \;
		\end{align*}
        \caption{An indistinguishable, recurrent but not uniformly recurrent asymptotic pair $(x,y)$ given by the two characteristic Fibonacci words ($\alpha = \frac{\sqrt{5}-1}{2}$) in the central row.}
		\label{fig_configuraciones_rec_no_urec}
	\end{figure}

	In particular, a convenient consequence of~\Cref{lem:words_appear_in_diff_set} in $d=1$ is that the complexity of any pair of indistinguishable configurations is linear and the bound is given by the size of the difference set. More precisely, if $x,y \in \Sigma^{\ZZ}$ is a non-trivial indistinguishable asymptotic pair whose difference set $F$ is contained in an interval $I$, then for every $n \geq 1$ \[n+1 \leq \#\Lcal_{\llbracket 1,n\rrbracket}(x) \leq n+\#(I)-1.\]
	See~\cite[Proposition 3.4]{BarLabSta2021}. This consequence also fails in the multidimensional setting as shown in the following example.

	\begin{example}
		Fix $\symb{k} \geq 1$. Let $u,v$ as in~\Cref{ex:rec_not_urec} and let $x,y \in \{\symb{0},\dots, \symb{k}-1\}^{\ZZ^2}$ be given by
		\[
		x(i,j) = \begin{cases}
			u(i) & \mbox{ if } j = 0\\
			\symb{j} \mod \symb{k} & \mbox{ if } j \neq 0\\
		\end{cases} \ \mbox{ and } \ 
		y(i,j) = \begin{cases}
			v(i) & \mbox{ if } j = 0\\
			\symb{j} \mod \symb{k} & \mbox{ if } j \neq 0\\
		\end{cases}.	\]
		We obtain an indistinguishable asymptotic pair whose difference set has size $2$, but such that the alphabet can be as large as required by taking larger values of $\symb{k}$. 
	\end{example}
	
	This shows that a naive analogue of the complexity upper-bound given by~\Cref{lem:words_appear_in_diff_set} also fails in the multidimensional setting. However, under special conditions which we introduce in~\Cref{subsec:multidim_indistinguishable}, we show that an analogue of~\Cref{lem:words_appear_in_diff_set} holds, which gives us a way to extend~\Cref{thm:sturmian_characterization} to $\ZZd$.

\subsection{The flip condition}\label{subsec:multidim_indistinguishable}

As shown in the examples of~\Cref{subsec:dim1}, indistinguishable asymptotic pairs in $\ZZd$ in general are not related to Sturmian configurations as strongly as in dimension $1$. Despite these discouraging examples, we show that if we consider indistinguishable asymptotic pairs satisfying an
additional hypothesis, then many of the good properties from dimension $1$ are still valid and we will be able to obtain a characterization of multidimensional Sturmian configurations in terms of indistinguishable asymptotic pairs.

\begin{definition}
	Let $x,y\in\alfad^{\ZZd}$ be an asymptotic pair. We say it satisfies the \define{flip condition} if:
	\begin{enumerate}
		\item the difference set of $x$ and $y$ is $F=\{\bzero, -\be_1,\dots,-\be_d\}$,
        \item the restriction $x|_F$ is a bijection $F \to \alfad$,
        \item the map defined by $x_\bn\mapsto y_\bn$ for every $\bn\in F$
            is a cyclic permutation on the alphabet $\alfad$.
	\end{enumerate}
    Without lost of generality, we assume that 
    $x_{\bzero} = \symb{0}$
    and
    $y_\bn = x_\bn - \symb{1} \bmod (d+\symb{1})$ for every $\bn\in F$.
\end{definition}

The flip condition may be interpreted as a symbolic coding of the act of
geometrically flipping the faces of a hypercube at the origin of a discrete
hyperplane as in~\Cref{fig:intro-sturmian-config-pair-discrete-plane}. 

\section{Multidimensional indistinguishable asymptotic pairs and their complexity}\label{sec:proof_of_thm_B}

The goal of this section is to prove~\Cref{maintheorem:ddim-complexity-is-FminusS}
which characterizes indistinguishable asymptotic pairs which satisfy the flip condition through their complexity.

\subsection{Special factors in higher dimensions}

In one dimension, the factor complexity is related to the valence of left and
right special factors~\cite{MR2759107}. Similarly, in higher dimensions, 
the pattern complexity is related with the valence of special patterns with
connected support. In this section we shall generalize the notion of special
factors to higher dimensions, which will be the fundamental tool in the proof of~\Cref{maintheorem:ddim-complexity-is-FminusS}.

Since we will often consider all patterns which appear in configurations $x,y \in\alfad^{\ZZd}$,
it is practical to introduce the notations $\Lcal(x,y) \isdef \Lcal(x) \cup \Lcal(y)$ and $\Lcal_S(x,y) \isdef \Lcal_S(x)\cup\Lcal_S(y)$
for every finite support $S\subset\ZZd$.
For a pattern $w\in\Lcal_S(x,y)$,
and a position $\ell\in\ZZd\setminus S$,
let the \define{extensions at position $\ell\in\ZZd$} of the pattern $w$ within the
language $\Lcal(x,y)$ be
\[
E^\ell(w) \isdef \{u_{\ell} \colon u\in\Lcal_{S\cup\{\ell\}}(x,y)\text{ and } u|_{S}=w \}.
\]
Observe that the extensions
$E^\ell(w)$ always depend on the language $\Lcal(x,y)$
but we do not write $E_{x,y}^\ell(w)$ to lighten the notations.
Following the terminology for $d=1$,
we say that a pattern $w\in\Lcal_{S}(x,y)$ such that
$\#E^{\ell}(w)\geq 2$
is \define{special at position $\ell\in\ZZd$}.
Notice that we have the equality
\begin{equation*}
	\#\Lcal_{S\cup\{\ell\}}(x,y)
	= \sum_{w\in\Lcal_{S}(x,y)} \#E^{\ell}(w).
\end{equation*}
Let $\ell,r \in\ZZd\setminus S$ be positions such that $\ell \neq r$.
We say that a pattern $w\in\Lcal_{S}(x,y)$ 
is \define{bispecial at positions $\ell,r$}
if 
$\#E^{\ell}(w)\geq 2$
and
$\#E^r(w)\geq 2$.
Moreover, for a pattern $w\in\Lcal_S(x,y)$
let the \define{bilateral extensions at positions $\ell,r \in\ZZd\setminus S$} 
of the pattern $w$ within the
language $\Lcal(x,y)$ be
\[
E^{\ell,r}(w)=\{(u_{\ell},u_r) \colon u\in\Lcal_{S\cup\{\ell,r\}}(x,y)\text{ and } u|_{S}=w \}.
\]
The \define{bilateral multiplicity} $m^{\ell,r}(w)$ of the pattern $w$
at the positions $\ell,r\in\ZZd\setminus S$
within the language $\Lcal(x,y)$ is given by the expression
\begin{equation*}
	m^{\ell,r}(w) = \#E^{\ell,r}(w)
	-\#E^{\ell}(w)
	-\#E^r(w)
	+1.
\end{equation*}
We use the same terminology as when $d=1$ \cite{MR2759107}
to describe bispecial factors: we say that a pattern $w\in\Lcal_{S}(x,y)$ is \define{strong} (resp. \define{weak}, \define{neutral})
at the positions $\ell,r\in\ZZd\setminus S$
if $m^{\ell,r}(w)>0$ (resp. $m^{\ell,r}(w)<0$, $m^{\ell,r}(w)=0$).

Notice that we may interpret $E^{\ell,r}(w)$ as an undirected bipartite graph
called \define{extension graph}, see \cite{MR3845381}. The vertices are given by the disjoint union $V = E^{\ell}(w) \sqcup E^{r}(w)$ and we have an edge $(a,b)\in E^{\ell}(w) \times E^{r}(w)$ if there is $u \in \Lcal_{S \cup \{\ell,r\}}(x,y)$ such that $u_{\ell} = a$, $u_{r} = b$ and $u|_S = w$. In this manner $\# E^{\ell,r}(w)$ corresponds to the number of edges of the graph and $\# E^{\ell}(w)+\# E^{r}(w)$ corresponds to the number of vertices.

In the next lemma, we show that combinatorial properties of the extension graph
$E^{\ell,r}(w)$ impose lower bounds on the bilateral multiplicity of the
pattern $w$.

\begin{lemma}\label{lem:lema_evidente_trivial}
	Let $w\in\Lcal_S(x,y)$ be a pattern and $c$ be the number 
            of connected components of $E^{\ell,r}(w)$.
	\begin{enumerate}
		\item $m^{\ell,r}(w)\geq 1-c$.
		\item The extension graph $E^{\ell,r}(w)$ is acyclic if and only if 
            $m^{\ell,r}(w)=1-c$.
		\item If $E^{\ell,r}(w)$ is connected, then $m^{\ell,r}(w)\geq 0$.
		\item If $E^{\ell,r}(w)$ is connected and contains a cycle, then $m^{\ell,r}(w)>0$.
	\end{enumerate}
\end{lemma}

\begin{proof} (1) Notice that
	\begin{align*}
		m^{\ell,r}(w) &= \#E^{\ell,r}(w)
		-\#E^\ell(w)
		-\#E^r(w)
		+1\\
		&= \#\textrm{edges}
		-\#\textrm{vertices}
		+1.
	\end{align*}
	In each connected component we have that the number of edges is at least the number of vertices minus $1$. 
    (2)
If $m^{\ell,r}(w)=1-c$, it implies that
		$\#\textrm{edges} -\#\textrm{vertices}=-1$
        in each connected component. Therefore each connected component is a tree and we deduce that
		the extension graph $E^{\ell,r}(w)$ is acyclic.
If $m^{\ell,r}(w)>1-c$, it implies there is a connected component in which
		$\#\textrm{edges} -\#\textrm{vertices}>-1$. That connected component must contain a cycle.
        Thus, the extension graph $E^{\ell,r}(w)$ is not acyclic.
Part (3) is an immediate consequence of (1). 
Part (4) is an immediate consequence of (2). 
\end{proof}

\subsection{Complexity of indistinguishable asymptotic pairs with the flip condition}

Here we shall show that the flip condition along with indistinguishability impose that every pattern in the language must occur in a position which intersects the difference set. 
This property implies an upper bound for the pattern complexity.

\begin{lemma}\label{lem:exists-occ-intersecting-F}
	Let $x,y \in\alfad^{\ZZd}$ be an indistinguishable asymptotic pair satisfying the flip condition.
	For every finite nonempty subset $S \subset \ZZd$, we have
	$\Lcal_S(x,y)\subset\{\sigma^{\bn}(x)|_{S}\colon \bn\in F-S\}$.
    In particular, $\#\Lcal_S(x,y)\leq \#(F-S)$.
\end{lemma}

\begin{proof} 
	For $\symb{i} \in \alfad$, let $\bg_{\symb{i}} = i_x - i_y$ where $i_x,i_y$ are the unique positions in $F$ so that $x_{i_x} = \symb{i}$ and $y_{i_y} = \symb{i}$. Let $\mathcal{G}^{x-y} = \{\bg_\symb{0},\dots,\bg_d\}$, $\mathcal{G}^{y-x} = - \mathcal{G}^{x-y}$ and $\mathcal{G} = \mathcal{G}^{x-y} \cup \mathcal{G}^{y-x}$. 
	
	We claim that the collection $\mathcal{G} = \mathcal{G}^{x-y} \cup \mathcal{G}^{y-x}$ generates $\ZZ^d$ as a monoid. Indeed, the flip condition ensures that every position in $F$ occurs exactly once as an $i_x$ (and exactly once as an $i_y$). Moreover, for every $\symb{i} \in \alfad$, $\bg_{\symb{i}} = i_x - i_y \neq 0$. As $\symb{0}_x = 0$, using the previous properties we can suitably add elements from $\mathcal{G}^{x-y}$ to produce all canonical vectors $\{e_1,\dots,e_d\}$. Similarly, adding elements from $\mathcal{G}^{y-x}$ we can produce $\{-e_1,\dots,-e_d\}$. This provides a set which generates $\ZZ^d$ as a monoid.
	
	 For $\bm \in \ZZ^d$, let $\norm{\bm}_{\mathcal{G}}$ be the word metric generated by $\mathcal{G}$, that is, the least number $\ell$ so that $\bm$ can be written as a sum of $\ell$ elements of $\mathcal{G}$ ($\bzero$ can be written as a sum of zero elements). Denote by $d_{\mathcal{G}}(\bm,\bm') = \norm{\bm-\bm'}_{\mathcal{G}}$ and for a set $K \subset \ZZ^d$ let $d_{\mathcal{G}}(\bm,K) = \min_{\bk \in K} d_{\mathcal{G}}(\bm,\bk)$.
	
	We just show that $\Lcal_S(x)\subset \{\sigma^{\bn}(x)|_{S}\colon \bn\in F-S\}$, as the other case is analogous. Let $p \in \Lcal_S(x)$. There exists $\bn\in\ZZ^d$ such that $\sigma^{\bn}(x)\in[p]$. Choose $\bn$ as above such that it minimizes $d_{\mathcal{G}}(\bn,F-S)$. We claim that $d_{\mathcal{G}}(\bn,F-S) = 0$. If this were not the case, there is ${f}\in F$ and ${s} \in S$ so that $d_{\mathcal{G}}(\bn,F-S) = d_{\mathcal{G}}(\bn,{f}-{s}) = \norm{\bn - ({f}-{s})}_{\mathcal{G}} \geq 1$.
	
	By definition, we can write $\bn - ({f}-{s}) = \sum_{j = 1}^{d_{\mathcal{G}}(\bn,F-S)} {h}_j$ with each ${h}_j \in \mathcal{G}$. Consider ${h}_1$. There are two cases:
	\begin{enumerate}
		\item If ${h}_1 \in \mathcal{G}^{x-y}$ then ${h}_1 = \bg_i$ for some $0 \leq i \leq d$. Consider the support $S' = \{ i_x \} \cup (\bn + S)$ and let $q = x|_{S'}$. By definition $x \in [q]$ and as $x,y$ are indistinguishable, there must exist $\bk \in F-S'$ so that $\sigma^{\bk}(y) \in [q]$ and thus $\sigma^{\bk+\bn}(y) \in [p]$. There are again two cases. 
		\begin{enumerate}
			\item If $\bk + \bn \in F-S$, then as $x,y$ are indistinguishable there must exist $\bn' \in F-S$ so that $\sigma^{\bn'}(x) \in [p]$. This contradicts the choice of $\bn$.
			\item If $\bk + \bn \notin F-S$, then necessarily $\bk \in F-\{i_x\}$. We obtain that there is ${f}^* \in F$ so that $\bk = {f}^*-i_x$. As $\sigma^{\bk}(y)_{i_x} = y_{{f}^*-i_x + i_x} = \symb{i}$, it follows by the flip condition that ${f}^* = i_y$ and so $\bk = i_y - i_x = -\bg_{i} = -{h}_1$. We deduce that $\sigma^{ \bn - {h}_1 }(y) \in [p]$. As $\bk + \bn = \bn - {h}_1 \notin F-S$ and $x,y$ are asymptotic, we have that $\sigma^{ \bn - {h}_1 }(x) \in [p]$ and that \[   
			{d_{\mathcal{G}}(\bn-{h}_1,F-S)} \leq  \norm{ \bn - {h}_1 - {f}-{s} }_{\mathcal{G}} = \norm{ \bn - {f}-{s} }_{\mathcal{G}}-1 = {d_{\mathcal{G}}(\bn,F-S)}-1.\]
			Letting $\bn' = \bn-{h}_1$, we have $\sigma^{\bn'}(x)\in [p]$ and $d_{\mathcal{G}}(\bn',F-S) \leq d_{\mathcal{G}}(\bn,F-S)-1$, contradicting the choice of $\bn$.
		\end{enumerate}
		\item If ${h}_1 \in \mathcal{G}^{y-x}$, then ${h}_1 = -\bg_i$ for some $0 \leq i \leq d$. The argument is analogous except that now we consider $S' = \{ i_y \} \cup (\bn + S)$ and $q = y|_{S'}$.
	\end{enumerate}
	We conclude that $d_{\mathcal{G}}(\bn,F-S) = 0$ and thus $\bn \in F-S$.\end{proof}

\Cref{lem:exists-occ-intersecting-F} generalizes~\Cref{lem:words_appear_in_diff_set} which is valid in $\ZZ$ without resorting to the flip condition. We say that a permutation is cyclic if it consists of a single cycle and has no fixed points. In order to obtain a lower bound and thus the equality, we will use the following technical result. 

\begin{lemma}\label{lem:cyclic-permutation-map}
	Let $\pi\colon U\to U$ be a cyclic permutation on a finite set $U$.
	Let $A\subset U$ and $f\colon A\to B$ be a surjective map
	for some finite set $B$.
	If $A\neq U$, then
	\[
	\#\{(a,f(a))\mid a\in A\}
	\cup
	\{(\pi(a),f(a))\mid a\in A\}
	\geq
	\# A + \#B.
	\]
\end{lemma}

\begin{proof}
	Let $P_1 = \{(a,f(a))\mid a\in A\}$ and $P_2 = \{(\pi(a),f(a))\mid a\in A\}$. It is clear that $P_1$ and $A$ have the same number of elements, it suffices thus to show that for every $b \in B$, there is $a \in A$ such that $(\pi(a),f(a))\in P_2\setminus P_1$ and $f(a)=b$.
	
	Indeed, fix $b \in B$ and let $Q = \{ a \in A : f(a) = b\}$. Clearly $Q \neq \varnothing$ as $f$ is surjective. Consider the directed graph $G = (Q,E)$ where $(q,r)\in E$ if and only if $\pi(q) =r$. Notice that $Q$ does not contain a cycle due to $\pi$ being cyclic on $U$ and $A \neq U$, therefore there is $\bar{q} \in Q$ such that $\pi(\bar{q})\notin Q$. Then we have $(\pi(\bar{q}),f(\bar{q}))\in P_2 \setminus P_1$ and $f(\bar{q})=b$.
\end{proof}

We will now use~\Cref{lem:cyclic-permutation-map} to prove a lower bound for the pattern complexity of
asymptotic pairs satisfying the flip condition. Notice that we do not use indistinguishability in what follows.

\begin{lemma}\label{lem:ddim-complexity-is-at-least-FminusS}
	Let $x,y \in\alfad^{\ZZd}$ be an asymptotic pair
	satisfying the flip condition.
	Then for every finite nonempty connected subset $S\subset\ZZd$,
	we have
	\[
	\#\Lcal_S(x,y)\geq \#(F-S).
	\]
\end{lemma}

\begin{proof}
	We do the proof of the inequality by induction on the cardinality of $S$.
	If $S=\{a\}$ is a singleton, the inequality holds since $\Lcal_{\{a\}}(x) = \Lcal_{\{a\}}(y) = \alfa{d}$ and thus
	\[
	\#(\Lcal_{\{a\}}(x) \cup \Lcal_{\{a\}}(y))
	= d+1
	= \#F
	= \#(F-\{a\}).
	\]
	Proceeding by induction, we assume that 
	$\#\Lcal_S(x,y)\geq \#(F-S)$
	holds for
	some finite connected subset
	$S\subset \ZZd$
	and we want to show it for $S\cup\{a\}$ for some $a\in\ZZd\setminus S$
	such that $S\cup\{a\}$ is connected.
	
	Let 
	\[
	G = (F-(S\cup\{a\})) \setminus (F-S)
	= (F-a) \setminus (F-S)
	\]
	be the set of vectors $m \in \ZZd$ such that $m+(S\cup\{a\})$ intersects $F$
	without $m+S$ intersecting $F$.
	Since $S\cup\{a\}$ is connected, $G$ is a strict subset of $F-a$.
	
	Let $f\colon G\to \Lcal_{S}(x)$ be the map defined by $f(\bm)
	=\sigma^{\bm}(x)|_{S}$, and $g\colon G\to \alfad$ be the map defined by
	$g(\bm) = (\sigma^{\bm}(x))_{a} = x_{\bm +a}$ for every $\bm \in G$. 
	Notice that if $\bm\in G$, then $f(m)=\sigma^{\bm}(y)|_{S}$
	and $f(G)=\{\sigma^\bm(x)|_S\colon\bm\in G\}
	=\{\sigma^\bm(y)|_S\colon\bm\in G\}$.
	
	Putting together the flip condition and that $G+a$ is a strict subset of $F$, it follows that $g$ is injective and its image is a strict subset of the alphabet $\alfad$.
	Also notice that the flip condition implies that $y_{m+a}=(g(m)-1)\bmod(d+1)$.
	
	Since the asymptotic pair $(x,y)$ satisfies the flip condition, we 
	have that $f(G)$ is a subset of patterns that are special at position $a$.
	This provides a lower bound for the pattern complexity. More precisely,
	because of the flip condition,
	for every $\bm\in G$, we have that
	$\sigma^{\bm}(x)|_{S\cup\{a\}}$
	and
	$\sigma^{\bm}(y)|_{S\cup\{a\}}$
	are two distinct extensions to the support $S\cup\{a\}$ of the pattern
	$\sigma^{\bm}(x)|_{S}=\sigma^{\bm}(y)|_{S}$.
	Therefore, we have the inclusion
	\[
	\bigcup_{w\in f(G)}
	\{u\in\Lcal_{S\cup\{a\}}(x,y)\colon u|_{S}=w\}
	\supseteq
	\left\{\sigma^{\bm}(x)|_{S\cup\{a\}}\colon\bm\in G\right\}
	\cup
	\left\{\sigma^{\bm}(y)|_{S\cup\{a\}}\colon\bm\in G\right\}.
	\]
	The union on the left is disjoint, therefore, taking the cardinality of
	both sides, we obtain
	\begin{align*}
		\sum_{w\in f(G)} \#E^a(w)
		&\geq 
		\# \left(\left\{\sigma^{\bm}(x)|_{S\cup\{a\}}\colon\bm\in G\right\}
		\cup
		\left\{\sigma^{\bm}(y)|_{S\cup\{a\}}\colon\bm\in G\right\}\right)\\
		&=
		\#\left(\left\{(g(\bm),f(\bm))\colon\bm\in G\right\}
		\cup
		\left\{(g(\bm)-1 \bmod(d+1),f(\bm))\colon\bm\in G\right\}\right)\\
		&=
		\#\left(\left\{(s,fg^{-1}(s))\colon s \in g(G)\right\}
		\cup
		\left\{(s-1 \bmod(d+1),fg^{-1}(s))\colon s\in g(G)\right\}\right)\\
		&\geq  \# g(G) + \# f(G) = \#G + \#f(G).
	\end{align*}
	In the penultimate line, we use that $g$ is injective and thus $g^{-1} \colon g(G)\to G$ is a bijection. In particular, this implies that $fg^{-1}\colon g(G) \to f(G)$ is surjective. As $g(G)$ is a strict subset of the alphabet $\alfad$ we obtain the last line using Lemma~\ref{lem:cyclic-permutation-map}.

	Since every pattern in $\Lcal_{S}(x,y)$ can be extended in at least one way
	to position $a$, we have $\#E^a(w)\geq 1$ for every $w\in\Lcal_{S}(x,y)$.
	Also since $f(G) \subset \Lcal_{S}(x,y)$, we have
	\begin{align*}
		\#\Lcal_{S\cup\{a\}}(x,y)
		-\#\Lcal_{S}(x,y)
		&= \sum_{w\in\Lcal_{S}(x,y)} (\#E^a(w)-1)
		\geq \sum_{w\in f(G)} (\#E^a(w)-1)\\
		&= \sum_{w\in f(G)} \#E^a(w) - \#f(G)\\
		&\geq (\#G + \#f(G)) - \#f(G)
		= \#G.
	\end{align*}
	Therefore,
	\[
	\#\Lcal_{S\cup\{a\}}(x,y) 
	\geq \#\Lcal_{S}(x,y) + \# G
	\geq \#(F-S) + \#G
	= \#(F-(S\cup\{a\})).\qedhere
	\]
\end{proof}

\subsection{Properties of asymptotic pairs with the flip condition and complexity $\#(F-S)$}

In this subsection, we fix an asymptotic pair $(x,y)$ which satisfies the flip condition and study the properties we can obtain from the assumption that $\Lcal_S(x,y)= \#(F-S)$ for every nonempty connected finite $S\subset \ZZd$. For the remainder of the subsection, we fix a (possibly empty) connected set $S\subset \ZZd$ and $\ell,r \in \ZZd\setminus S$ such that $S\cup \{\ell\}, S\cup \{r\}$ and $S\cup \{\ell,r\}$ are connected. We also convene that $\Lcal_{\varnothing}(x,y) = \{\varepsilon\}$, where $\varepsilon$ is the empty pattern. As our proof will be by induction, we shall often make use of the following condition which will correspond to the inductive hypothesis.
 
\begin{definition}
	We say that $(x,y)$ satisfies condition (\textbf{IND}) if for every $S' \in \{S, S\cup \{\ell\}, S\cup \{r\}\}$ any pattern $p' \in \Lcal_{S'}(x,y)$ occurs intersecting $F$ in $x$, that is, for every $p' \in \Lcal_{S'}(x,y)$ there is $t'\in F-S'$ such that we have $\sigma^{t'}(x)\in [p']$.
\end{definition}

It is clear that condition (\textbf{IND}) implies that $\# \Lcal_{S'}(x,y)\leq \#(F-S')$. By~\Cref{lem:ddim-complexity-is-at-least-FminusS}, we have the other inequality and thus condition (\textbf{IND}) in fact states two things: that $\# \Lcal_{S'}(x,y)= \#(F-S')$ and that the position $t'\in F-S'$ such that $\sigma^{t'}(x)\in [p']$ is unique.

Our general strategy will be similar to the proof of~\Cref{lem:ddim-complexity-is-at-least-FminusS}, that is, we will look at the positions in $F-(S\cup \{\ell,r\})$ for which only one of $\{\ell,r\}$ intersects the difference set and nothing else does, this will provide us with the means to describe $E^{\ell,r}(w)$ for words $w \in \alfa{d}^S$ and ultimately to prove~\Cref{maintheorem:ddim-complexity-is-FminusS}.

Let $w \in \Lcal_S(x,y)$. We are going to define three special subsets of $E^{\ell,r}(w)$

\[ \Gamma_{\ell}(w) = \{ (x_{t+\ell},x_{t+r})\in E^{\ell,r}(w) : \mbox{ there is } t \in \ZZd \mbox{ such that } \sigma^t(x) \in [w], t+\ell \in F, (t+(S \cup \{r\})) \cap F = \varnothing \}. \]
\[ \Gamma_{r}(w) = \{ (x_{t+\ell},x_{t+r})\in E^{\ell,r}(w) : \mbox{ there is } t \in \ZZd \mbox{ such that } \sigma^t(x) \in [w], t+r \in F, (t+(S \cup \{\ell\})) \cap F = \varnothing \}. \]
\begin{align*}
	\Gamma_{\star}(w) = \{ (x_{t+\ell},x_{t+r})\in E^{\ell,r}(w) : & \mbox{ there is } t \in \ZZd \mbox{ with } \sigma^t(x) \in [w] \mbox{ such that either }  (t+S) \cap F \neq \varnothing \\
	& \mbox{or } t+\ell,t+r \in F \mbox{ and } (t+S) \cap F = \varnothing
	\}. 
\end{align*} 

The set $\Gamma_{\ell}(w)$ consists of all edges in $E^{\ell,r}(w)$ which can be obtained by a pattern (with support $S \cup \{\ell,r\}$ and whose restriction to $S$ is $w$) which intersects $F$ solely on position $\ell$. Similarly, $\Gamma_{r}(w)$ consists of all edges in $E^{\ell,r}(w)$ which can be obtained by a pattern which intersects $F$ solely on position $r$. Finally, $\Gamma_{\star}(w)$ represents the edges in $E^{\ell,r}(w)$ which occur in some pattern which intersects $F$, but does so either having $S$ intersect $F$, or having both $\ell$ and $r$ do so at the same time. Notice that these three sets cover all possible ways that $S\cup \{\ell,r\}$ can intersect the difference set $F$.

In particular, if we want to show that no pattern appears twice on $x$ intersecting the difference set, we would need to show that $\Gamma_{\ell}(w) \cap \Gamma_{r}(w) = \varnothing$. This will be the main goal of this section.

We shall first show that under condition \textbf{(IND)} we can use the set $\Gamma_{\star}(w)$ to bound the number of connected components of $E^{\ell,r}(w)$.

\begin{lemma}\label{lem:lema_fundamental_grafitos}
    For a symbol $\kappa \in \alfa{d}$, let us denote by $\kappa^* = (\kappa -1) \bmod{(d+1)}$. Assume condition \textbf{(IND)} and consider the bipartite graph $E^{\ell,r}(w)$. 
	\begin{enumerate}
        \item If $(a,b) \in \Gamma_{\ell}(w)$, then $(a^*,b)\in E^{\ell,r}(w)$ and there is $b'$ such that $(a^*,b') \in \Gamma_{\ell}(w)\cup \Gamma_{\star}(w)$.
        \item If $(a,b) \in \Gamma_{r}(w)$, then $(a,b^*)\in E^{\ell,r}(w)$ and there is $a'$ such that $(a',b^*) \in \Gamma_{r}(w)\cup \Gamma_{\star}(w)$.
		\item The number of connected components of $E^{\ell,r}(w)$ is bounded above by $\#\Gamma_{\star}(w)$.
	\end{enumerate} 
\end{lemma}

\begin{proof}
	Let us show (1). Fix $(a,b) \in \Gamma_{\ell}(w)$ and let $w'$ be the pattern with support $S\cup \{\ell,r\}$ such that $w'|_{S}=w$, $w'_{\ell}=a$ and $w'_{r}=b$. As $(a,b) \in \Gamma_{\ell}(w)$, there is $t \in \ZZd$ such that $t+\ell \in F$, $(t+(S\cup \{r\})) \cap F =\varnothing$ and $\sigma^{t}(x)\in [w']$. On the one hand, as $x,y$ are asymptotic with difference set $F$, we have $x|_{t+(S\cup \{r\})} = y|_{t+(S\cup \{r\})}$ and thus $y_{t+r} = x_{t+r}=b$. On the other hand, by the flip condition $y_{t+\ell} = x_{t+\ell}-1 \bmod{ d+1} = a^*$, which means we have both $(a,b)$ and $(a^*,b)$ in $E^{\ell,r}(w)$. Furthermore, if we let $w''$ be the pattern with support $S\cup \{\ell\}$ such that $w''|_{S}=w$ and $w''_{\ell}=a^*$, condition \textbf{(IND)} implies that $w''$ must occur in $x$ intersecting the difference set. It follows that there is $b'$ such that $(a^*,b') \in \Gamma_{\ell}(w)\cup \Gamma_{\star}(w)$. The second claim is analogous to the first one.
	
	Next we shall provide a bound on the number of edges of $\Gamma_{\ell}(w)$ and $\Gamma_{r}(w)$. Indeed, notice that by condition \textbf{(IND)} we have that \[\#\Gamma_{\ell}(w) \leq \#\{t\in \ZZd : t + \ell \in F \mbox{ and } (t + (S \cup \{r\})) \cap F = \varnothing)\}.\] As $S\cup \{\ell,r\}$ is connected, there is $u\in \ZZd$ with $\norm{u}_1 \leq 1$ such that $\ell +u \in S \cup \{r\}$. In particular, there is at least one $t \in \ZZd$ is such that $t+\ell \in F$ and $t+\ell+u \in F$. As $\# F = d+1$, we deduce that $\# \Gamma_{\ell}(w)\leq d$. Analogously, we have $\# \Gamma_{r}(w)\leq d$.
	
	Let us finally prove (3). Let $a \in E^{\ell}(w)$ and consider again the pattern $w'$ with support $S\cup \{\ell\}$ such that $w'|_{S}=w$ and $w'_{\ell}=a$. By condition \textbf{(IND)}, it must occur intersecting $F$ and thus we have that $a$ must occur in some edge in $\Gamma_{\ell}(w)\cup \Gamma_{\star}(w)$. If it occurs in an edge of $\Gamma_{\star}(w)$ we are done, otherwise by (1), we know it is connected to $a^*= a-1 \bmod{d+1}$ and that $a^*$ occurs in some edge in $\Gamma_{\ell}(w)\cup \Gamma_{\star}(w)$. If said edge is in $\Gamma_{\star}(w)$ we are done, otherwise we iterate the process, as $\# \Gamma_{\ell}(w) \leq d$ it follows that we eventually end up in a vertex which belongs to an edge in $\Gamma_{\star}(w)$. After an analogous argument for $b \in E^r(w)$ we obtain that every connected component of $E^{\ell,r}(w)$ must contain an edge of $\Gamma_{\star}(w)$, and thus the number of connected components is bounded by $\#\Gamma_{\star}(w)$. \end{proof}

Next we will have to estimate the size of $\Gamma_{\star}(w)$ in order to have a lower bound on the multiplicities $m^{\ell,r}(w)$. It turns out that
one particular case is harder to deal with and thus we shall give it a special name to simplify the upcoming statements.

\begin{definition}\label{example:simultaneous-extension-to-F}
	Let $S\subset\ZZd$ be a connected nonempty finite support
	and $\ell,r\in\ZZd\setminus S$ with $\ell\neq r$.
	We say that $(S,\ell,r)$ is \define{evil} 
	if there exists $t\in\ZZd$ such that
	$\{t+\ell,t+r\}\subset F$
	and $(t+S)\cap F=\varnothing$.
	
	We also say that $w \in \Lcal_S(x)$ is an \define{evil pattern} if for $t\in\ZZd$ such that
	$\{t+\ell,t+r\}\subset F$
	and $(t+S)\cap F=\varnothing$ we have $x_{t+s} = w_s$ for every $s \in S$.
\end{definition}

We remark that by definition the empty pattern $\varepsilon$ with support $S = \varnothing$ is not evil. \Cref{example:simultaneous-extension-to-F}
is illustrated in
\Cref{fig:simultaneous-extension-to-F} when $d=2$.

\begin{figure}[ht]
	\begin{tikzpicture}[scale=.5]
		\begin{scope}[yshift=3cm]
			\draw[fill=black!20] (0,0) -- ++ (-2,0) -- ++ (0,-1) -- ++ (1,0) 
			-- ++ (0,-1) -- ++ (1,0) -- ++ (0,2);
			\node at (-.5,-.5) {$F$};
		\end{scope}
		\begin{scope}[xshift=2cm]
			\draw (0,0) rectangle node {$S_1$} (5,3);
			\draw (4,3) rectangle node {$\ell_1$} (5,4);
			\draw (5,2) rectangle node {$r_1$} (6,3);
		\end{scope}
		\begin{scope}[xshift=10cm]
			\draw (0,0) rectangle node {$S_2$} (5,3);
			\draw (3,3) rectangle node {$\ell_2$} (4,4);
			\draw (5,2) rectangle node {$r_2$} (6,3);
		\end{scope}
		\begin{scope}[xshift=18cm]
			\draw (0,0) rectangle node {$S_3$} (5,3);
			\draw (-1,0) rectangle node {$\ell_3$} (0,1);
			\draw (0,-1) rectangle node {$r_3$} (1,0);
		\end{scope}
	\end{tikzpicture}
	\caption{
		$(S_1,\ell_1,r_1)$ is evil, as both $\ell_1$ and $r_1$ can simultaneously overlap $F$ without $S_1$ doing so. Notice that $(S_2,\ell_2,r_2)$ is not evil
		since $\ell_2-r_2\notin F-F$. $(S_3,\ell_3,r_3)$ is also not evil
		since the unique $t\in\ZZ^2$ with
		$t+\{\ell_3,r_3\}\subset F$
		is such that
		$(t+S_3)\cap F\neq\varnothing$.}
	\label{fig:simultaneous-extension-to-F}
\end{figure}
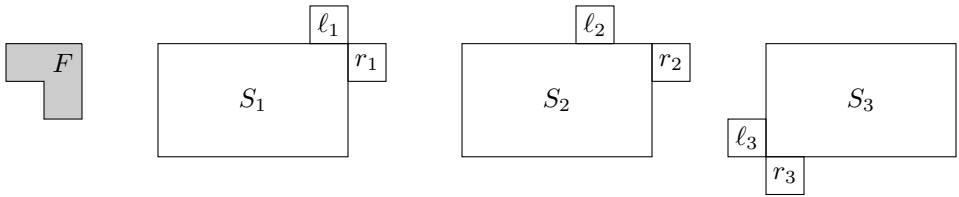

\begin{lemma}\label{lem:size_gamma_star}
	Let $w \in \Lcal_{S}(x,y)$ and assume conditon \textbf{(IND)}. If $w$ is an evil pattern, then $m^{\ell,r}(w) \geq-1$. If $w$ is non-evil, then $m^{\ell,r}(w) \geq 0$
\end{lemma}

\begin{proof}
	In the case when $S = \varnothing$, as $\ell \neq r$ there is at most a unique $t \in \ZZd$ such that $t+\ell, t+r \in F$, and thus $\# \Gamma_{\star}(w) \leq 1$. If $S \neq \varnothing$, condition \textbf{(IND)} implies that there is a unique $t \in \ZZd$ such that $(t+S)\cap F \neq \varnothing$ and $\sigma^{t}(x)\in [w]$. The second possibility, namely, that there is $t'$ such that $(t'+S)\cap F = \varnothing$ and $t'+\ell,t'+r \in F$ can only occur if $w$ is evil, therefore we obtain that $\# \Gamma_{\star}(w) \leq 1$ if $w$ is non-evil and $\# \Gamma_{\star}(w) \leq 2$ if $w$ is evil.
	
	By~\Cref{lem:lema_fundamental_grafitos} we obtain that the number of connected components of $E^{\ell,r}(w)$ is bounded by $1$ if $w$ is non-evil, and by $2$ if it is evil. Using~\Cref{lem:lema_evidente_trivial}, we obtain that 
	$m^{\ell,r}(w) \geq 0$ whenever $w$ is non-evil, and $m^{\ell,r}(w)\geq -1$ if $w$ is evil.
\end{proof}

Next we shall show that the bound in~\Cref{lem:size_gamma_star} is tight. In order to do so we will produce a formula for the sum of a bilateral multiplicities.

\begin{lemma}\label{lem:suma_odiosa}
	Suppose that $\#\Lcal_{S \cup \{\ell,r\}}(x,y) = \#(F-(S \cup \{\ell,r\}))$ and that condition (\textbf{IND}) holds. For every $w \in \Lcal_{S}(x,y)$, let $c$ be the number of connected components of $E^{\ell,r}(w)$. We have
	\[
	m^{\ell,r}(w)
    =
    1-c
	=
	\begin{cases}
		-1 & \text{ if } w \text{ is evil,}\\
		0  & \text{ otherwise.} 
	\end{cases}
	\]
    In particular, 
    the extension graph $E^{\ell,r}(w)$ contains no cycle.
\end{lemma}

\begin{proof}
	Let us first deal with the case $S\neq \varnothing$. Summing each term in the definition of multiplicity we obtain
	\begin{align*}
		\sum_{w\in\Lcal_S(x,y)} m^{\ell,r}(w) 
		&= \sum_{w\in\Lcal_S(x,y)} 
		\#E^{\ell,r}(w)
		-\sum_{w\in\Lcal_S(x,y)} 
		\#E^{\ell}(w)
		-\sum_{w\in\Lcal_S(x,y)} 
		\#E^{r}(w)
		+\sum_{w\in\Lcal_S(x,y)} 1\\
		&= \#\Lcal_{S\cup\{\ell,r\}}(x,y)
		-\#\Lcal_{S\cup\{\ell\}}(x,y)
		-\#\Lcal_{S\cup\{r\}}(x,y)
		+\#\Lcal_S(x,y).
	\end{align*}On the one hand, we have the hypothesis that $\#\Lcal_{S \cup \{\ell,r\}}(x,y) = \#(F-(S \cup \{\ell,r\}))$. On the other hand, condition (\textbf{IND}) implies that 
$\# \Lcal_{S'}(x,y)= \#(F-S')$ for every $S' \in \{S, S\cup \{\ell\}, S\cup \{r\}\}$. We obtain
\begin{align*}
	\sum_{w\in\Lcal_S(x,y)} m^{\ell,r}(w) 
	&= \#(F-(S\cup\{\ell,r\}))
	-\#(F-(S\cup\{r\}))
	-\#(F-(S\cup\{\ell\}))
	+\#(F-S)\\
	&= \# ( {(F-(S\cup \{\ell,r\}) ) \setminus (F - (S \cup \{r\}))) }-  \# {( (F-(S\cup \{\ell\}) ) \setminus (F - S))}\\
	&= \# {((F-\{\ell\} ) \setminus (F - (S \cup \{r\}))) }-  \# {(( F-\{\ell\} ) \setminus (F - S))}\\
	&= -\# (( (F-\{r\})\setminus (F-S)) \cap (F-\{\ell\})).
\end{align*} Clearly if $( (F-\{r\})\setminus (F-S)) \cap (F-\{\ell\}) = \varnothing$ the value of the sum is $0$. Otherwise, there is $t \in \ZZd$ such that $t+r \in F$, $t+\ell \in F$ but $t+s \notin F$ for every $s \in S$, which is precisely the evil case. Notice that as $\ell \neq r$, if such a $t$ exists then it is unique (because any non-trivial intersection $F \cap (t+F)$ has size at most $1$), and thus in this case the sum has value $-1$. We obtain thus that for $S \neq \varnothing$ we have
\[
\sum_{w \in \Lcal_{S}(x,y)}m^{\ell,r}(w)
=
\begin{cases}
	-1 & \text{ if } (S,\ell,r) \text{ is evil,}\\
	0  & \text{ otherwise.} 
\end{cases}
\]
Using~\Cref{lem:size_gamma_star} and the fact that there is exactly one evil pattern for an evil $(S,\ell,r)$, we obtain that $m^{\ell,r}(w) = 1-c$.
By~\Cref{lem:lema_evidente_trivial}, this implies that
    the extension graph $E^{\ell,r}(w)$ is acyclic.
	
	Let us now deal with the case when $S = \varnothing$. By assumption, $S\cup \{\ell,r\}$ is connected and thus without loss of generality we may write $r = \ell + \be_i$ for some $i \in \{1,\dots,d\}$. By definition $m^{\ell,r}(\varepsilon) = \# E^{\ell,r}(\varepsilon) -  \# E^{\ell}(\varepsilon)  -  \# E^{r}(\varepsilon) + 1$. Clearly $\# E^{\ell}(\varepsilon)  = \# E^{r}(\varepsilon) = \Lcal_{\{0\}}(x,y) = d+1$ and one can easily verify that for $U=\{\ell,\ell+\be_i\}$ we have $\# E^{\ell,r}(\varepsilon) = \#(F-U) = 2d+1$. It follows that $m^{\ell,r}(\varepsilon)=0$. As the number of connected components of $E^{\ell,r}(\varepsilon)$ is bounded by $\#\Gamma_{\star}(\varepsilon)\leq 1$, we conclude that $c = 1$ and thus $m^{\ell,r}(\varepsilon)=1-c$. By~\Cref{lem:lema_evidente_trivial}, this yields that
    the extension graph $E^{\ell,r}(w)$ is acyclic.
\end{proof}

\begin{lemma}\label{lem:lemadelfinal}
	Suppose that $\#\Lcal_{S \cup \{\ell,r\}}(x,y) = \#(F-(S \cup \{\ell,r\}))$ and that condition (\textbf{IND}) holds. For every non-evil $w \in \Lcal_{S}(x,y)$, if $\Gamma_{\ell}(w) \cap \Gamma_{r}(w)\neq\varnothing$, then the extension graph $E^{\ell,r}(w)$ contains a cycle. 
\end{lemma}

\begin{proof}
	Let $w\in \Lcal_{S}(x,y)$ be a non-evil pattern. It follows that $\#\Gamma_{\star}(w) =1$. Let $(\hat{a},\hat{b})$ be the sole element of $\Gamma_{\star}(w)$. 
	
	Suppose that $\Gamma_{\ell}(w) \cap \Gamma_{r}(w)\neq\varnothing$. Let $(a,b)\in \Gamma_{\ell}(w)\cap \Gamma_{r}(w)$ and $p$ be the pattern with support $S\cup \{\ell,r\}$ with $p|_{S}=w$, $p_{\ell}=a$, $p_{r}=b$. It follows that there exists $t,t' \in \ZZd$ such that $\sigma^{t}(x),\sigma^{t'}(x)\in [p]$ with $t+\ell \in F$, $(t+(S\cup \{r\}))\cap F = \varnothing$, and $t'+r \in F$, $(t'+(S\cup \{\ell\}))\cap F = \varnothing$. It follows that the subpatterns $q_{\ell}$ and $q_{r}$ of $p$, with supports $S\cup \{\ell\}$ and $S\cup \{r\}$ respectively, already intersect the support $F$ in $x$ (with vectors $t$ and $t'$ respectively), and thus if $t'' \in \ZZd$ is such that $\sigma^{t''}(x)\in [w]$ and $(t'' +S) \cap F \neq \varnothing$, then both $\hat{a} = x_{t''+\ell}\neq a$ and $\hat{b} = x_{t''+r}\neq b$ (otherwise the intersections of $q_{\ell}$ and $q_r$ with $F$ on $x$ would not be unique).

	Iterating~\Cref{lem:lema_fundamental_grafitos} part (1), we can construct a path $\pi_a$ in $E^{\ell,r}(w)$ from $a$ to $\hat{a}$ which begins by the edge $(a,b)$. Similarly, using part (2) we can build a path $\pi_2$ from $b$ to $\hat{b}$ which begins with the edge $(b,a)$. Notice that in each path either the edge $(\hat{a},\hat{b})$ does not appear, or it is the last edge on the path.
	
	If $(\hat{a},\hat{b})$ does not appear in either $\pi_a$ nor $\pi_b$, we can put them together to construct a path from $\hat{a}$ to $\hat{b}$ which does not use the edge $(\hat{a},\hat{b})$. Similarly, if $(\hat{a},\hat{b})$ appears in both of the paths, we can remove it from both paths and again we have a path from $\hat{a}$ to $\hat{b}$ which does not use the edge $(\hat{a},\hat{b})$. In both cases we obtain a cycle in $E^{\ell,r}(w)$.
	
	Finally, let us suppose that the (undirected) edge $(\hat{a},\hat{b})$ appears at the end of just one path. As both cases are analogous, let us assume that it appears in $\pi_a$. If we remove the first and last edge from $\pi_a$ we obtain a path from $b$ to $\hat{b}$ which does not use the edge $(a,b)$ and in which $a$ does not appear. If we remove the first edge from $\pi_b$ we obtain a path from $a$ to $\hat{b}$ which does not use the edge $(b,a)$ and where $b$ does not appear. Thus $a$ and $b$ are connected through a path that does not use the edge $(a,b)$ and thus we obtain a cycle in $E^{\ell,r}(w)$.
	 \end{proof}

\begin{remark}
	Using a variation of the previous argument, it is also possible to show for evil patterns that if $\Gamma_{\ell}(w) \cap \Gamma_{r}(w)\neq \varnothing$, then $E^{\ell,r}(w)$ contains a cycle. However, we shall not need that statement for the proof of~\Cref{maintheorem:ddim-complexity-is-FminusS}.
\end{remark}

\begin{proposition}\label{prop:complexity-implies-one-occurrence}
    Let $d\geq1$ and $x,y \in\alfad^{\ZZd}$ be an asymptotic pair
    satisfying the flip condition with difference set 
    $F = \{ \bzero, -\be_1,\dots,-\be_d\}$. 
    Assume that for every nonempty finite connected subset $S\subset\ZZd$, the
    pattern complexity of $x$ and $y$ is $\#\Lcal_S(x)=\#\Lcal_S(y) = \#(F-S)$.
    Then for every nonempty finite connected subset $S\subset\ZZd$ and 
    $p \in \Lcal_S(x)\cup\Lcal_S(y)$, we have
    \begin{equation}\label{eq:indistinguishable-for-p}
        \#\left(\occ_p(x)\setminus \occ_p(y)\right) 
        = 1
        = \#\left(\occ_p(y)\setminus \occ_p(x)\right).
    \end{equation}
\end{proposition}

\begin{proof}
	The proof is done by induction on the cardinality of $S$.
    If $\# S=1$, then it follows from the flip condition
    that \Cref{eq:indistinguishable-for-p} holds
    for all patterns $p\colon S\to\alfad$ with support $S$ of cardinality 1.
    Let us assume (by the induction hypothesis) that \Cref{eq:indistinguishable-for-p}
    holds for all supports $S\subset\ZZd$ with cardinality $\#S\leq k$
    for some integer $k\geq 1$.
    For the sake of contradiction, let us suppose that
    there exists a 
    finite connected subset $U\subset\ZZd$ of cardinality $\#U=k+1$
    such that
    \Cref{eq:indistinguishable-for-p} does not hold
    for some pattern $p \in \Lcal_U(x,y)$.
    As $\# \Lcal_U(x) = \#\Lcal_U(y) = \#(F-U)$, we may assume without loss of
    generality that 
    there exists a pattern $p \in \Lcal_U(x,y)$ such that
        $\#\left(\occ_p(x)\setminus \occ_p(y)\right)\geq 2$.
                In other words,
    there are two distinct vectors $t,t' \in F-U$ such that
    both $\sigma^t(x),\sigma^{t'}(x)\in [p]$.
	We claim that there exist $\ell,r \in U$ which satisfy the following properties:
	\begin{enumerate}[(A)]
		\item $\ell \neq r$.
		\item $U$ is a path in $\ZZd$ whose extreme elements are $\ell$ and $r$. 
		\item $\ell$ is the unique element of $U$ such that $t+\ell \in F$.
		\item $r$ is the unique element of $U$ such that $t'+r \in F$.
	\end{enumerate}
	Indeed, as $t,t' \in F-U$, there are $\ell,r \in U$ such that $t+\ell,t'+r \in F$. If $\ell = r$, then as $t \neq t'$ it follows that $t+\ell \neq t'+\ell$ are two distinct positions in $F$, however, as $\sigma^t(x),\sigma^{t'}(x)\in [p]$, it follows that $x_{t+\ell} = x_{t'+\ell} = p_{\ell}$ which contradicts the flip condition. Therefore we must have that $\ell \neq r$. As $U$ is connected, we may extract a path $U' \subseteq U$ in $\ZZd$ which connects $\ell$ and $r$. It follows that $p' = p|_{U'}$ also breaks \Cref{eq:indistinguishable-for-p} because $\sigma^t(x),\sigma^{t'}(x)\in [p']$ and $t,t' \in F-U'$, and thus from the induction hypothesis we must have $U' = U$ and thus $U$ is a path in $\ZZd$ whose extreme elements are $\ell,r$. Using the same idea, suppose there is $\ell' \in U$ such that $t+\ell' \in F$, then we could take the sub-path $U''\subset U$ which begins in $\ell'$ and ends in $r$ and again $p'' = p|_{U''}$ would violate the induction hypothesis. Thus $\ell$ is unique. Similarly, $r$ is unique.
	
	Let $S = U \setminus \{\ell,r\}$, $w = p|_S$, $a = p_{\ell}$ and $b = p_{r}$. Notice that the induction hypothesis implies that condition (\textbf{IND}) holds in $(x,y)$ for $(S,\ell,r)$.
	
	As $t \neq t'$, it follows that $w$ is not an evil pattern. Furthermore, conditions (C) and (D) give that $(a,b)\in \Gamma_{\ell}(w)$ and $(a,b)\in \Gamma_{r}(w)$ respectively. Therefore we have $\Gamma_{\ell}(w) \cap \Gamma_{r}(w) \neq \varnothing$ and thus~\Cref{lem:lemadelfinal},
	yields that the extension graph $E^{\ell,r}(w)$ contains a cycle. 
    This is a contradiction with~\Cref{lem:suma_odiosa} that states that $E^{\ell,r}(w)$ is acyclic.
    
    We conclude that~\Cref{eq:indistinguishable-for-p}
    holds for all patterns $p \in \Lcal_U(x,y)$
    for all finite connected subset $U\subset\ZZd$ of cardinality $\#U=k+1$.
\end{proof}

\subsection{Proof of~\Cref{maintheorem:ddim-complexity-is-FminusS}}

We shall now prove our characterization of indistinguishable asymptotic pairs with the flip condition through complexity. For the convenience of the reader, we recall the statement.

\begin{THEC}
    \MainTheoremComplexityStatement
\end{THEC}

\begin{proof}[Proof of \Cref{maintheorem:ddim-complexity-is-FminusS}]
	Let $x,y \in\alfad^{\ZZd}$ be an asymptotic pair
	satisfying the flip condition with difference set 
	$F = \{ \bzero, -\be_1,\dots,-\be_d\}$. 
    By~\Cref{prop:trivialite} it follows that (i) implies (ii).
	
	Assume (ii) holds and let $S\subset\ZZd$ be a finite nonempty connected subset. As $(x,y)$ is indistinguishable, we have $\Lcal_S(x)=\Lcal_S(y)=\Lcal_S(x,y)$. Furthermore, from \Cref{lem:ddim-complexity-is-at-least-FminusS}, we have
	$\#\Lcal_S(x,y) \geq \#(F-S)$.
	From~\Cref{lem:exists-occ-intersecting-F}, 
    we have $\#\Lcal_S(x,y) \leq \#(F-S)$.
	We conclude that
	$\#\Lcal_S(x)=\#\Lcal_S(y) = \#(F-S)$ and thus (iii) holds.

    In \Cref{prop:complexity-implies-one-occurrence}, we proved that (iii)
    implies (i).
\end{proof}

\Cref{maintheorem:ddim-complexity-is-FminusS} gives us two descriptions of the language of $x$ and $y$ for any connected support.

\begin{corollary}\label{cor:language-is-the-one-intersecting-F}
	Let $x,y \in\alfad^{\ZZd}$ be an indistinguishable asymptotic pair satisfying the flip condition.
	For every finite nonempty connected subset $S \subset \ZZd$, 
	we have that
	the maps 
	given by $n\mapsto\sigma^{\bn}(x)|_{S}$ 
	and $n\mapsto\sigma^{\bn}(y)|_{S}$ are two distinct bijections from
	$F-S$ to $\Lcal_S(x)$.
\end{corollary}

\begin{proof}
	We proved 
	$\Lcal_S(x)\subset\{\sigma^{\bn}(x)|_{S}\colon \bn\in F-S\}$
	in \Cref{lem:exists-occ-intersecting-F}.
	The equality 
	\begin{equation*}
		\Lcal_S(x)=\{\sigma^{\bn}(x)|_{S}\colon \bn\in F-S\}
	\end{equation*}
	follows from~\Cref{maintheorem:ddim-complexity-is-FminusS}. From this equality we deduce that the map $n\mapsto\sigma^{\bn}(x)|_{S}$ is a bijection
	from $F-S$ to $\Lcal_S(x)$. As $\Lcal_S(x)=\Lcal_S(y)$ we deduce that \[ \Lcal_S(x)= \Lcal_S(y) = \{\sigma^{\bn}(y)|_{S}\colon \bn\in F-S\}. \]
	Thus we conclude that the map $n\mapsto\sigma^{\bn}(y)|_{S}$ is another bijection from $F-S$ to $\Lcal_S(x)$.
\end{proof}

\subsection{Rectangular pattern complexity}

We have so far shown that for any indistinguishable pair $(x,y)\in \alfad^{\ZZd}$ which satisfies the flip condition and any finite nonempty connected subset $S \subset \ZZd$ we have $\# \Lcal_S(x) = \# \Lcal_S(y) = \#(F-S).$ This equation takes a beautiful form when $S$ is a $d$-dimensional box.

 For a positive integer vector $\bm = (m_1,\dots,m_d)\in \NN^d$, let $S(\bm) \subset\ZZd$ denote the support \[ S(m) \isdef \prod_{i =1}^d \llbracket 0,m_i-1\rrbracket = \{ (n_i)_{1 \leq i \leq d} \in \ZZd : 0 \leq n_i < m_i \mbox{ for every } 1 \leq i \leq d\}, \]
which represents the $d$-dimensional box whose sides have lengths $m_1,\dots,m_d$. Also, for $x \in \Sigma^{\ZZd}$ we write $\Lcal_{\bm}(x) = \Lcal_{S(\bm)}(x)$ to denote the set of patterns with support $S(\bm)$ occurring in $x$. We refer to the function which maps $\bm$ to $\# \Lcal_{\bm}(x)$ as the \define{rectangular pattern complexity} of $x$.

\begin{proof}[Proof of \Cref{maincorollary:ddim-complexity}]
	Let $m=(m_1,\dots,m_d) \in \NN^d$ be a positive integer vector. From~\Cref{maintheorem:ddim-complexity-is-FminusS}, we have that $\Lcal_{\bm}(x) = \Lcal_{\bm}(y) = \#(F-S(m))$.
	
	By a simple counting argument, we have that $\#(F-S(m))$ is equal to the volume of $S(m)$ plus the volume of each of its $d-1$ dimensional faces. We conclude that
	\[
	\#\Lcal_\bm(x) = \#\Lcal_{\bm}(y) = m_1\cdots m_d \left(1+\frac{1}{m_1}+\dots+\frac{1}{m_d}\right).\qedhere
	\]
\end{proof}

The geometrical interpretation of the rectangular complexity of an indistinguishable pair $(x,y)\in \alfad^{\ZZd}$ which satisfies the flip condition provides meaning to a curious relation that one can find perhaps by boredom or accident. Let us express the rectangular complexity as a real map $f  \colon  \RR^{d} \to \RR$ given by \[f(x_1,\dots,x_d) = x_1\cdots x_d \left( 1+ \frac{1}{x_1} + \dots + \frac{1}{x_d} \right). \]
If we consider the derivative of $f$ with respect to some $x_i$ we obtain \begin{align*}
	\frac{\partial }{\partial x_i}f(x_1\cdots x_d) & = \frac{x_1\cdots x_d}{x_i}  \left( 1+ \frac{1}{x_1} + \dots + \frac{1}{x_d} \right) + x_1\cdots x_d \left( \frac{-1}{x_i^2}\right)\\
	& =  \frac{x_1\cdots x_d}{x_i}\left( 1+ \frac{1}{x_1} + \dots + \frac{1}{x_d} - \frac{1}{x_i} \right).
\end{align*}

In other words, the derivative of the complexity function of a $d$-dimensional indistinguishable pair $(x,y)\in \alfad^{\ZZd}$ which satisfies the flip condition with respect to any variable yields the rectangular complexity function of a $(d-1)$-dimensional indistinguishable pair $(x',y')\in \alfa{d-1}^{\ZZd}$ which satisfies the flip condition. The geometrical interpretation is that as this complexity corresponds to the volume of a $d$-dimensional box of size $(m_1\cdots m_d)$ plus the sum of the volume of the $(d-1)$-dimensional faces, then taking the derivative with respect to a canonical direction $e_i$ yields from the box the volume $\frac{m_1\cdots m_d}{m_i}$ of the corresponding $(d-1)$-dimensional face orthogonal to $e_i$, and for each of the $(d-1)$-dimensional faces we either obtain the $(d-2)$-dimensional face orthogonal to $e_i$, or $0$ if the $(d-1)$-dimensional face is orthogonal to $e_i$.

\section{Characteristic Sturmian configurations in $\ZZd$}~\label{sec:multidim_sturmian}

In this section, we introduce characteristic multidimensional Sturmian
configurations from codimension-one cut and project schemes.
We show that they are examples of indistinguishable asymptotic pairs satisfying
the flip condition.

\subsection{Codimension-one cut and project schemes for symbolic configurations}\label{sec:alternate-def-cut-and-project}

Cut and project schemes of codimension-one (dimension of the internal space)
can be defined in several ways (for a different definition
see~\cite{MR3480345}). In what follows we follow the formalism of~\cite[\S
7]{baake_aperiodic_2013}, but note that we need to adapt it in order
to describe symbolic configurations over a lattice $\ZZd$.
Let $d\geq 1$ be an integer and
\[
\begin{array}{rccc}
    \pi:&\RR^{d+1} & \to & \RR^d\\
        &(x_0,x_1,\dots,x_d) & \mapsto & (x_1,\dots,x_d)
\end{array}
\]
be the projection of $\RR^{d+1}$ in the \define{physical space} $\RR^d$.
Let
$\alpha_0=1$, $\alpha_{d+1}=0$ and
$\balpha=(\alpha_1,\dots,\alpha_d)\in[0,1)^d$ be a
\define{totally irrational} vector,
that is such that $\{1,\alpha_1, \dots, \alpha_d\}$ is linearly independent
over $\QQ$. 
Let
\[
\begin{array}{rccc}
    \pi_{\rm int}:&\RR^{d+1} & \to & \RR/\ZZ \\
        &(x_0,x_1,\dots,x_d) & \mapsto & 
        \sum_{i=0}^d x_i \alpha_i
\end{array}
\]
be the projection of $\RR^{d+1}$ in the \define{internal space} $\RR/\ZZ$.
Consider the lattice $\Lcal=\ZZ^{d+1}\subset\RR^{d+1}$
whose image is $\pi(\Lcal)=\ZZ^d$.
This is the setting of a codimension-one \define{cut and project scheme} 
summarized in the following diagram adapted from 
\cite[\S 7.2]{baake_aperiodic_2013}:
\[
\begin{tikzcd}[ampersand replacement=\&]
    W
    \arrow[draw=none]{r}[sloped,auto=false]{\subset}
    \&[-25pt] \RR/\ZZ 
    \& \RR^{d+1} \arrow[swap]{l}{\pi_{\rm int}} \arrow{r}{\pi} \& \RR^d\\
    \&\pi_{\rm int}(\Lcal) 
    \arrow[draw=none]{u}[sloped,auto=false]{\subset}
    \arrow[draw=none]{u}[swap]{\text{ dense}}
    \& \Lcal 
    \arrow[draw=none]{u}[sloped,auto=false]{\subset}
    \arrow{l}
    \arrow{r}%
    \& \pi(\Lcal) 
    \arrow[draw=none]{u}[sloped,auto=false]{\subset}
    \arrow[bend left=20,"\star" description]{ll}{}
    \&[-15pt] \curlywedge(W)
    \arrow[draw=none]{l}[sloped,auto=false]{\supset}
\end{tikzcd}
\]

\begin{remark}
The usual condition imposed in cut and project schemes is that
$\pi|_\Lcal$ is injective, see \cite[\S 7.2]{baake_aperiodic_2013},
which does not hold in our case.
    Here, it is more convenient to relax this condition to
\begin{equation}\label{eq:relaxed-CPS-condition}
    \Ker\pi\cap\Lcal \subseteq \Ker\pi_{\rm int}.
\end{equation}
Of course if $\pi|_\Lcal$ is injective, then \eqref{eq:relaxed-CPS-condition}
is satisfied since $\Ker\pi\cap\Lcal = \{0\} \subset \Ker\pi_{\rm int}$.
Also, we may observe that if \eqref{eq:relaxed-CPS-condition} holds,
then the \define{star map} $\pi(\Lcal)\to \pi_{\rm int}(\Lcal)$
is still well defined:
\[
x\mapsto x^\star\isdef \pi_{\rm int}\left(\Lcal\cap\pi^{-1}(x)\right).
\]
\end{remark}

With the definition of $\pi$ and $\pi_{\rm int}$ above, we have
that \eqref{eq:relaxed-CPS-condition} holds since
    $\Ker\pi\cap\Lcal=\ZZ\times\{0\}^d \subseteq \Ker\pi_{\rm int}$.
Moreover,
\begin{equation}\label{eq:star-map-is-nalpha}
    n^\star=\alpha\cdot n \bmod 1
\end{equation}
for every $n\in\pi(\Lcal)=\ZZd$.
For a given \define{window} %
$W\subset\RR/\ZZ$ in the internal space,
\[
    \curlywedge(W)\isdef \{x\in L\mid x^\star\in W\}
\]
is the projection set within the cut and project scheme, where $L=\pi(\Lcal)$.
If $W\subset \RR/\ZZ$ is a relatively compact set
with non-empty interior, 
    any translate
$t+\curlywedge(W)$ 
    of the projection set,
    $t\in\RR^d$, is called a \define{model set}.

If $W=[0,1)$, then $\curlywedge(W)=\ZZd$.
Thus, if $W\subset [0,1)$, then $\curlywedge(W)\subset\ZZd$.
Moreover, if
$\{W_i\}_{i\in\{0,\dots,d\}}$
is a partition of $[0,1)$, then
$\{\curlywedge(W_i)\}_{i\in\{0,\dots,d\}}$
is a partition of $\ZZd$.
Using this idea,
we now build configurations $\ZZ^d\to\alfad$
according to a partition of $\RR/\ZZ$,
or equivalently of the interval $[0,1)$,
into consecutive intervals.

\begin{definition}
    Let
    $\balpha=(\alpha_1,\dots,\alpha_d)\in[0,1)^d$ be a totally irrational vector
    and
    $\tau$ be the permutation 
    of $\{\symb{1},\dots,d\}\cup \{\symb{0},d+1\}$ which fixes $\{\symb{0},d+1\}$ and
    such that $0=\alpha_{\tau(d+1)}<\alpha_{\tau(d)}<\dots<\alpha_{\tau(\symb{1})}<\alpha_{\tau(\symb{0})}=1$. 
For every $i\in\{0,1,\dots,d\}$, let
    \begin{align*}
        W_i &= [1-\alpha_{\tau(i)},1-\alpha_{\tau(i+1)}),
        &W'_i &= (1-\alpha_{\tau(i)},1-\alpha_{\tau(i+1)}]
    \end{align*} 
be such that 
$\{W_i\}_{i\in\{0,\dots,d\}}$
and
$\{W'_i\}_{i\in\{0,\dots,d\}}$
are two partitions of the interval $[0,1)$.
The configurations
	\[
	\begin{array}{rccl}
	c_{\balpha}:&\ZZd & \to & \alfad\\
	&\bn & \mapsto &  
        i \text{ if } n^\star\in W_i
	\end{array}
	\quad
	\text{ and }
	\quad
	\begin{array}{rccl}
	c'_{\balpha}:&\ZZd & \to & \alfad\\
	&\bn & \mapsto & 
        i \text{ if } n^\star\in W'_i
	\end{array}
	\]
    are respectively the \define{lower} and \define{upper characteristic
    $d$-dimensional Sturmian configurations} with slope $\alpha\in[0,1)^d$.
    Moreover, if $\rho\in\RR/\ZZ$, the configurations
	\[
	\begin{array}{rccl}
	s_{\balpha,\rho}:&\ZZd & \to & \alfad\\
	&\bn & \mapsto &  
         i \text{ if } n^\star+\rho\in W_i
	\end{array}
	\quad
	\text{ and }
	\quad
	\begin{array}{rccl}
	s'_{\balpha,\rho}:&\ZZd & \to & \alfad\\
	&\bn & \mapsto & 
        i \text{ if } n^\star+\rho\in W'_i
	\end{array}
	\]
    are respectively 
    the \define{lower} and \define{upper $d$-dimensional
    Sturmian configurations} with \define{slope} $\alpha\in[0,1)^d$ and \define{intercept} $\rho\in\RR/\ZZ$.
\end{definition}

It turns out that configurations $s_{\balpha,\rho}$ and $s'_{\balpha,\rho}$ can
be expressed by a formula involving a sum of differences of floor functions
thus extending the definition of Sturmian sequences by mechanical sequences
\cite{MR0000745}.  It also reminds of recent progresses on Nivat's conjecture
where configurations with low pattern complexity are proved to be sums of
periodic configurations \cite{MR4073398,MR3775541}, although here it involves a
sum of non-periodic configurations.

\begin{lemma}\label{lem:sturmian-formual-from-cut-and-project}
Let
$\balpha=(\alpha_1,\dots,\alpha_d)\in[0,1)^d$ be a totally irrational vector
and $\rho\in\RR/\ZZ$.
    The lower and upper $d$-dimensional
    Sturmian configurations with slope $\balpha$ and intercept $\rho$ are given
    by the following rules:
	\[
	\begin{array}{rccl}
	s_{\balpha,\rho}:&\ZZd & \to & \alfad\\
	&\bn & \mapsto &  \sum\limits_{i=1}^d \left(\lfloor\alpha_i+\bn\cdot\balpha+\rho\rfloor
	 -\lfloor\bn\cdot\balpha+\rho\rfloor\right),
	\end{array}
\]
and
    \[
	\begin{array}{rccl}
	s'_{\balpha,\rho}:&\ZZd & \to & \alfad\\
	&\bn & \mapsto & \sum\limits_{i=1}^d \left(\lceil\alpha_i+\bn\cdot\balpha+\rho\rceil
	-\lceil\bn\cdot\balpha+\rho\rceil\right).
	\end{array}
	\]
\end{lemma}

\begin{proof}
    Let $n\in\ZZd$ and
    $j\in\{0,1,\dots,d\}$
    be such that $n^\star+\rho\in W_j$.
    Therefore $s_{\balpha,\rho}(n)=j$.
    From \Cref{eq:star-map-is-nalpha}, recall that we have $n^\star=n\cdot\alpha\bmod 1$.
    Since the intervals $W_0$, $W_1$, $\dots$, $W_d$ are ordered from left to
    right on the interval $[0,1)$, we must have
    \begin{align*}
        s_{\balpha,\rho}(n) = j 
          &= \#\{ i\in\{1,\dots,d\} : 1-\alpha_i \leq n\cdot\alpha+\rho-\lfloor n\cdot\alpha+\rho\rfloor \}\\
          &= \#\{ i\in\{1,\dots,d\} : 1 \leq \alpha_i + n\cdot\alpha+\rho-\lfloor n\cdot\alpha+\rho\rfloor \}\\
          &= \#\{ i\in\{1,\dots,d\} : \lfloor\alpha_i+\bn\cdot\balpha+\rho\rfloor
	 -\lfloor\bn\cdot\balpha+\rho\rfloor = 1 \}\\
          &= \sum\limits_{i=1}^d \left(\lfloor\alpha_i+\bn\cdot\balpha+\rho\rfloor
	 -\lfloor\bn\cdot\balpha+\rho\rfloor\right).
    \end{align*}
    The proof for $s'_{\alpha,\rho}$ follows the same argument after replacing
    inequalities ($\leq$) by strict inequalities ($<$)
    and floor functions ($\lfloor\cdot\rfloor$) by ceil functions
    ($\lceil\cdot\rceil$).
\end{proof}

When $d=1$, $s_{\balpha,\rho}$ and $s'_{\balpha,\rho}$ 
correspond to lower and upper mechanical words defined in~\cite{MR0000745},
see also \cite{MR1905123,MR1970391,MR1997038}.
When $d=2$, they are in direct correspondence to discrete planes as defined in~\cite{MR1782038,MR1906478,MR2074952}.
See also Jolivet's Ph.D. thesis~\cite{jolivet_phd_2013}. 
In general, we say that a configuration in $\alfad^{\ZZd}$ is
\define{Sturmian}, if it coincides either with $s_{\balpha,\rho}$ or
$s'_{\balpha,\rho}$ for some 
$\rho \in \RR$ and
totally irrational $\balpha \in[0,1)^d$.

When $\rho=0$, we have
$s_{\alpha,0}=c_{\alpha}$
and
$s'_{\alpha,0}=c'_{\alpha}$.
Thus, 
\Cref{eq:characteristic-sturmian-formula-in-intro}
in the Introduction
follows from
\Cref{lem:sturmian-formual-from-cut-and-project}.

The fact that the configurations $c_\alpha$ and $c'_\alpha$ are encodings
of codimension-one cut and project schemes is illustrated
with $\balpha=(\alpha_1,\alpha_2)=(\sqrt{2}/2,\sqrt{19}-4)$
in \Cref{fig:intro-sturmian-config-pair-discrete-plane}
in which we see a discrete plane in dimension 3 of normal vector
$(1-\alpha_1,\alpha_1-\alpha_2,\alpha_2) \approx(0.293, 0.348, 0.359)$.
Below, we exhibit another example and compute its language for
small rectangular supports.

\begin{example}
    \def\thevector{{(\sqrt{3}-1,\sqrt{2}-1)}}
	Let $\balpha=\thevector$. 
	The $2$-dimensional characteristic Sturmian configurations $c_{\balpha}$ and
	$c'_{\balpha}$ are shown on 
	Figure~\ref{fig:2d-sturmian-config-pair}. In order to motivate the main ideas of the proofs in the next section, we explicitly compute the language of these configurations for some rectangular supports of small size.
	
	\begin{figure}[ht]
		\begin{center}
			\begin{tabular}{cc}
                $c_\thevector$  & $c'_\thevector$\\[3mm]
				\begin{tikzpicture}
[baseline=-\the\dimexpr\fontdimen22\textfont2\relax,ampersand replacement=\&]
  \matrix[matrix of math nodes,nodes={
       minimum size=1.2ex,text width=1.2ex,
       text height=1.2ex,inner sep=3pt,draw={gray!20},align=center,
       anchor=base
     }, row sep=1pt,column sep=1pt
  ] (config) {
{\color{black!60}\symb{2}}\&{\color{black!60}\symb{1}}\&{\color{black!60}\symb{0}}\&{\color{black!60}\symb{2}}\&{\color{black!60}\symb{2}}\&{\color{black!60}\symb{1}}\&{\color{black!60}\symb{0}}\&{\color{black!60}\symb{2}}\&{\color{black!60}\symb{2}}\&{\color{black!60}\symb{1}}\&{\color{black!60}\symb{0}}\&{\color{black!60}\symb{2}}\&{\color{black!60}\symb{1}}\&{\color{black!60}\symb{1}}\&{\color{black!60}\symb{0}}\\
{\color{black!60}\symb{1}}\&{\color{black!60}\symb{0}}\&{\color{black!60}\symb{2}}\&{\color{black!60}\symb{1}}\&{\color{black!60}\symb{1}}\&{\color{black!60}\symb{0}}\&{\color{black!60}\symb{2}}\&{\color{black!60}\symb{1}}\&{\color{black!60}\symb{0}}\&{\color{black!60}\symb{2}}\&{\color{black!60}\symb{2}}\&{\color{black!60}\symb{1}}\&{\color{black!60}\symb{0}}\&{\color{black!60}\symb{2}}\&{\color{black!60}\symb{2}}\\
{\color{black!60}\symb{2}}\&{\color{black!60}\symb{2}}\&{\color{black!60}\symb{1}}\&{\color{black!60}\symb{0}}\&{\color{black!60}\symb{2}}\&{\color{black!60}\symb{2}}\&{\color{black!60}\symb{1}}\&{\color{black!60}\symb{0}}\&{\color{black!60}\symb{2}}\&{\color{black!60}\symb{1}}\&{\color{black!60}\symb{0}}\&{\color{black!60}\symb{2}}\&{\color{black!60}\symb{2}}\&{\color{black!60}\symb{1}}\&{\color{black!60}\symb{0}}\\
{\color{black!60}\symb{1}}\&{\color{black!60}\symb{0}}\&{\color{black!60}\symb{2}}\&{\color{black!60}\symb{2}}\&{\color{black!60}\symb{1}}\&{\color{black!60}\symb{0}}\&{\color{black!60}\symb{2}}\&{\color{black!60}\symb{2}}\&{\color{black!60}\symb{1}}\&{\color{black!60}\symb{0}}\&{\color{black!60}\symb{2}}\&{\color{black!60}\symb{1}}\&{\color{black!60}\symb{1}}\&{\color{black!60}\symb{0}}\&{\color{black!60}\symb{2}}\\
{\color{black!60}\symb{0}}\&{\color{black!60}\symb{2}}\&{\color{black!60}\symb{1}}\&{\color{black!60}\symb{1}}\&{\color{black!60}\symb{0}}\&{\color{black!60}\symb{2}}\&{\color{black!60}\symb{1}}\&{\color{black!60}\symb{0}}\&{\color{black!60}\symb{2}}\&{\color{black!60}\symb{2}}\&{\color{black!60}\symb{1}}\&{\color{black!60}\symb{0}}\&{\color{black!60}\symb{2}}\&{\color{black!60}\symb{2}}\&{\color{black!60}\symb{1}}\\
{\color{black!60}\symb{2}}\&{\color{black!60}\symb{1}}\&{\color{black!60}\symb{0}}\&{\color{black!60}\symb{2}}\&{\color{black!60}\symb{2}}\&{\color{black!60}\symb{1}}\&{\color{black!60}\symb{0}}\&{\color{black!60}\symb{2}}\&{\color{black!60}\symb{1}}\&{\color{black!60}\symb{1}}\&{\color{black!60}\symb{0}}\&{\color{black!60}\symb{2}}\&{\color{black!60}\symb{1}}\&{\color{black!60}\symb{0}}\&{\color{black!60}\symb{2}}\\
{\color{black!60}\symb{1}}\&{\color{black!60}\symb{0}}\&{\color{black!60}\symb{2}}\&{\color{black!60}\symb{1}}\&{\color{black!60}\symb{0}}\&{\color{black!60}\symb{2}}\&{\color{black!60}\symb{2}}\&{\color{black!60}\symb{1}}\&{\color{black!60}\symb{0}}\&{\color{black!60}\symb{2}}\&{\color{black!60}\symb{2}}\&{\color{black!60}\symb{1}}\&{\color{black!60}\symb{0}}\&{\color{black!60}\symb{2}}\&{\color{black!60}\symb{1}}\\
{\color{black!60}\symb{2}}\&{\color{black!60}\symb{2}}\&{\color{black!60}\symb{1}}\&{\color{black!60}\symb{0}}\&{\color{black!60}\symb{2}}\&{\color{black!60}\symb{1}}\&{\color{red}\symb{1}}\&{\color{red}\symb{0}}\&{\color{black!60}\symb{2}}\&{\color{black!60}\symb{1}}\&{\color{black!60}\symb{0}}\&{\color{black!60}\symb{2}}\&{\color{black!60}\symb{2}}\&{\color{black!60}\symb{1}}\&{\color{black!60}\symb{0}}\\
{\color{black!60}\symb{1}}\&{\color{black!60}\symb{0}}\&{\color{black!60}\symb{2}}\&{\color{black!60}\symb{2}}\&{\color{black!60}\symb{1}}\&{\color{black!60}\symb{0}}\&{\color{black!60}\symb{2}}\&{\color{red}\symb{2}}\&{\color{black!60}\symb{1}}\&{\color{black!60}\symb{0}}\&{\color{black!60}\symb{2}}\&{\color{black!60}\symb{1}}\&{\color{black!60}\symb{0}}\&{\color{black!60}\symb{2}}\&{\color{black!60}\symb{2}}\\
{\color{black!60}\symb{0}}\&{\color{black!60}\symb{2}}\&{\color{black!60}\symb{1}}\&{\color{black!60}\symb{0}}\&{\color{black!60}\symb{2}}\&{\color{black!60}\symb{2}}\&{\color{black!60}\symb{1}}\&{\color{black!60}\symb{0}}\&{\color{black!60}\symb{2}}\&{\color{black!60}\symb{2}}\&{\color{black!60}\symb{1}}\&{\color{black!60}\symb{0}}\&{\color{black!60}\symb{2}}\&{\color{black!60}\symb{1}}\&{\color{black!60}\symb{1}}\\
{\color{black!60}\symb{2}}\&{\color{black!60}\symb{1}}\&{\color{black!60}\symb{0}}\&{\color{black!60}\symb{2}}\&{\color{black!60}\symb{1}}\&{\color{black!60}\symb{1}}\&{\color{black!60}\symb{0}}\&{\color{black!60}\symb{2}}\&{\color{black!60}\symb{1}}\&{\color{black!60}\symb{0}}\&{\color{black!60}\symb{2}}\&{\color{black!60}\symb{2}}\&{\color{black!60}\symb{1}}\&{\color{black!60}\symb{0}}\&{\color{black!60}\symb{2}}\\
{\color{black!60}\symb{0}}\&{\color{black!60}\symb{2}}\&{\color{black!60}\symb{2}}\&{\color{black!60}\symb{1}}\&{\color{black!60}\symb{0}}\&{\color{black!60}\symb{2}}\&{\color{black!60}\symb{2}}\&{\color{black!60}\symb{1}}\&{\color{black!60}\symb{0}}\&{\color{black!60}\symb{2}}\&{\color{black!60}\symb{1}}\&{\color{black!60}\symb{1}}\&{\color{black!60}\symb{0}}\&{\color{black!60}\symb{2}}\&{\color{black!60}\symb{1}}\\
{\color{black!60}\symb{2}}\&{\color{black!60}\symb{1}}\&{\color{black!60}\symb{1}}\&{\color{black!60}\symb{0}}\&{\color{black!60}\symb{2}}\&{\color{black!60}\symb{1}}\&{\color{black!60}\symb{0}}\&{\color{black!60}\symb{2}}\&{\color{black!60}\symb{2}}\&{\color{black!60}\symb{1}}\&{\color{black!60}\symb{0}}\&{\color{black!60}\symb{2}}\&{\color{black!60}\symb{2}}\&{\color{black!60}\symb{1}}\&{\color{black!60}\symb{0}}\\
{\color{black!60}\symb{1}}\&{\color{black!60}\symb{0}}\&{\color{black!60}\symb{2}}\&{\color{black!60}\symb{2}}\&{\color{black!60}\symb{1}}\&{\color{black!60}\symb{0}}\&{\color{black!60}\symb{2}}\&{\color{black!60}\symb{1}}\&{\color{black!60}\symb{0}}\&{\color{black!60}\symb{2}}\&{\color{black!60}\symb{2}}\&{\color{black!60}\symb{1}}\&{\color{black!60}\symb{0}}\&{\color{black!60}\symb{2}}\&{\color{black!60}\symb{2}}\\
{\color{black!60}\symb{2}}\&{\color{black!60}\symb{2}}\&{\color{black!60}\symb{1}}\&{\color{black!60}\symb{0}}\&{\color{black!60}\symb{2}}\&{\color{black!60}\symb{2}}\&{\color{black!60}\symb{1}}\&{\color{black!60}\symb{0}}\&{\color{black!60}\symb{2}}\&{\color{black!60}\symb{1}}\&{\color{black!60}\symb{1}}\&{\color{black!60}\symb{0}}\&{\color{black!60}\symb{2}}\&{\color{black!60}\symb{1}}\&{\color{black!60}\symb{0}}\\
};
\node[draw,rectangle,dashed,help lines,fit=(config), inner sep=0.5ex] {};
\end{tikzpicture} &
                \begin{tikzpicture}
[baseline=-\the\dimexpr\fontdimen22\textfont2\relax,ampersand replacement=\&]
  \matrix[matrix of math nodes,nodes={
       minimum size=1.2ex,text width=1.2ex,
       text height=1.2ex,inner sep=3pt,draw={gray!20},align=center,
       anchor=base
     }, row sep=1pt,column sep=1pt
  ] (config) {
{\color{black!60}\symb{2}}\&{\color{black!60}\symb{1}}\&{\color{black!60}\symb{0}}\&{\color{black!60}\symb{2}}\&{\color{black!60}\symb{2}}\&{\color{black!60}\symb{1}}\&{\color{black!60}\symb{0}}\&{\color{black!60}\symb{2}}\&{\color{black!60}\symb{2}}\&{\color{black!60}\symb{1}}\&{\color{black!60}\symb{0}}\&{\color{black!60}\symb{2}}\&{\color{black!60}\symb{1}}\&{\color{black!60}\symb{1}}\&{\color{black!60}\symb{0}}\\
{\color{black!60}\symb{1}}\&{\color{black!60}\symb{0}}\&{\color{black!60}\symb{2}}\&{\color{black!60}\symb{1}}\&{\color{black!60}\symb{1}}\&{\color{black!60}\symb{0}}\&{\color{black!60}\symb{2}}\&{\color{black!60}\symb{1}}\&{\color{black!60}\symb{0}}\&{\color{black!60}\symb{2}}\&{\color{black!60}\symb{2}}\&{\color{black!60}\symb{1}}\&{\color{black!60}\symb{0}}\&{\color{black!60}\symb{2}}\&{\color{black!60}\symb{2}}\\
{\color{black!60}\symb{2}}\&{\color{black!60}\symb{2}}\&{\color{black!60}\symb{1}}\&{\color{black!60}\symb{0}}\&{\color{black!60}\symb{2}}\&{\color{black!60}\symb{2}}\&{\color{black!60}\symb{1}}\&{\color{black!60}\symb{0}}\&{\color{black!60}\symb{2}}\&{\color{black!60}\symb{1}}\&{\color{black!60}\symb{0}}\&{\color{black!60}\symb{2}}\&{\color{black!60}\symb{2}}\&{\color{black!60}\symb{1}}\&{\color{black!60}\symb{0}}\\
{\color{black!60}\symb{1}}\&{\color{black!60}\symb{0}}\&{\color{black!60}\symb{2}}\&{\color{black!60}\symb{2}}\&{\color{black!60}\symb{1}}\&{\color{black!60}\symb{0}}\&{\color{black!60}\symb{2}}\&{\color{black!60}\symb{2}}\&{\color{black!60}\symb{1}}\&{\color{black!60}\symb{0}}\&{\color{black!60}\symb{2}}\&{\color{black!60}\symb{1}}\&{\color{black!60}\symb{1}}\&{\color{black!60}\symb{0}}\&{\color{black!60}\symb{2}}\\
{\color{black!60}\symb{0}}\&{\color{black!60}\symb{2}}\&{\color{black!60}\symb{1}}\&{\color{black!60}\symb{1}}\&{\color{black!60}\symb{0}}\&{\color{black!60}\symb{2}}\&{\color{black!60}\symb{1}}\&{\color{black!60}\symb{0}}\&{\color{black!60}\symb{2}}\&{\color{black!60}\symb{2}}\&{\color{black!60}\symb{1}}\&{\color{black!60}\symb{0}}\&{\color{black!60}\symb{2}}\&{\color{black!60}\symb{2}}\&{\color{black!60}\symb{1}}\\
{\color{black!60}\symb{2}}\&{\color{black!60}\symb{1}}\&{\color{black!60}\symb{0}}\&{\color{black!60}\symb{2}}\&{\color{black!60}\symb{2}}\&{\color{black!60}\symb{1}}\&{\color{black!60}\symb{0}}\&{\color{black!60}\symb{2}}\&{\color{black!60}\symb{1}}\&{\color{black!60}\symb{1}}\&{\color{black!60}\symb{0}}\&{\color{black!60}\symb{2}}\&{\color{black!60}\symb{1}}\&{\color{black!60}\symb{0}}\&{\color{black!60}\symb{2}}\\
{\color{black!60}\symb{1}}\&{\color{black!60}\symb{0}}\&{\color{black!60}\symb{2}}\&{\color{black!60}\symb{1}}\&{\color{black!60}\symb{0}}\&{\color{black!60}\symb{2}}\&{\color{black!60}\symb{2}}\&{\color{black!60}\symb{1}}\&{\color{black!60}\symb{0}}\&{\color{black!60}\symb{2}}\&{\color{black!60}\symb{2}}\&{\color{black!60}\symb{1}}\&{\color{black!60}\symb{0}}\&{\color{black!60}\symb{2}}\&{\color{black!60}\symb{1}}\\
{\color{black!60}\symb{2}}\&{\color{black!60}\symb{2}}\&{\color{black!60}\symb{1}}\&{\color{black!60}\symb{0}}\&{\color{black!60}\symb{2}}\&{\color{black!60}\symb{1}}\&{\color{red}\symb{0}}\&{\color{red}\symb{2}}\&{\color{black!60}\symb{2}}\&{\color{black!60}\symb{1}}\&{\color{black!60}\symb{0}}\&{\color{black!60}\symb{2}}\&{\color{black!60}\symb{2}}\&{\color{black!60}\symb{1}}\&{\color{black!60}\symb{0}}\\
{\color{black!60}\symb{1}}\&{\color{black!60}\symb{0}}\&{\color{black!60}\symb{2}}\&{\color{black!60}\symb{2}}\&{\color{black!60}\symb{1}}\&{\color{black!60}\symb{0}}\&{\color{black!60}\symb{2}}\&{\color{red}\symb{1}}\&{\color{black!60}\symb{1}}\&{\color{black!60}\symb{0}}\&{\color{black!60}\symb{2}}\&{\color{black!60}\symb{1}}\&{\color{black!60}\symb{0}}\&{\color{black!60}\symb{2}}\&{\color{black!60}\symb{2}}\\
{\color{black!60}\symb{0}}\&{\color{black!60}\symb{2}}\&{\color{black!60}\symb{1}}\&{\color{black!60}\symb{0}}\&{\color{black!60}\symb{2}}\&{\color{black!60}\symb{2}}\&{\color{black!60}\symb{1}}\&{\color{black!60}\symb{0}}\&{\color{black!60}\symb{2}}\&{\color{black!60}\symb{2}}\&{\color{black!60}\symb{1}}\&{\color{black!60}\symb{0}}\&{\color{black!60}\symb{2}}\&{\color{black!60}\symb{1}}\&{\color{black!60}\symb{1}}\\
{\color{black!60}\symb{2}}\&{\color{black!60}\symb{1}}\&{\color{black!60}\symb{0}}\&{\color{black!60}\symb{2}}\&{\color{black!60}\symb{1}}\&{\color{black!60}\symb{1}}\&{\color{black!60}\symb{0}}\&{\color{black!60}\symb{2}}\&{\color{black!60}\symb{1}}\&{\color{black!60}\symb{0}}\&{\color{black!60}\symb{2}}\&{\color{black!60}\symb{2}}\&{\color{black!60}\symb{1}}\&{\color{black!60}\symb{0}}\&{\color{black!60}\symb{2}}\\
{\color{black!60}\symb{0}}\&{\color{black!60}\symb{2}}\&{\color{black!60}\symb{2}}\&{\color{black!60}\symb{1}}\&{\color{black!60}\symb{0}}\&{\color{black!60}\symb{2}}\&{\color{black!60}\symb{2}}\&{\color{black!60}\symb{1}}\&{\color{black!60}\symb{0}}\&{\color{black!60}\symb{2}}\&{\color{black!60}\symb{1}}\&{\color{black!60}\symb{1}}\&{\color{black!60}\symb{0}}\&{\color{black!60}\symb{2}}\&{\color{black!60}\symb{1}}\\
{\color{black!60}\symb{2}}\&{\color{black!60}\symb{1}}\&{\color{black!60}\symb{1}}\&{\color{black!60}\symb{0}}\&{\color{black!60}\symb{2}}\&{\color{black!60}\symb{1}}\&{\color{black!60}\symb{0}}\&{\color{black!60}\symb{2}}\&{\color{black!60}\symb{2}}\&{\color{black!60}\symb{1}}\&{\color{black!60}\symb{0}}\&{\color{black!60}\symb{2}}\&{\color{black!60}\symb{2}}\&{\color{black!60}\symb{1}}\&{\color{black!60}\symb{0}}\\
{\color{black!60}\symb{1}}\&{\color{black!60}\symb{0}}\&{\color{black!60}\symb{2}}\&{\color{black!60}\symb{2}}\&{\color{black!60}\symb{1}}\&{\color{black!60}\symb{0}}\&{\color{black!60}\symb{2}}\&{\color{black!60}\symb{1}}\&{\color{black!60}\symb{0}}\&{\color{black!60}\symb{2}}\&{\color{black!60}\symb{2}}\&{\color{black!60}\symb{1}}\&{\color{black!60}\symb{0}}\&{\color{black!60}\symb{2}}\&{\color{black!60}\symb{2}}\\
{\color{black!60}\symb{2}}\&{\color{black!60}\symb{2}}\&{\color{black!60}\symb{1}}\&{\color{black!60}\symb{0}}\&{\color{black!60}\symb{2}}\&{\color{black!60}\symb{2}}\&{\color{black!60}\symb{1}}\&{\color{black!60}\symb{0}}\&{\color{black!60}\symb{2}}\&{\color{black!60}\symb{1}}\&{\color{black!60}\symb{1}}\&{\color{black!60}\symb{0}}\&{\color{black!60}\symb{2}}\&{\color{black!60}\symb{1}}\&{\color{black!60}\symb{0}}\\
};
\node[draw,rectangle,dashed,help lines,fit=(config), inner sep=0.5ex] {};
\end{tikzpicture} %
			\end{tabular}
		\end{center}
		\caption{The $2$-dimensional configurations $c_{\balpha}$ and
			$c'_{\balpha}$ when $\balpha=\thevector$ are shown 
			on the support $\llbracket -7,7\rrbracket \times \llbracket -7,7\rrbracket$. The two configurations are 
            equal except on the difference set $F=\{(0,0),(-1,0),(0,-1)\}$ shown in red.}
		\label{fig:2d-sturmian-config-pair}
	\end{figure}
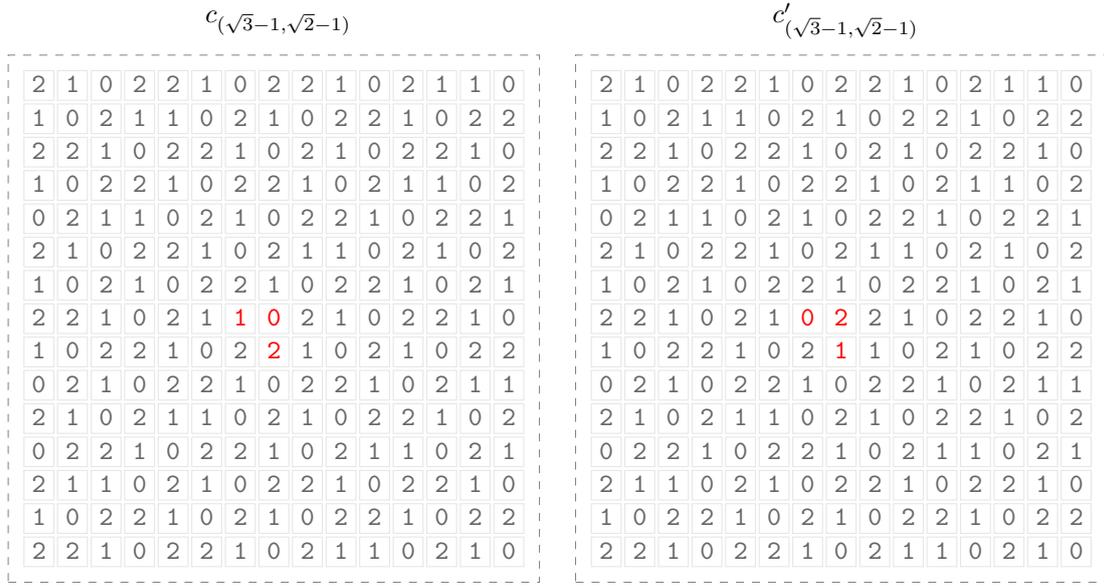
	
	The patterns of shape $(1,3)$ that we see in $c_\thevector$ and
	$c'_\thevector$ are
	\[
	\arraycolsep=1.5pt
	\left\{
\begin{array}
{ccccccc}
\begin{tikzpicture}
[baseline=-\the\dimexpr\fontdimen22\textfont2\relax,ampersand replacement=\&]
  \matrix[matrix of math nodes,nodes={
       minimum size=1.2ex,text width=1.2ex,
       text height=1.2ex,inner sep=3pt,draw={gray!20},align=center,
       anchor=base
     }, row sep=1pt,column sep=1pt
  ] (config) {
{\color{black!60}\symb{0}}\\
{\color{black!60}\symb{2}}\\
{\color{black!60}\symb{0}}\\
};
\end{tikzpicture},&\begin{tikzpicture}
[baseline=-\the\dimexpr\fontdimen22\textfont2\relax,ampersand replacement=\&]
  \matrix[matrix of math nodes,nodes={
       minimum size=1.2ex,text width=1.2ex,
       text height=1.2ex,inner sep=3pt,draw={gray!20},align=center,
       anchor=base
     }, row sep=1pt,column sep=1pt
  ] (config) {
{\color{black!60}\symb{0}}\\
{\color{black!60}\symb{2}}\\
{\color{black!60}\symb{1}}\\
};
\end{tikzpicture},&\begin{tikzpicture}
[baseline=-\the\dimexpr\fontdimen22\textfont2\relax,ampersand replacement=\&]
  \matrix[matrix of math nodes,nodes={
       minimum size=1.2ex,text width=1.2ex,
       text height=1.2ex,inner sep=3pt,draw={gray!20},align=center,
       anchor=base
     }, row sep=1pt,column sep=1pt
  ] (config) {
{\color{black!60}\symb{1}}\\
{\color{black!60}\symb{0}}\\
{\color{black!60}\symb{2}}\\
};
\end{tikzpicture},&\begin{tikzpicture}
[baseline=-\the\dimexpr\fontdimen22\textfont2\relax,ampersand replacement=\&]
  \matrix[matrix of math nodes,nodes={
       minimum size=1.2ex,text width=1.2ex,
       text height=1.2ex,inner sep=3pt,draw={gray!20},align=center,
       anchor=base
     }, row sep=1pt,column sep=1pt
  ] (config) {
{\color{black!60}\symb{1}}\\
{\color{black!60}\symb{2}}\\
{\color{black!60}\symb{1}}\\
};
\end{tikzpicture},&\begin{tikzpicture}
[baseline=-\the\dimexpr\fontdimen22\textfont2\relax,ampersand replacement=\&]
  \matrix[matrix of math nodes,nodes={
       minimum size=1.2ex,text width=1.2ex,
       text height=1.2ex,inner sep=3pt,draw={gray!20},align=center,
       anchor=base
     }, row sep=1pt,column sep=1pt
  ] (config) {
{\color{black!60}\symb{2}}\\
{\color{black!60}\symb{0}}\\
{\color{black!60}\symb{2}}\\
};
\end{tikzpicture},&\begin{tikzpicture}
[baseline=-\the\dimexpr\fontdimen22\textfont2\relax,ampersand replacement=\&]
  \matrix[matrix of math nodes,nodes={
       minimum size=1.2ex,text width=1.2ex,
       text height=1.2ex,inner sep=3pt,draw={gray!20},align=center,
       anchor=base
     }, row sep=1pt,column sep=1pt
  ] (config) {
{\color{black!60}\symb{2}}\\
{\color{black!60}\symb{1}}\\
{\color{black!60}\symb{0}}\\
};
\end{tikzpicture},&\begin{tikzpicture}
[baseline=-\the\dimexpr\fontdimen22\textfont2\relax,ampersand replacement=\&]
  \matrix[matrix of math nodes,nodes={
       minimum size=1.2ex,text width=1.2ex,
       text height=1.2ex,inner sep=3pt,draw={gray!20},align=center,
       anchor=base
     }, row sep=1pt,column sep=1pt
  ] (config) {
{\color{black!60}\symb{2}}\\
{\color{black!60}\symb{1}}\\
{\color{black!60}\symb{2}}\\
};
\end{tikzpicture}\\
\end{array}
\right\}
	\]
	The patterns of shape $(3,1)$ that we see in $c_\thevector$ and
	$c'_\thevector$ are
	\[
	\arraycolsep=1.5pt
	\left\{
\begin{array}
{ccccccc}
\begin{tikzpicture}
[baseline=-\the\dimexpr\fontdimen22\textfont2\relax,ampersand replacement=\&]
  \matrix[matrix of math nodes,nodes={
       minimum size=1.2ex,text width=1.2ex,
       text height=1.2ex,inner sep=3pt,draw={gray!20},align=center,
       anchor=base
     }, row sep=1pt,column sep=1pt
  ] (config) {
{\color{black!60}\symb{0}}\&{\color{black!60}\symb{2}}\&{\color{black!60}\symb{1}}\\
};
\end{tikzpicture},&\begin{tikzpicture}
[baseline=-\the\dimexpr\fontdimen22\textfont2\relax,ampersand replacement=\&]
  \matrix[matrix of math nodes,nodes={
       minimum size=1.2ex,text width=1.2ex,
       text height=1.2ex,inner sep=3pt,draw={gray!20},align=center,
       anchor=base
     }, row sep=1pt,column sep=1pt
  ] (config) {
{\color{black!60}\symb{0}}\&{\color{black!60}\symb{2}}\&{\color{black!60}\symb{2}}\\
};
\end{tikzpicture},&\begin{tikzpicture}
[baseline=-\the\dimexpr\fontdimen22\textfont2\relax,ampersand replacement=\&]
  \matrix[matrix of math nodes,nodes={
       minimum size=1.2ex,text width=1.2ex,
       text height=1.2ex,inner sep=3pt,draw={gray!20},align=center,
       anchor=base
     }, row sep=1pt,column sep=1pt
  ] (config) {
{\color{black!60}\symb{1}}\&{\color{black!60}\symb{0}}\&{\color{black!60}\symb{2}}\\
};
\end{tikzpicture},&\begin{tikzpicture}
[baseline=-\the\dimexpr\fontdimen22\textfont2\relax,ampersand replacement=\&]
  \matrix[matrix of math nodes,nodes={
       minimum size=1.2ex,text width=1.2ex,
       text height=1.2ex,inner sep=3pt,draw={gray!20},align=center,
       anchor=base
     }, row sep=1pt,column sep=1pt
  ] (config) {
{\color{black!60}\symb{1}}\&{\color{black!60}\symb{1}}\&{\color{black!60}\symb{0}}\\
};
\end{tikzpicture},&\begin{tikzpicture}
[baseline=-\the\dimexpr\fontdimen22\textfont2\relax,ampersand replacement=\&]
  \matrix[matrix of math nodes,nodes={
       minimum size=1.2ex,text width=1.2ex,
       text height=1.2ex,inner sep=3pt,draw={gray!20},align=center,
       anchor=base
     }, row sep=1pt,column sep=1pt
  ] (config) {
{\color{black!60}\symb{2}}\&{\color{black!60}\symb{1}}\&{\color{black!60}\symb{0}}\\
};
\end{tikzpicture},&\begin{tikzpicture}
[baseline=-\the\dimexpr\fontdimen22\textfont2\relax,ampersand replacement=\&]
  \matrix[matrix of math nodes,nodes={
       minimum size=1.2ex,text width=1.2ex,
       text height=1.2ex,inner sep=3pt,draw={gray!20},align=center,
       anchor=base
     }, row sep=1pt,column sep=1pt
  ] (config) {
{\color{black!60}\symb{2}}\&{\color{black!60}\symb{1}}\&{\color{black!60}\symb{1}}\\
};
\end{tikzpicture},&\begin{tikzpicture}
[baseline=-\the\dimexpr\fontdimen22\textfont2\relax,ampersand replacement=\&]
  \matrix[matrix of math nodes,nodes={
       minimum size=1.2ex,text width=1.2ex,
       text height=1.2ex,inner sep=3pt,draw={gray!20},align=center,
       anchor=base
     }, row sep=1pt,column sep=1pt
  ] (config) {
{\color{black!60}\symb{2}}\&{\color{black!60}\symb{2}}\&{\color{black!60}\symb{1}}\\
};
\end{tikzpicture}\\
\end{array}
\right\}
	\]
	The patterns of shape $(2,2)$ that we see in $c_\thevector$ and
	$c'_\thevector$ are
	\[
	\arraycolsep=1.5pt
	\left\{
\begin{array}
{cccccccc}
\begin{tikzpicture}
[baseline=-\the\dimexpr\fontdimen22\textfont2\relax,ampersand replacement=\&]
  \matrix[matrix of math nodes,nodes={
       minimum size=1.2ex,text width=1.2ex,
       text height=1.2ex,inner sep=3pt,draw={gray!20},align=center,
       anchor=base
     }, row sep=1pt,column sep=1pt
  ] (config) {
{\color{black!60}\symb{0}}\&{\color{black!60}\symb{2}}\\
{\color{black!60}\symb{2}}\&{\color{black!60}\symb{1}}\\
};
\end{tikzpicture},&\begin{tikzpicture}
[baseline=-\the\dimexpr\fontdimen22\textfont2\relax,ampersand replacement=\&]
  \matrix[matrix of math nodes,nodes={
       minimum size=1.2ex,text width=1.2ex,
       text height=1.2ex,inner sep=3pt,draw={gray!20},align=center,
       anchor=base
     }, row sep=1pt,column sep=1pt
  ] (config) {
{\color{black!60}\symb{1}}\&{\color{black!60}\symb{0}}\\
{\color{black!60}\symb{0}}\&{\color{black!60}\symb{2}}\\
};
\end{tikzpicture},&\begin{tikzpicture}
[baseline=-\the\dimexpr\fontdimen22\textfont2\relax,ampersand replacement=\&]
  \matrix[matrix of math nodes,nodes={
       minimum size=1.2ex,text width=1.2ex,
       text height=1.2ex,inner sep=3pt,draw={gray!20},align=center,
       anchor=base
     }, row sep=1pt,column sep=1pt
  ] (config) {
{\color{black!60}\symb{1}}\&{\color{black!60}\symb{0}}\\
{\color{black!60}\symb{2}}\&{\color{black!60}\symb{2}}\\
};
\end{tikzpicture},&\begin{tikzpicture}
[baseline=-\the\dimexpr\fontdimen22\textfont2\relax,ampersand replacement=\&]
  \matrix[matrix of math nodes,nodes={
       minimum size=1.2ex,text width=1.2ex,
       text height=1.2ex,inner sep=3pt,draw={gray!20},align=center,
       anchor=base
     }, row sep=1pt,column sep=1pt
  ] (config) {
{\color{black!60}\symb{1}}\&{\color{black!60}\symb{1}}\\
{\color{black!60}\symb{0}}\&{\color{black!60}\symb{2}}\\
};
\end{tikzpicture},&\begin{tikzpicture}
[baseline=-\the\dimexpr\fontdimen22\textfont2\relax,ampersand replacement=\&]
  \matrix[matrix of math nodes,nodes={
       minimum size=1.2ex,text width=1.2ex,
       text height=1.2ex,inner sep=3pt,draw={gray!20},align=center,
       anchor=base
     }, row sep=1pt,column sep=1pt
  ] (config) {
{\color{black!60}\symb{2}}\&{\color{black!60}\symb{1}}\\
{\color{black!60}\symb{0}}\&{\color{black!60}\symb{2}}\\
};
\end{tikzpicture},&\begin{tikzpicture}
[baseline=-\the\dimexpr\fontdimen22\textfont2\relax,ampersand replacement=\&]
  \matrix[matrix of math nodes,nodes={
       minimum size=1.2ex,text width=1.2ex,
       text height=1.2ex,inner sep=3pt,draw={gray!20},align=center,
       anchor=base
     }, row sep=1pt,column sep=1pt
  ] (config) {
{\color{black!60}\symb{2}}\&{\color{black!60}\symb{1}}\\
{\color{black!60}\symb{1}}\&{\color{black!60}\symb{0}}\\
};
\end{tikzpicture},&\begin{tikzpicture}
[baseline=-\the\dimexpr\fontdimen22\textfont2\relax,ampersand replacement=\&]
  \matrix[matrix of math nodes,nodes={
       minimum size=1.2ex,text width=1.2ex,
       text height=1.2ex,inner sep=3pt,draw={gray!20},align=center,
       anchor=base
     }, row sep=1pt,column sep=1pt
  ] (config) {
{\color{black!60}\symb{2}}\&{\color{black!60}\symb{2}}\\
{\color{black!60}\symb{1}}\&{\color{black!60}\symb{0}}\\
};
\end{tikzpicture},&\begin{tikzpicture}
[baseline=-\the\dimexpr\fontdimen22\textfont2\relax,ampersand replacement=\&]
  \matrix[matrix of math nodes,nodes={
       minimum size=1.2ex,text width=1.2ex,
       text height=1.2ex,inner sep=3pt,draw={gray!20},align=center,
       anchor=base
     }, row sep=1pt,column sep=1pt
  ] (config) {
{\color{black!60}\symb{2}}\&{\color{black!60}\symb{2}}\\
{\color{black!60}\symb{1}}\&{\color{black!60}\symb{1}}\\
};
\end{tikzpicture}\\
\end{array}
\right\}
	\]
	The patterns of shape $(2,3)$ that we see in $c_\thevector$ and
	$c'_\thevector$ are
	\[
	\arraycolsep=1.5pt
	\left\{
\begin{array}
{ccccccccccc}
\begin{tikzpicture}
[baseline=-\the\dimexpr\fontdimen22\textfont2\relax,ampersand replacement=\&]
  \matrix[matrix of math nodes,nodes={
       minimum size=1.2ex,text width=1.2ex,
       text height=1.2ex,inner sep=3pt,draw={gray!20},align=center,
       anchor=base
     }, row sep=1pt,column sep=1pt
  ] (config) {
{\color{black!60}\symb{0}}\&{\color{black!60}\symb{2}}\\
{\color{black!60}\symb{2}}\&{\color{black!60}\symb{1}}\\
{\color{black!60}\symb{0}}\&{\color{black!60}\symb{2}}\\
};
\end{tikzpicture},&\begin{tikzpicture}
[baseline=-\the\dimexpr\fontdimen22\textfont2\relax,ampersand replacement=\&]
  \matrix[matrix of math nodes,nodes={
       minimum size=1.2ex,text width=1.2ex,
       text height=1.2ex,inner sep=3pt,draw={gray!20},align=center,
       anchor=base
     }, row sep=1pt,column sep=1pt
  ] (config) {
{\color{black!60}\symb{0}}\&{\color{black!60}\symb{2}}\\
{\color{black!60}\symb{2}}\&{\color{black!60}\symb{1}}\\
{\color{black!60}\symb{1}}\&{\color{black!60}\symb{0}}\\
};
\end{tikzpicture},&\begin{tikzpicture}
[baseline=-\the\dimexpr\fontdimen22\textfont2\relax,ampersand replacement=\&]
  \matrix[matrix of math nodes,nodes={
       minimum size=1.2ex,text width=1.2ex,
       text height=1.2ex,inner sep=3pt,draw={gray!20},align=center,
       anchor=base
     }, row sep=1pt,column sep=1pt
  ] (config) {
{\color{black!60}\symb{1}}\&{\color{black!60}\symb{0}}\\
{\color{black!60}\symb{0}}\&{\color{black!60}\symb{2}}\\
{\color{black!60}\symb{2}}\&{\color{black!60}\symb{1}}\\
};
\end{tikzpicture},&\begin{tikzpicture}
[baseline=-\the\dimexpr\fontdimen22\textfont2\relax,ampersand replacement=\&]
  \matrix[matrix of math nodes,nodes={
       minimum size=1.2ex,text width=1.2ex,
       text height=1.2ex,inner sep=3pt,draw={gray!20},align=center,
       anchor=base
     }, row sep=1pt,column sep=1pt
  ] (config) {
{\color{black!60}\symb{1}}\&{\color{black!60}\symb{0}}\\
{\color{black!60}\symb{2}}\&{\color{black!60}\symb{2}}\\
{\color{black!60}\symb{1}}\&{\color{black!60}\symb{0}}\\
};
\end{tikzpicture},&\begin{tikzpicture}
[baseline=-\the\dimexpr\fontdimen22\textfont2\relax,ampersand replacement=\&]
  \matrix[matrix of math nodes,nodes={
       minimum size=1.2ex,text width=1.2ex,
       text height=1.2ex,inner sep=3pt,draw={gray!20},align=center,
       anchor=base
     }, row sep=1pt,column sep=1pt
  ] (config) {
{\color{black!60}\symb{1}}\&{\color{black!60}\symb{0}}\\
{\color{black!60}\symb{2}}\&{\color{black!60}\symb{2}}\\
{\color{black!60}\symb{1}}\&{\color{black!60}\symb{1}}\\
};
\end{tikzpicture},&\begin{tikzpicture}
[baseline=-\the\dimexpr\fontdimen22\textfont2\relax,ampersand replacement=\&]
  \matrix[matrix of math nodes,nodes={
       minimum size=1.2ex,text width=1.2ex,
       text height=1.2ex,inner sep=3pt,draw={gray!20},align=center,
       anchor=base
     }, row sep=1pt,column sep=1pt
  ] (config) {
{\color{black!60}\symb{1}}\&{\color{black!60}\symb{1}}\\
{\color{black!60}\symb{0}}\&{\color{black!60}\symb{2}}\\
{\color{black!60}\symb{2}}\&{\color{black!60}\symb{1}}\\
};
\end{tikzpicture},&\begin{tikzpicture}
[baseline=-\the\dimexpr\fontdimen22\textfont2\relax,ampersand replacement=\&]
  \matrix[matrix of math nodes,nodes={
       minimum size=1.2ex,text width=1.2ex,
       text height=1.2ex,inner sep=3pt,draw={gray!20},align=center,
       anchor=base
     }, row sep=1pt,column sep=1pt
  ] (config) {
{\color{black!60}\symb{2}}\&{\color{black!60}\symb{1}}\\
{\color{black!60}\symb{0}}\&{\color{black!60}\symb{2}}\\
{\color{black!60}\symb{2}}\&{\color{black!60}\symb{1}}\\
};
\end{tikzpicture},&\begin{tikzpicture}
[baseline=-\the\dimexpr\fontdimen22\textfont2\relax,ampersand replacement=\&]
  \matrix[matrix of math nodes,nodes={
       minimum size=1.2ex,text width=1.2ex,
       text height=1.2ex,inner sep=3pt,draw={gray!20},align=center,
       anchor=base
     }, row sep=1pt,column sep=1pt
  ] (config) {
{\color{black!60}\symb{2}}\&{\color{black!60}\symb{1}}\\
{\color{black!60}\symb{1}}\&{\color{black!60}\symb{0}}\\
{\color{black!60}\symb{0}}\&{\color{black!60}\symb{2}}\\
};
\end{tikzpicture},&\begin{tikzpicture}
[baseline=-\the\dimexpr\fontdimen22\textfont2\relax,ampersand replacement=\&]
  \matrix[matrix of math nodes,nodes={
       minimum size=1.2ex,text width=1.2ex,
       text height=1.2ex,inner sep=3pt,draw={gray!20},align=center,
       anchor=base
     }, row sep=1pt,column sep=1pt
  ] (config) {
{\color{black!60}\symb{2}}\&{\color{black!60}\symb{1}}\\
{\color{black!60}\symb{1}}\&{\color{black!60}\symb{0}}\\
{\color{black!60}\symb{2}}\&{\color{black!60}\symb{2}}\\
};
\end{tikzpicture},&\begin{tikzpicture}
[baseline=-\the\dimexpr\fontdimen22\textfont2\relax,ampersand replacement=\&]
  \matrix[matrix of math nodes,nodes={
       minimum size=1.2ex,text width=1.2ex,
       text height=1.2ex,inner sep=3pt,draw={gray!20},align=center,
       anchor=base
     }, row sep=1pt,column sep=1pt
  ] (config) {
{\color{black!60}\symb{2}}\&{\color{black!60}\symb{2}}\\
{\color{black!60}\symb{1}}\&{\color{black!60}\symb{0}}\\
{\color{black!60}\symb{0}}\&{\color{black!60}\symb{2}}\\
};
\end{tikzpicture},&\begin{tikzpicture}
[baseline=-\the\dimexpr\fontdimen22\textfont2\relax,ampersand replacement=\&]
  \matrix[matrix of math nodes,nodes={
       minimum size=1.2ex,text width=1.2ex,
       text height=1.2ex,inner sep=3pt,draw={gray!20},align=center,
       anchor=base
     }, row sep=1pt,column sep=1pt
  ] (config) {
{\color{black!60}\symb{2}}\&{\color{black!60}\symb{2}}\\
{\color{black!60}\symb{1}}\&{\color{black!60}\symb{1}}\\
{\color{black!60}\symb{0}}\&{\color{black!60}\symb{2}}\\
};
\end{tikzpicture}\\
\end{array}
\right\}
	\]
	The patterns of shape $(3,2)$ that we see in $c_\thevector$ and
	$c'_\thevector$ are
	\[
	\arraycolsep=1.5pt
	\left\{
\begin{array}
{cccccc}
\begin{tikzpicture}
[baseline=-\the\dimexpr\fontdimen22\textfont2\relax,ampersand replacement=\&]
  \matrix[matrix of math nodes,nodes={
       minimum size=1.2ex,text width=1.2ex,
       text height=1.2ex,inner sep=3pt,draw={gray!20},align=center,
       anchor=base
     }, row sep=1pt,column sep=1pt
  ] (config) {
{\color{black!60}\symb{0}}\&{\color{black!60}\symb{2}}\&{\color{black!60}\symb{1}}\\
{\color{black!60}\symb{2}}\&{\color{black!60}\symb{1}}\&{\color{black!60}\symb{0}}\\
};
\end{tikzpicture},&\begin{tikzpicture}
[baseline=-\the\dimexpr\fontdimen22\textfont2\relax,ampersand replacement=\&]
  \matrix[matrix of math nodes,nodes={
       minimum size=1.2ex,text width=1.2ex,
       text height=1.2ex,inner sep=3pt,draw={gray!20},align=center,
       anchor=base
     }, row sep=1pt,column sep=1pt
  ] (config) {
{\color{black!60}\symb{0}}\&{\color{black!60}\symb{2}}\&{\color{black!60}\symb{2}}\\
{\color{black!60}\symb{2}}\&{\color{black!60}\symb{1}}\&{\color{black!60}\symb{0}}\\
};
\end{tikzpicture},&\begin{tikzpicture}
[baseline=-\the\dimexpr\fontdimen22\textfont2\relax,ampersand replacement=\&]
  \matrix[matrix of math nodes,nodes={
       minimum size=1.2ex,text width=1.2ex,
       text height=1.2ex,inner sep=3pt,draw={gray!20},align=center,
       anchor=base
     }, row sep=1pt,column sep=1pt
  ] (config) {
{\color{black!60}\symb{0}}\&{\color{black!60}\symb{2}}\&{\color{black!60}\symb{2}}\\
{\color{black!60}\symb{2}}\&{\color{black!60}\symb{1}}\&{\color{black!60}\symb{1}}\\
};
\end{tikzpicture},&\begin{tikzpicture}
[baseline=-\the\dimexpr\fontdimen22\textfont2\relax,ampersand replacement=\&]
  \matrix[matrix of math nodes,nodes={
       minimum size=1.2ex,text width=1.2ex,
       text height=1.2ex,inner sep=3pt,draw={gray!20},align=center,
       anchor=base
     }, row sep=1pt,column sep=1pt
  ] (config) {
{\color{black!60}\symb{1}}\&{\color{black!60}\symb{0}}\&{\color{black!60}\symb{2}}\\
{\color{black!60}\symb{0}}\&{\color{black!60}\symb{2}}\&{\color{black!60}\symb{1}}\\
};
\end{tikzpicture},&\begin{tikzpicture}
[baseline=-\the\dimexpr\fontdimen22\textfont2\relax,ampersand replacement=\&]
  \matrix[matrix of math nodes,nodes={
       minimum size=1.2ex,text width=1.2ex,
       text height=1.2ex,inner sep=3pt,draw={gray!20},align=center,
       anchor=base
     }, row sep=1pt,column sep=1pt
  ] (config) {
{\color{black!60}\symb{1}}\&{\color{black!60}\symb{0}}\&{\color{black!60}\symb{2}}\\
{\color{black!60}\symb{2}}\&{\color{black!60}\symb{2}}\&{\color{black!60}\symb{1}}\\
};
\end{tikzpicture},&\begin{tikzpicture}
[baseline=-\the\dimexpr\fontdimen22\textfont2\relax,ampersand replacement=\&]
  \matrix[matrix of math nodes,nodes={
       minimum size=1.2ex,text width=1.2ex,
       text height=1.2ex,inner sep=3pt,draw={gray!20},align=center,
       anchor=base
     }, row sep=1pt,column sep=1pt
  ] (config) {
{\color{black!60}\symb{1}}\&{\color{black!60}\symb{1}}\&{\color{black!60}\symb{0}}\\
{\color{black!60}\symb{0}}\&{\color{black!60}\symb{2}}\&{\color{black!60}\symb{2}}\\
};
\end{tikzpicture}\\
\begin{tikzpicture}
[baseline=-\the\dimexpr\fontdimen22\textfont2\relax,ampersand replacement=\&]
  \matrix[matrix of math nodes,nodes={
       minimum size=1.2ex,text width=1.2ex,
       text height=1.2ex,inner sep=3pt,draw={gray!20},align=center,
       anchor=base
     }, row sep=1pt,column sep=1pt
  ] (config) {
{\color{black!60}\symb{2}}\&{\color{black!60}\symb{1}}\&{\color{black!60}\symb{0}}\\
{\color{black!60}\symb{0}}\&{\color{black!60}\symb{2}}\&{\color{black!60}\symb{2}}\\
};
\end{tikzpicture},&\begin{tikzpicture}
[baseline=-\the\dimexpr\fontdimen22\textfont2\relax,ampersand replacement=\&]
  \matrix[matrix of math nodes,nodes={
       minimum size=1.2ex,text width=1.2ex,
       text height=1.2ex,inner sep=3pt,draw={gray!20},align=center,
       anchor=base
     }, row sep=1pt,column sep=1pt
  ] (config) {
{\color{black!60}\symb{2}}\&{\color{black!60}\symb{1}}\&{\color{black!60}\symb{0}}\\
{\color{black!60}\symb{1}}\&{\color{black!60}\symb{0}}\&{\color{black!60}\symb{2}}\\
};
\end{tikzpicture},&\begin{tikzpicture}
[baseline=-\the\dimexpr\fontdimen22\textfont2\relax,ampersand replacement=\&]
  \matrix[matrix of math nodes,nodes={
       minimum size=1.2ex,text width=1.2ex,
       text height=1.2ex,inner sep=3pt,draw={gray!20},align=center,
       anchor=base
     }, row sep=1pt,column sep=1pt
  ] (config) {
{\color{black!60}\symb{2}}\&{\color{black!60}\symb{1}}\&{\color{black!60}\symb{1}}\\
{\color{black!60}\symb{1}}\&{\color{black!60}\symb{0}}\&{\color{black!60}\symb{2}}\\
};
\end{tikzpicture},&\begin{tikzpicture}
[baseline=-\the\dimexpr\fontdimen22\textfont2\relax,ampersand replacement=\&]
  \matrix[matrix of math nodes,nodes={
       minimum size=1.2ex,text width=1.2ex,
       text height=1.2ex,inner sep=3pt,draw={gray!20},align=center,
       anchor=base
     }, row sep=1pt,column sep=1pt
  ] (config) {
{\color{black!60}\symb{2}}\&{\color{black!60}\symb{2}}\&{\color{black!60}\symb{1}}\\
{\color{black!60}\symb{1}}\&{\color{black!60}\symb{0}}\&{\color{black!60}\symb{2}}\\
};
\end{tikzpicture},&\begin{tikzpicture}
[baseline=-\the\dimexpr\fontdimen22\textfont2\relax,ampersand replacement=\&]
  \matrix[matrix of math nodes,nodes={
       minimum size=1.2ex,text width=1.2ex,
       text height=1.2ex,inner sep=3pt,draw={gray!20},align=center,
       anchor=base
     }, row sep=1pt,column sep=1pt
  ] (config) {
{\color{black!60}\symb{2}}\&{\color{black!60}\symb{2}}\&{\color{black!60}\symb{1}}\\
{\color{black!60}\symb{1}}\&{\color{black!60}\symb{1}}\&{\color{black!60}\symb{0}}\\
};
\end{tikzpicture}\\
\end{array}
\right\}
	\]
	We may check on Figure~\ref{fig:2d-sturmian-config-pair} that
	each of these patterns has exactly one occurrence intersecting the difference set. This is the main tool that allows
    us to show that $d$-dimensional characteristic Sturmian configurations are
    indistinguishable.
\end{example}

\subsection{Characteristic $d$-dimensional Sturmian configurations and the flip condition}

\begin{lemma}\label{lem:sturmian_is_unirec}
	For any $\balpha \in[0,1)^d$ and $\rho \in \RR$, the configurations $s_{\balpha,\rho}$ and $s'_{\balpha,\rho}$ are uniformly recurrent.
\end{lemma}

\begin{proof}
	If all coordinates in $\balpha = (\alpha_1,\dots,\alpha_d)$ are rational, it is clear that $s_{\balpha,\rho}$ and $s'_{\balpha,\rho}$ have finite orbits under the shift action and are thus uniformly recurrent. 
	
    Suppose there is $1 \leq i \leq d$ such that $\alpha_i$ is irrational and let $S \subset \ZZd$ be finite and $p \in \Lcal_S(s_{\balpha,\rho})$. From the definition we have that $\sigma^n(s_{\balpha,\rho}) \in [p]$ if and only if \[ n\cdot \alpha+\rho \in \bigcap_{k \in S} (W_{p_k} - \alpha \cdot k + \ZZ).\]
    From the definition, it is easy to see that for each $\symb{j}\in \alfad$, the set $W_\symb{j}$ is either empty or has nonempty interior. As $p \in \Lcal_S(s_{\balpha,\rho})$, it follows that $\bigcap_{k \in S} (\operatorname{Int} (W_{p_k} )- \alpha \cdot k + \ZZ)$ is nonempty and thus contains an open interval $U \subseteq \RR/\ZZ$.
	
	As $\alpha_i$ is irrational, it follows that there is $M \in \NN$ such that for any $b \in \RR/\ZZ$ there is $0 \leq m \leq M$ such that $b+m\alpha_i \in U$. It follows that for any $n \in \ZZd$, there is $0 \leq m \leq M$ such that $(n + me_i)\cdot \alpha + \rho \in U$, and therefore $\sigma^{n+me_i}(s_{\alpha,\rho})\in [p]$. This shows that $s_{\alpha,\rho}$ is uniformly recurrent. The argument for $s'_{\alpha,\rho}$ is analogous. \end{proof}

\begin{lemma}\label{lem:sturmian-is-asymptotic}
    If $\balpha$ is totally irrational, then
    $\big(c_{\balpha},c'_{\balpha}\big)$ is an asymptotic pair
    whose difference set is
    $F = \{\bzero, -e_1, \dots, -e_d  \}$.
\end{lemma}

\begin{proof}
    Since $\balpha$ is totally irrational, we have that
$\alpha_i+\bn\cdot\balpha$
    is an integer if and only if $\bn=-\be_i$
    and
$\bn\cdot\balpha$
    is an integer if and only if $\bn=\bzero$.
    Therefore, we have that
    \[
    \lfloor\alpha_i+\bn\cdot\balpha\rfloor-\lfloor\bn\cdot\balpha\rfloor
        =
    \lceil\alpha_i+\bn\cdot\balpha\rceil-\lceil\bn\cdot\balpha\rceil
    \]
    for every $\bn\in\ZZ^d\setminus\{\bzero,-\be_i\}$
    and $i\in\{1,\dots,d\}$.
    Therefore, 
    \[
        c_{\balpha}(\bn)= c'_{\balpha}(\bn)
    \]
    for every $\bn\in\ZZ^d\setminus\{\bzero,-\be_0,\dots,-\be_d\}$.
    This shows that
        $(c_{\balpha},c'_{\balpha})$ is an asymptotic pair
        whose difference set is $F=\{\bzero,-\be_0,\dots,-\be_d\}$.
\end{proof}

\begin{proposition}\label{prop:sturmian_is_flip}
	Let $\alpha \in [0,1)^d$ be totally irrational. The characteristic
	$d$-dimensional Sturmian configurations $c_{\balpha}$ and $c'_{\balpha}$
	satisfy the flip condition. 
\end{proposition}

\begin{proof}
	From~\Cref{lem:sturmian-is-asymptotic},
	if $\balpha=(\alpha_1,\dots,\alpha_d)\in[0,1)^d$ 
	is totally irrational,
	then $(c_{\balpha},c'_{\balpha})$
	is an asymptotic pair
	whose difference set is $F = \{\bzero, -e_1, \dots, -e_d  \}$. 
	
From \Cref{lem:sturmian-formual-from-cut-and-project},
	we have that for $\bn \in \ZZ^d$,
	\[ (c_{\alpha})_\bn = \sum_{\symb{i}=1}^d \left(\lfloor\alpha_\symb{i}+ \bn \cdot \balpha  \rfloor -\lfloor \bn \cdot \balpha  \rfloor\right)  \mbox{ and } (c'_{\alpha})_\bn =  \sum_{\symb{i}=1}^d \left(\lceil\alpha_\symb{i}+ \bn \cdot \balpha  \rceil -\lceil \bn \cdot \balpha  \rceil\right). \]
	From here we obtain directly that $(c_{\alpha})_\bzero = \symb{0}$ and $(c'_{\alpha})_\bzero = d$. For $\bn = -\be_{\symb{i}}$ we get
	\begin{align}
		\label{eq:c-alpha-ei}
		(c_{\alpha})_{-\be_{\symb{i}}} &= \sum_{\symb{j}=1}^d \left(\lfloor\alpha_\symb{j} -\alpha_{\symb{i}}  \rfloor -\lfloor -\alpha_{\symb{i}}  \rfloor\right) = d - \#\{ \symb{j} : \alpha_{\symb{j}} < \alpha_{\symb{i}}   \} = \#\{ \symb{j} : \alpha_{\symb{j}} \geq \alpha_{\symb{i}}   \}, \\
		\label{eq:c'-alpha-ei}
		(c'_{\alpha})_{-\be_{\symb{i}}} &= \sum_{\symb{j}=1}^d \left(\lceil\alpha_\symb{j} -\alpha_{\symb{i}}  \rceil -\lceil -\alpha_{\symb{i}}  \rceil\right) = \#\{ \symb{j} : \alpha_{\symb{j}} > \alpha_{\symb{i}}   \}.
	\end{align}
	
	As $\alpha$ is totally irrational, all $\alpha_{\symb{i}}$ are distinct non-zero values. From the above formula we obtain that $(c_{\alpha})|_F$ and $(c'_{\alpha})|_F$ are bijections onto $\alfad$ and $(c_{\alpha})_{-\be_{\symb{i}}}-(c'_{\alpha})_{-\be_{\symb{i}}} = \symb{1}$, from where the second condition follows.
\end{proof}

\subsection{Characteristic $d$-dimensional Sturmian configurations are indistinguishable}

For every $\symb{i}$ with $\symb{0}\leq \symb{i}\leq d$,
the set $W_\symb{i}$ is a left-closed right-open interval
sharing the same end-points as $W'_\symb{i}$
which is a right-closed left-open interval.
We have that
$\mathcal{P}=\{W_{\symb{i}}\}_{\symb{0}\leq\symb{i}\leq d}$ 
and
$\mathcal{P}'=\{W'_{\symb{i}}\}_{\symb{0}\leq\symb{i}\leq d}$ 
are two partitions of the circle $\RR/\ZZ$ 
illustrated in Figure~\ref{fig:d-dim-coding-unit-interval}.
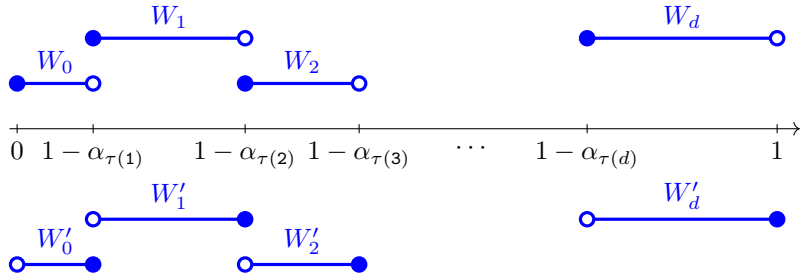
\begin{figure}[ht]
\begin{center}
    \begin{tikzpicture}[xscale=10,yscale=6]
\tikzstyle{vide}=[circle,draw=blue, fill=white, inner sep=0pt, minimum size=5pt]
\tikzstyle{plein} =[circle,draw=blue, fill=blue, inner sep=0pt, minimum size=5pt]
    \def\t{.01}
    \def\aI{.9}
    \def\aII{.7}
    \def\aIII{.55}
    \def\ad{.25}
    \def\theframe{
    \draw[->] (-\t,0) -- (1+3*\t,0);
    \draw (1-\aI,-\t) node[below] {$1-\alpha_{\tau(\symb{1})}$} -- (1-\aI,+\t);
    \draw (1-\aII,-\t) node[below] {$1-\alpha_{\tau(\symb{2})}$} -- (1-\aII,+\t);
    \draw (1-\aIII,-\t) node[below] {$1-\alpha_{\tau(\symb{3})}$} -- (1-\aIII,+\t);
    \draw (1-\ad,-\t) node[below] {$1-\alpha_{\tau( d)}$} -- (1-\ad,+\t);
    \draw (0,-\t) node[below] {$0$} -- (0,+\t);
    \draw (1,-\t) node[below] {$1$} -- (1,+\t);
    \node[below] at (.60,-\t) {$\cdots$};
    } %
    \begin{scope}[yshift=4mm]
        \theframe
        \draw[very thick,blue] (0     ,.1) node[plein]{} -- node[above]{$W_0$} (1-\aI  ,.1) node[vide]{};
        \draw[very thick,blue] (1-\aI ,.2) node[plein]{} -- node[above]{$W_1$} (1-\aII ,.2) node[vide]{};
        \draw[very thick,blue] (1-\aII,.1) node[plein]{} -- node[above]{$W_2$} (1-\aIII,.1) node[vide]{};
        \draw[very thick,blue] (1-\ad ,.2) node[plein]{} -- node[above]{$W_d$} (1      ,.2) node[vide]{};
    \end{scope}
    \begin{scope}
        \draw[very thick,blue] (0     ,.1) node[vide]{} -- node[above]{$W'_0$} (1-\aI  ,.1) node[plein]{};
        \draw[very thick,blue] (1-\aI ,.2) node[vide]{} -- node[above]{$W'_1$} (1-\aII ,.2) node[plein]{};
        \draw[very thick,blue] (1-\aII,.1) node[vide]{} -- node[above]{$W'_2$} (1-\aIII,.1) node[plein]{};
        \draw[very thick,blue] (1-\ad ,.2) node[vide]{} -- node[above]{$W'_d$} (1      ,.2) node[plein]{};
    \end{scope}
\end{tikzpicture}
\end{center}
    \caption{Define $\alpha_{\symb{0}} = 1$ and $\alpha_{d+\symb{1}}=0$ and let $\tau$ be the permutation 
    of $\{\symb{1},\dots,d\}\cup \{\symb{0},d+1\}$ which fixes $\{\symb{0},d+1\}$ and
    such that $0<\alpha_{\tau(d)}<\dots<\alpha_{\tau(\symb{1})}<1$. 
    The intervals $W_i=[1-\alpha_{\tau(\symb{i})}, 1-\alpha_{\tau(\symb{i+1})})$
    form a partition of the circle $\RR/\ZZ$
    and similarly for
    the intervals $W'_i=(1-\alpha_{\tau(\symb{i})}, 1-\alpha_{\tau(\symb{i+1})}]$.
    }
    \label{fig:d-dim-coding-unit-interval}
\end{figure}

Let $S\subset \ZZd$ be finite and $p\colon S\to\alfad$ be a pattern with support $S$.
Let
    \begin{equation}\label{eq:Ip-I'p}
        I_p =  \bigcap_{\bn\in S} \left(W_{p(n)} - \balpha\cdot\bn\right)
        \qquad
        \text{ and }
        \qquad
        I'_p = \bigcap_{\bn\in S} \left(W'_{p(n)} - \balpha\cdot\bn\right).
    \end{equation}
    
And let $\mathcal{P}^S = \{ I_p\}_{p \in \alfad^S}$ and $(\mathcal{P}')^S = \{ I'_p\}_{p \in \alfad^S}$ be the partitions of $\RR/\ZZ$ determined by the support $S$. Notice that the sets $I_p$ and $I'_p$ have the same interior (which is nonempty if and only if these sets are nonempty) and thus differ only on their boundary points.

The pattern $p$ appears in $c_{\balpha}$
if and only if $\operatorname{Int}(I_p)\neq\varnothing$. 
Similarly, the pattern $p$ appears in $c'_{\balpha}$
if and only if $\operatorname{Int}(I'_p)\neq\varnothing$. 
As $\operatorname{Int}(I_p) = \operatorname{Int}(I'_{p})$, we obtain that $p$ appears in $c_{\balpha}$
if and only if $p$ appears in $c'_{\balpha}$.
Therefore, the configurations  $c_{\balpha}$ and $c'_{\balpha}$ share the same
language, that is: $\Lcal_S(c_{\balpha}) = \Lcal_S(c'_{\balpha})$ for every finite $S \subset \ZZd$.
Also, $\#\Lcal_S(c_{\balpha})$ is
equal to the number of non-empty sets $I_p$ for $p \in \alfad^S$. 

\begin{lemma}\label{lem:partition_into_intervals}
	For every nonempty connected finite set $S\subset \ZZd$ and pattern $p \in \alfad^{S}$, the sets $I_{p},I'_{p}$ are either empty or intervals in $\RR/\ZZ$. 
\end{lemma}

\begin{proof}
	We only prove this for $I_{p}$, the argument for $I'_{p}$ follows from the considerations stated above. Let us notice that the intersection of two left-closed, right-open intervals on the circle is either empty, a left-closed and right-open interval, or a disconnected union of them. This third case can only occur when the sum of the lengths of both intervals exceeds $1$.
	
	Let us now prove the lemma by induction. If $S =\{n\}$ is a singleton, then the result is direct: if $p(n)=\symb{i}$, we have that $I_{p} = I_{\symb{i}}-\alpha \cdot n$, which is clearly a non-empty interval on the circle.
	
	Now let $S$ be a nonempty connected finite set, $p \in \alfad^{S}$ and suppose the result holds for every strict nonempty connected subset of $S$. As $S$ is finite, we can find $n \in S$ such that $S' = S \setminus \{n\}$ is also connected by removing a leaf in some spanning tree of $S$. Let $p'$ be the restriction of $p$ to $S'$, then we have \[ I_p = I_{p'} \cap (I_{p(n)}-\alpha \cdot n).  \]
	By the inductive hypothesis $I_{p'}$ is an interval. We also have that $(I_{p(n)}-\alpha \cdot n)$ is an interval. Therefore the only case where $I_p$ might not be an interval is when the sum of the lengths of $I_{p'}$ and $I_{p(n)}$ is strictly larger than $1$.
	
	As $S$ is connected, there is $1 \leq j \leq d$ such that $n - e_j \in S'$ or $n+e_j \in S'$. Let us proceed in the case where $n - e_j \in S'$, the other case is analogous. We have that $I_{p'} \subset I_{p(n-e_j)} - \alpha \cdot (n-e_j)$. Notice that if $p(n-e_j)\neq p(n)$, then the sum of the lengths of $I_{p'}$ and $I_{p(n)}$ is at most $1$, hence the only issue can arise when $p(n-e_j)= p(n)$ (and $I_{p(n)}$ has length larger than $\frac{1}{2}$).

	Suppose it is the case and let $\symb{i} = p(n-e_j)= p(n)$. Let $\pi$ be the permutation of $\{\symb{1},\dots,d\} \cup \{\symb{0},d+1\}$ which fixes $\symb{0}$ and $d+1$ and such that \[0= \alpha_{\pi(d+\symb{1})}< \alpha_{\pi(d)} < \dots < \alpha_{\pi(\symb{1})} < \alpha_{\pi(\symb{0})}=1.\] With this $I_{\symb{i}} = [1-\alpha_{\pi(\symb{i})}, 1-\alpha_{\pi(\symb{i+1})})$, and hence we have \[ I_{p'} \subset I_{\symb{i}} - \alpha \cdot (n-e_j) = [1-\alpha_{\pi(\symb{i})}+\alpha_{\symb{j}}, 1-\alpha_{\pi(\symb{i+1})}+\alpha_{\symb{j}}) - \alpha \cdot n\]
	\[ I_{p(n)} - \alpha \cdot n = [1-\alpha_{\pi(\symb{i})}, 1-\alpha_{\pi(\symb{i+1})}) - \alpha \cdot n.   \]

	There are two cases to consider:
	 \begin{enumerate}
	 	\item If $\alpha_j \leq \alpha_{\pi(\symb{i+1})}$, we have that $I_{p'} + \alpha\cdot n \subset  (I_{\symb{i}}+\alpha_j) \subset [0,1)$. It follows that $I_p + \alpha \cdot n = (I_{p'} + \alpha\cdot n) \cap I_{p(n)}$ is either empty or an interval in $\RR/\ZZ$ and therefore so is $I_{p}$.
	 	\item If $\alpha_j \geq \alpha_{\pi(\symb{i})}$, we have $I_{p'} + \alpha\cdot n \subset (I_{p}+\alpha_j) \subset [1,2)$. It follows that $I_p + \alpha \cdot n = (I_{p'} + \alpha\cdot n) \cap I_{p(n)}$ is either empty or an interval in $\RR/\ZZ$ and therefore so is $I_{p}$.
	 \end{enumerate}
 We conclude that in both of the problematic cases, $I_{p}$ is either empty or an interval in $\RR/\ZZ$.\end{proof}

\begin{lemma}\label{lem:ddim-1-occ-if-interval}
	Let $\balpha=(\alpha_1,\dots,\alpha_d)\in[0,1)^d$ 
    be totally irrational and $S$ be a nonempty finite connected subset of
    $\ZZd$. For every $p \in \Lcal_S(c_{\balpha})=\Lcal_S(c'_{\balpha})$, the sets
    $\occ_p(c_{\balpha})\cap (F-S)$ and
    $\occ_p(c'_{\balpha})\cap (F-S)$ are singletons.
\end{lemma}

\begin{proof}
	The partition $\mathcal{P}=\{I_{\symb{i}}\}_{\symb{0}\leq \symb{i} \leq d}$ of $\RR/\ZZ$ is a partition into $d+1$ intervals corresponding to the $d+1$ symbols in the alphabet. The boundary points of the intervals $I_{\symb{i}} \in \mathcal{P}$ are
	\[
	F\cdot\balpha = \{0, 1-\alpha_1, \dots, 1-\alpha_d\}.
	\]
	 Notice that $\mathcal{P}^S = \{I_p : p\in\Lcal_S(c_{\alpha})\} = \{I_p : p\in\alfad^S \text{ and } I_p\neq\varnothing\}$.  Using~\Cref{lem:partition_into_intervals}, we obtain that $\mathcal{P}^S$ is a partition of $\RR/\ZZ$ into nonempty (left-closed, right-open) intervals. It is therefore clear from the definition of the intervals $I_p$ that their unique boundary points are described by the set $F\cdot\balpha-S\cdot\balpha = (F-S)\cdot\balpha$. 
	
	For each $p$, there exists 
	a unique boundary point $\xi \in(F-S)\cdot\balpha$
	which belongs to $I_p$ (the left-end point of $I_p$).
	Since $\balpha$ is totally irrational, 
	the map $\bn\mapsto\bn\cdot\balpha+\ZZ$ is injective,
	thus
	there is a unique vector $\bn\in F-S$
	such that $\bn\cdot\balpha=\xi$. We have that $\sigma^{\bn}(c_{\balpha}) = s_{\alpha,\xi} \in [p]$ and so $\bn \in \occ_p(c_{\balpha})$.
	
	The argument for $c'_{\balpha}$ is identical, the only difference being that the unique boundary point is now the right-end point of $I'_p$.
\end{proof}

\begin{theorem}\label{thm:ddim-sturmian-is-indistinguishable}
    If $\balpha=(\alpha_1,\dots,\alpha_d)\in[0,1)^d$ 
is totally irrational,
then $(c_{\balpha},c'_{\balpha})$
is a non-trivial indistinguishable asymptotic pair which satisfies the flip condition.
\end{theorem}

\begin{proof}
By~\Cref{lem:sturmian-is-asymptotic}, we have that $(c_{\balpha},c'_{\balpha})$ is a non-trivial asymptotic pair whose differences set is $F=\{\bzero,-\be_1,\dots,-\be_d\}$. Furthermore, by~\Cref{prop:sturmian_is_flip}, it satisfies the flip condition. Let $S$ be a nonempty connected finite subset of $\ZZd$ and $p \in \alfad^S$. From Lemma~\ref{lem:ddim-1-occ-if-interval},
we obtain that the set of occurrences of $p$ intersects $F-S$ exactly once for both $c_{\balpha}$ and $c'_{\balpha}$, that is
\[ \#( \occ_p(c_{\balpha}) \cap (F-S) ) = 1 = \#( \occ_p(c'_{\balpha}) \cap (F-S) ).  \]
By~\Cref{prop:trivialite} it suffices to check the above condition for patterns
whose support 
is a nonempty finite connected subset of $\ZZd$.
We conclude that $(c_{\balpha},c'_{\balpha})$ is indistinguishable.
\end{proof}

\begin{remark}
	If we take a sequence of totally irrational vectors $(\alpha_n)_{n \in \NN}$ it follows that each associated pair $(c_{\alpha_n},c'_{\alpha_n})$ satisfies the flip condition. It follows that if both $(c_{\alpha_n})_{n \in \NN}$ and $(c'_{\alpha_n})_{n \in \NN}$ converge to $c$ and $c'$ in the prodiscrete topology, then $(c_{\alpha_n},c'_{\alpha_n})$ converges in the asymptotic relation to the \'etale limit $(c,c')$ which thus also satisfies the flip condition. By~\Cref{prop:limit-of-indist-is-indist} we get that $(c,c')$ is therefore an indistinguishable asymptotic pair. This can be used to provide examples of indistinguishable asymptotic pairs which satisfy the flip condition but that are not uniformly recurrent. See~\Cref{fig:etale}.
\end{remark}

\section{Uniformly recurrent indistinguishable asymptotic pairs are Sturmian}\label{sec:proof_of_thm_A}

The goal of this section is to prove \Cref{thm:multidim_sturmian_characterization}.
We already proved in
    \Cref{thm:ddim-sturmian-is-indistinguishable}
    that if $\balpha=(\alpha_1,\dots,\alpha_d)\in[0,1)^d$ 
is totally irrational,
then $(c_{\balpha},c'_{\balpha})$
is a non-trivial indistinguishable asymptotic pair which satisfies the flip condition.
Thus, it remains to show the existence of a totally irrational vector
$\balpha=(\alpha_1,\dots,\alpha_d)\in[0,1)^d$ describing an
indistinguishable asymptotic pair which satisfies the flip condition whenever the configurations are uniformly recurrent. The proof relies on an induction argument on the dimension of $\ZZd$ and on the existence of a factor map
between the symbolic dynamical system generated by a multidimensional Sturmian configuration
and rotations on the circle $\RR/\ZZ$.

\subsection{Symbolic representations}

Consider $\dynsys{\RR/\ZZ}{\ZZd}{R}$ a continuous $\ZZd$-action on $\RR/\ZZ$ where
$R\colon \ZZd\times\RR/\ZZ\to\RR/\ZZ$.
For some finite set $\A$,
a \define{topological partition} of $\RR/\ZZ$ (in the sense of Definition 6.5.3 of~\cite{MR1369092}) is a
collection $\{P_a\}_{a\in\A}$ of disjoint open sets $P_a\subset\RR/\ZZ$
such that $\RR/\ZZ = \bigcup_{a\in\A} \overline{P_a}$. If $S\subset\ZZd$ is a finite set,
we say that a pattern $w\in\A^S$
is \define{allowed} for $\Pcal,R$ if
\begin{equation}\label{eq:allowed-if-nonempty}
	\bigcap_{\bk\in S} R^{-\bk}(P_{w_\bk}) \neq \varnothing.
\end{equation}
The intersection in \Cref{eq:allowed-if-nonempty}
is related to the definition of $I'_w$ and $I_w$ done in \Cref{eq:Ip-I'p}
except here the sets $P_{w_\bk}$ are open.

Let us recall that a $\ZZd$-\define{subshift} is a set of the form $X \subset \A^{\ZZd}$ which is closed in the prodiscrete topology and invariant under the shift action; and its language is the union of $\Lcal(x)$ for every $x \in X$. Let $\Lcal_{\Pcal,R}$ be the collection of all allowed patterns for $\Pcal,R$.
The set $\Lcal_{\Pcal,R}$ is the language of a subshift 
$\Xcal_{\Pcal,R}\subseteq\A^{\ZZd}$ defined as follows,
see \cite[Prop.~9.2.4]{MR3525488},
\[
\Xcal_{\Pcal,R} = 
\{x\in\A^{\ZZd} \mid \sigma^\bn(x)|_S \in\Lcal_{\Pcal,R}
\text{ for every } \bn\in\ZZd \text{ and finite subset } S\subset\ZZd\}.
\]
We call $\Xcal_{\Pcal,R}$ the \define{symbolic extension} of $\dynsys{\RR/\ZZ}{\ZZd}{R}$ determined by $\Pcal$.

For each $x\in\Xcal_{\Pcal,R}$ and $m\geq 0$ there is a corresponding nonempty open set
\[
D_m(x) = \bigcap_{\Vert\bk\Vert_{\infty} \leq m} R^{-\bk}(P_{x_\bk}) \subset \RR/\ZZ.
\]
The sequence of compact closures $(\overline{D}_m(x))_{m \in \NN}$ of these sets is nested and thus it follows that their intersection is nonempty. Notice that there is no reason why $\operatorname{diam}(\overline{D}_m(x))$ should converge to zero, and thus the intersection could contain more than one point. In order for $\Xcal_{\Pcal,R}$ to capture the dynamics of $\dynsys{\RR/\ZZ}{\ZZd}{R}$, this intersection should contain only one point.
This leads to the following definition.

\begin{definition}\label{def:toppartition}
	A topological partition $\Pcal$ of $\RR/\ZZ$ gives a \define{symbolic representation} $\Xcal_{\Pcal,R}$ of $\dynsys{\RR/\ZZ}{\ZZd}{R}$
	if for every $x\in\Xcal_{\Pcal,R}$ the intersection
	$\bigcap_{m=0}^{\infty}\overline{D}_m(x)$ consists of exactly one
	point $\rho \in \RR/\ZZ$.
	We call $x$ a \define{symbolic representation of $\rho$}. 
\end{definition}

If $\Pcal$ gives a symbolic representation of the 
dynamical system 
$\dynsys{\RR/\ZZ}{\ZZd}{R}$,
then there is a well-defined map
$f\colon \Xcal_{\Pcal,R} \to \RR/\ZZ$ which maps a configuration
$x\in\Xcal_{\Pcal,R}\subset\A^{\ZZd}$ to the unique point
$f(x)\in\RR/\ZZ$ in the intersection
$\cap_{n=0}^{\infty}\overline{D}_n(w)$. It is not hard to prove that $f$ is in fact a factor map, that is, such that $f$ is continuous, surjective and $\ZZd$-equivariant ($f(\sigma^k(x)) = R^k(f(x))$ for every $k \in \ZZd$). A proof of this fact for the case $d=1$ can be found in \cite[Prop.~6.5.8]{MR1369092}.
A proof for $\ZZ^2$-actions can be found in \cite[Prop.~5.1]{labbe_markov_2021} and a proof for general group actions follows the same arguments.

Now let us turn back to circle rotations. Let $\balpha\in[0,1)^d$ and consider
the dynamical system $\dynsys{\RR/\ZZ}{\ZZd}{R}$ where
$R\colon \ZZd\times\RR/\ZZ\to\RR/\ZZ$ 
is the continuous $\ZZd$-action on $\RR/\ZZ$
defined by
\[
R^\bn(x)\isdef R(\bn,x)=x + \bn\cdot\balpha
\]
for every $\bn\in\ZZd$.

Recall that an action is minimal if every orbit is dense. The following lemma is well known, we write it down for future reference and we give a quick proof sketch.

\begin{lemma}\label{lem:ddim-sturmian-factor-map}
	Let $\balpha\in[0,1)^d$ be totally irrational
	and consider the topological partition of the circle 
    \[
        \Pcal=\{\operatorname{Int}(W_{\symb{i}})\}_{\symb{i}\in\alfad}.
    \]
	\begin{enumerate}
		\item The partition $\Pcal$ gives a symbolic
		representation of the dynamical system $\dynsys{\RR/\ZZ}{\ZZd}{R}$.
		\item The symbolic dynamical system $\Xcal_{\Pcal,R}$ is minimal and
		satisfies $\Xcal_{\Pcal,R}= \overline{\{\sigma^\bk
			c_\balpha\colon\bk\in\ZZd\}}$.
		\item $f\colon \Xcal_{\Pcal,R}\to\RR/\ZZ$ where
		$f(x) \in \bigcap_{n=0}^{\infty}\overline{D}_n(w)$ is a factor map.
	\end{enumerate}
\end{lemma}

\begin{proof}
	As $\balpha$ is totally irrational, then every component $\alpha_i$ is irrational and hence it follows that the action $\dynsys{\RR/\ZZ}{\ZZd}{R}$ is minimal. From here it follows by standard arguments that $\Pcal$ gives a symbolic representation $\Xcal_{\Pcal,R}$ of the $\dynsys{\RR/\ZZ}{\ZZd}{R}$, as every $\operatorname{Int}(W_{\symb{i}})$ is invariant only under the trivial rotation (e.g., see~\cite[Lemma
	3.4]{labbe_markov_2021}). The second statement follows easily from the definitions of $\Xcal_{\Pcal,R}$ and $c_{\alpha}$, and the third statement follows from the discussion below~\Cref{def:toppartition}.\end{proof}

\subsection{Ordered flip condition}

In order to simplify the proofs in this section, we consider a particular case of the flip condition in which the values of $x|_F$ and $y|_F$ are fixed.

\begin{definition}
    Let $d\geq1$ be an integer.
	An indistinguishable asymptotic pair $x,y\in\alfad^{\ZZd}$ 
    satisfies the \define{ordered flip condition} if:
	\begin{enumerate}
        \item the difference set of $x$ and $y$ is $F=\{\bzero,
            -\be_1,\dots,-\be_d\}$,
		\item $x_0 = \symb{0}$ and $x_{-\be_i} = \symb{i}$ for all $\symb{1} \leq \symb{i} \leq d$,
        \item $y_0 = d$ and $y_{-\be_{i}} = \symb{i}-\symb{1}$ for all $\symb{1} \leq \symb{i} \leq d$.
	\end{enumerate}
\end{definition}

Observe that if
two configurations satisfy the 
ordered flip condition, they also satisfy the flip condition.
Moreover, notice that the ordered flip condition corresponds to the permutation of $F$ given by \[ \bzero \mapsto -\be_{\symb{1}} \mapsto -\be_{\symb{2}} \mapsto \dots \mapsto -\be_{d} \mapsto \bzero. \]

\begin{lemma}\label{lem:flip-reduces-to-ordered-flip}
    Let $d\geq1$ be an integer.
    Let $x,y \in\alfad^{\ZZd}$ form an indistinguishable asymptotic pair
    satisfying the flip condition. 
    Then there exists a matrix $A\in\operatorname{GL}_d(\ZZ)$ which permutes the canonical base $\{\be_1,\dots,\be_d\}$ such that $(x\circ A,y\circ A)$ is an indistinguishable asymptotic pair satisfying the \define{ordered} flip condition. 
\end{lemma}

\begin{proof}
    As $x,y$ satisfy the flip condition, then the restrictions of $x$ and $y$ to $F$ are bijections $F\to\alfad$, $x_0 = 0$ and 
    $y_\bn = x_\bn - \symb{1} \bmod (d+\symb{1})$ for every $\bn\in F$.
    Let $A \in \operatorname{GL}_d(\ZZ)$ be the permutation matrix 
    which sends $-\be_i$ to $x|_F^{-1}(\symb{i})$
    for all $i$ with $1\leq i\leq d$.
    Thus it satisfies $x(-A\be_i)=\symb{i}$.
    
    By~\Cref{prop:shifted_SI}, $x\circ A,y\circ A$ is an indistinguishable
    asymptotic pair. It is clear by definition of $A$ that their difference set
    is $F$, that $(x \circ A)_0 = x_0 = \symb{0}$ and 
    $(x \circ A)_{-\be_i} = x(-A\be_i) = \symb{i}$ 
    for all $\symb{1} \leq \symb{i} \leq d$.
    
    Finally, $(y\circ A)_0 = y_0 = x_0-\symb{1} = 0-\symb{1} = d\bmod
    (d+\symb{1})$, 
    and for $\symb{1} \leq \symb{i} \leq d$, we have 
    $(y\circ A)_{-\be_i}= y(-A\be_i) = x(-A\be_i)-\symb{1} = \symb{i-1}\bmod (d+\symb{1})$.
    Thus $(x\circ A, y\circ A)$ satisfy the ordered flip condition.
\end{proof}

It follows that if we show that every pair $x,y\in\alfad^{\ZZd}$ which satisfies the {ordered flip condition} is equal to $c_{\balpha},c'_{\balpha}$ for some totally irrational $\balpha$, we immediately obtain that every non-trivial indistinguishable asymptotic pair which satisfies the flip condition also coincides with $c_{\balpha'},c'_{\balpha'}$ for some totally irrational slope $\alpha'$ where $\alpha'$ is a permutation of $\alpha$.

\begin{proposition}\label{prop:sturmian_is_ordered_flip}
    Let $d\geq1$ be an integer.
	Let $\alpha \in [0,1)^d$ be totally irrational
such that $1>\alpha_{\symb{1}}> \alpha_{\symb{2}} > \dots >
    \alpha_{d}>0$.
    The characteristic $d$-dimensional Sturmian configurations $c_{\balpha}$
    and $c'_{\balpha}$ satisfy the ordered flip condition. 
\end{proposition}

\begin{proof}
    From \Cref{prop:sturmian_is_flip}, 
    $(c_{\balpha},c'_{\balpha})$ satisfy the flip condition. 
    Following \Cref{eq:c-alpha-ei} and \Cref{eq:c'-alpha-ei}, we get
    \begin{align*}
        (c_{\alpha})_{-\be_{\symb{i}}} 
        &= \#\{ \symb{j} : \alpha_{\symb{j}} \geq \alpha_{\symb{i}} \}
         = \symb{i}, \\
        (c'_{\alpha})_{-\be_{\symb{i}}} 
        &= \#\{ \symb{j} : \alpha_{\symb{j}} > \alpha_{\symb{i}} \}
         = \symb{i}-1.
    \end{align*}
    Thus $(c_{\balpha},c'_{\balpha})$ satisfy the ordered flip condition. 
\end{proof}
	
\subsection{Indistinguishable asymptotic pairs restricted to a $(d-1)$-dimensional submodule}

In what follows we show that indistinguishable asymptotic pairs which satisfy the ordered flip condition are Sturmian. Our strategy is to reduce the dimension of the underlying group by restricting the values of the configurations to the $(d-1)$-dimensional submodule orthogonal to $\be_1$, and then to apply a suitable projection which fuses two symbols into a single one. We show that the resulting configurations in $\ZZ^{d-1}$ also satisfy the ordered flip condition, and thus it gives us the means to prove our result inductively.

In order to develop this strategy, we introduce the following notation. Let $B=\{b_1,\dots,b_k\}\subset\ZZd$.
For each starting point $\bv\in\ZZd$, let 
\[
\begin{array}{rccl}
	\ell_{\bv,B}\colon&\ZZ^k & \to & \ZZd\\
	&n & \mapsto & \bv+ n_1b_1 +\dots+n_k b_k.
\end{array}
\]

If $x\in\Sigma^{\ZZd}$ is a configuration,
then
$x\circ\ell_{\bv,B}\in\Sigma^{\ZZ^k}$
is the $k$-dimensional configuration which occurs in $x$
starting at position $\bv\in\ZZd$
and following the directions $b_i\in B$.
Below, we use the shorter notation
$\be_1^\perp\isdef \{\be_2,\dots,\be_d\}$ to denote the canonical basis without the vector $\be_1$.

Let us consider the projection

\[
\begin{array}{rccl}
\pi:&\alfad & \to & \{\symb{0},\dots,d-1\}\\
&\symb{j} & \mapsto & 
\begin{cases}
\symb{0}  & \text{ if } \symb{j} = \symb{0},\\
\symb{j}-\symb{1}  & \text{ if } \symb{j}\neq\symb{0}.
\end{cases}
\end{array}
\]
which extends to configurations $x\in\alfad^{\ZZd}$ by letting
\begin{align*}
\pi(x)&=\left(\pi(x_{\bn})\right)_{\bn\in\ZZd}\in\{\symb{0},\dots,d-1\}^{\ZZd}.
\end{align*}

\begin{proposition}\label{prop:d-1dim-flip-condition}
    Let $d\geq2$ be an integer.
	Let $x,y \in\alfad^{\ZZd}$ be an indistinguishable asymptotic pair satisfying the ordered flip condition.
    Then $\pi\circ x\circ\ell_{0,\be_1^\perp}$ and $\pi\circ y\circ\ell_{0,\be_1^\perp}$ are indistinguishable asymptotic configurations in $\alfa{d-1}^{\ZZ^{d-1}}$ which satisfy the ordered flip condition in dimension $d-1$.
\end{proposition}

\begin{proof}
	By~\Cref{prop:invariance_sliding_block_code} we have that 
	$(\pi(x), \pi(y))$ is an indistinguishable asymptotic pair.
	Under the ordered flip condition, the difference set of $(\pi(x), \pi(y))$ is 
	$F\setminus\{-\be_1\}= \{ \bzero, -\be_{\symb{2}}, -\be_{\symb{3}}, \dots, -\be_{d}\}$ and thus it follows that for any pattern $p$ with support $S \subset \langle \be_1^{\perp}\rangle$ we have
	
	\[  \sum_{u \in (F\setminus\{-\be_1\})-S } \indicator{[p]}(\sigma^u(\pi(y)))-\indicator{[p]}(\sigma^u(\pi(x))) = 0.    \]
	It follows that the pair $(\pi\circ x\circ\ell_{0,\be_1^\perp},
	\pi\circ y\circ\ell_{0,\be_1^\perp})$ is also indistinguishable. It can be checked directly that it also satisfies the ordered flip condition.
\end{proof}

If we were also able to show that $\pi\circ x\circ\ell_{0,\be_1^\perp}$ is uniformly recurrent, then~\Cref{prop:d-1dim-flip-condition} provides a way to prove~\Cref{thm:multidim_sturmian_characterization}
by induction on the dimension. Namely, if we were to proceed by induction we would obtain that
$\pi\circ x\circ\ell_{0,\be_1^\perp}$ and 
$\pi\circ y\circ\ell_{0,\be_1^\perp}$
are $(d-1)$-dimensional characteristic Sturmian configurations
associated to a totally irrational slope
$(\alpha^{(2)},\dots,\alpha^{(d)})\in[0,1)^{d-1}$, that is,
\[
\pi\circ x\circ\ell_{0,\be_1^\perp}=c_{(\alpha^{(2)},\dots,\alpha^{(d)})} 
\quad \mbox{ and }\quad 
\pi\circ y\circ\ell_{0,\be_1^\perp}=c'_{(\alpha^{(2)},\dots,\alpha^{(d)})}.
\]
And we could proceed from there to obtain our desired result.

The next two lemmas show that,
for all $\bv\in\ZZd$,
the
parallel $(d-1)$-dimensional configurations
$\pi\circ x\circ\ell_{\bv,\be_1^\perp}$
and
$\pi\circ y\circ\ell_{\bv,\be_1^\perp}$
belong to 
$\overline{\Orb(\pi\circ x\circ\ell_{0,\be_1^\perp})}
=\overline{\Orb(\pi\circ y\circ\ell_{0,\be_1^\perp})}$,
that is the
$(d-1)$-dimensional subshift whose language is 
$\Lcal\left(c_{(\alpha^{(2)},\dots,\alpha^{(d)})}\right)
=\Lcal\left(c'_{(\alpha^{(2)},\dots,\alpha^{(d)})}\right)$.

\begin{lemma}\label{lem:ddim_intersection_non_empty}
    Let $d\geq2$ be an integer.
	Let $x,y \in\alfad^{\ZZd}$ be an indistinguishable asymptotic pair satisfying the ordered flip condition.
    For each finite nonempty connected subset $S \subset \{0\}\times\ZZ^{d-1}$, let
    \begin{align*}
        A_S &=\Lcal_S(x\circ\ell_{\bzero,\be_1^\perp})\cup 
              \Lcal_S(y\circ\ell_{\bzero,\be_1^\perp}),\\
        B_S &=\Lcal_S(x\circ\ell_{-\be_1,\be_1^\perp})\cup 
              \Lcal_S(y\circ\ell_{-\be_1,\be_1^\perp}).
    \end{align*}
    We have $A_S\cap B_S\neq\varnothing$.
\end{lemma}

\begin{proof}
    Let $S \subset \{0\}\times\ZZ^{d-1}$ be a finite nonempty connected subset.
    Since $(\pi\circ x\circ\ell_{\bzero,\be_1^\perp},
            \pi\circ y\circ\ell_{\bzero,\be_1^\perp})$ 
    satisfies the $(d-1)$-dimensional ordered flip condition with difference
    set $F\setminus\{-\be_1\}$, from
    \Cref{maintheorem:ddim-complexity-is-FminusS}, we have
    \[
        \#A_S \geq \#\pi(A_S) = \Lcal_S(\pi \circ x\circ\ell_{\bzero,\be_1^\perp})\cup 
        \Lcal_S(\pi \circ y\circ\ell_{\bzero,\be_1^\perp}) 
        = \# (F\setminus\{-\be_1\}-S)
        = \# (F-S) - \#S.
    \]

    By contradiction, assume that $A_S\cap B_S=\varnothing$.
    From \Cref{cor:language-is-the-one-intersecting-F},
    we have that
        $\#\Lcal_S(x\circ\ell_{-\be_1,\be_1^\perp})\geq\#S$
        and
        $\#\Lcal_S(y\circ\ell_{-\be_1,\be_1^\perp})\geq\#S$
        so that $\#B_S\geq\#S$.
        The case $\#B_S=\#S$ is impossible.
        Indeed, $\#B_S=\#S$ implies that
        $B_S=\Lcal_S(x\circ\ell_{-\be_1,\be_1^\perp})=
             \Lcal_S(y\circ\ell_{-\be_1,\be_1^\perp})$.
        Observe that $x(-\be_1)=\symb{1}$ and $y(-\be_1)=\symb{0}$.
        Let $w\in B_S$ be a pattern with the most occurrences of the symbol $\symb{0}$. 
        Since $x,y$ is an indistinguishable asymptotic pair satisfying the flip condition,
        \Cref{cor:language-is-the-one-intersecting-F} implies
        that the pattern $w$ must appear in $x$ intersecting the difference set $F$.
        Since $A_S \cap B_S = \varnothing$, then necessarily the pattern $w$
        appears in $x$ intersecting the position $-\be_1$. 
        Over the same support, there is a pattern in $y$ with one more occurrence of the
        symbol $\symb{0}$. 
        This pattern also belongs to $B_S$, thus it contradicts the 
        maximality of the number of occurrences of the symbol $\symb{0}$ in $w$
        among all patterns in $B_S$.
        Therefore, $\#B_S\geq\#S+1$.

    From \Cref{maintheorem:ddim-complexity-is-FminusS}, we have
    $\#\left(A_S\cup B_S\right)\leq
        \#\Lcal_S(x)=\#(F-S)$.
	Thus
    \begin{align*}
        \#\left(A_S\cap B_S\right)
        &= \#A_S + \#B_S - \#\left(A_S\cup B_S\right)\\
        &\geq  \left(\#(F-S) - \#S\right)
              +\left(\#S+1\right)
              -\#\left(F-S\right)
        = 1.
    \end{align*}
    This contradicts the assumption $A_S\cap B_S=\varnothing$.
    Thus $A_S\cap B_S\neq\varnothing$.
\end{proof}

\begin{lemma}\label{lem:pi_perp_ei_perp_is_sturmian}
    Let $d\geq2$ be an integer.
	Let $x,y \in\alfad^{\ZZd}$ be an indistinguishable asymptotic pair satisfying the ordered flip condition. 
    For every $\bv\in\ZZd$, we have 
    \[
        \Lcal(\pi\circ x\circ\ell_{\bv,\be_1^\perp}) \subset 
        \Lcal(\pi\circ x\circ\ell_{\bzero,\be_1^\perp})
            \quad\text{ and }\quad
        \Lcal(\pi\circ y\circ\ell_{\bv,\be_1^\perp}) \subset 
        \Lcal(\pi\circ y\circ\ell_{\bzero,\be_1^\perp}).
    \]
\end{lemma}

\begin{proof}
	As $\pi(x),\pi(y)$ is an indistinguishable asymptotic pair whose difference set is contained in $\be_1^{\perp}$ it follows that $\pi\circ x\circ\ell_{v,\be_1^\perp} = \pi\circ y\circ\ell_{v,\be_1^\perp}$ for every $v \notin \langle \be_1^\perp\rangle$, and that $\Lcal(\pi\circ x\circ\ell_{\bzero,\be_1^\perp})
	=\Lcal(\pi\circ y\circ\ell_{\bzero,\be_1^\perp})$. Therefore 
    it is sufficient 
    to prove the inclusion $\Lcal(\pi\circ x\circ\ell_{\bv,\be_1^\perp}) \subset 
    \Lcal(\pi\circ x\circ\ell_{\bzero,\be_1^\perp})$.
    By contradiction, suppose that there is $\bv\in\ZZ\be_1$ such that
    $w\in\Lcal(\pi\circ x\circ\ell_{\bv,\be_1^\perp}) \setminus 
         \Lcal(\pi\circ x\circ\ell_{\bzero,\be_1^\perp})\neq\varnothing$. Let $w' \in \pi^{-1}(w)$, it follows that $w' \in \Lcal(x\circ\ell_{\bv,\be_1^\perp})\setminus \Lcal(x\circ\ell_{0,\be_1^\perp})$. Using that $x,y$ are indistinguishable and satisfy the flip condition we conclude using~\Cref{lem:exists-occ-intersecting-F} that $w' \in \Lcal(x\circ\ell_{-\be_1,\be_1^\perp})\setminus \Lcal(x\circ\ell_{0,\be_1^\perp})$ and thus that $w\in\Lcal(\pi\circ x\circ\ell_{-\be_1,\be_1^\perp}) \setminus 
         \Lcal(\pi\circ x\circ\ell_{\bzero,\be_1^\perp})$. In other words, without loss of generality we may assume that $\bv=-\be_1$.
    
    For every sufficiently large $n \in \NN$, if we let 
    $S_n = \{-1\}\times \llbracket-n,n\rrbracket^{d-1}$, 
    then the pattern $p = (\pi(x))|_{S_n}$ contains $w$ 
    and thus does
    not occur in $\pi(x)\circ\ell_{\bzero,\be_1^\perp}$. Define $\be_0 = 0$ and let $j \in \{0,\dots,d\}\setminus \{1\}$, $S_n^j = S_n \cup \{-\be_j \}$ and $p^j = \pi(x)|_{S_n^j}$. 
    As $\pi(x), \pi(y)$ is indistinguishable,
    there must exist ${u_j} \in  (F\setminus\{-\be_1\})-S_n^j$ so that $\sigma^{{u_j}}(\pi(y)) \in
    [p^j]$. As $p = p^j|_{S_n}$ does not occur in
    $\pi(x)\circ\ell_{\bzero,\be_1^\perp}$, we have that ${u_j} \notin 
    (F\setminus\{-\be_1\})-S_n$. From where we obtain that 
    $u_j\in \be_j+(F\setminus\{-\be_1\})$. By the ordered flip condition we have that 
    \begin{enumerate}
    	\item If $j = 0$, then $(\pi(x))_{0} = \symb{0}$ and $(\pi(y))_{-\be_2}=0$, hence we have $u_j =   \be_0-\be_2 = -\be_2$.
    	\item If $2 \leq j < d$, then $(\pi(x))_{-\be_j} = \symb{j}-\symb{1}$ and $(\pi(y))_{-\be_{j+1}}=\symb{j}-\symb{1}$, hence we have $u_j = \be_j - \be_{j+1}$.
    	\item If $j = d$, then $(\pi(x))_{-\be_d} = d-\symb{1}$ and $(\pi(y))_{0}=d-\symb{1}$, hence we have $u_j = -\be_d+\be_0 = \be_d$.
    \end{enumerate}
    
    Notice that in any case we have $u_j \in \langle\be_1^{\perp}\rangle $. As $\pi(x), \pi(y)$ is asymptotic outside of
    $F\setminus\{-\be_1\}$, we conclude that 
    \[
        \sigma^{u_j}(\pi(x))|_{S_{n}} = 
        \sigma^{u_j}(\pi(y))|_{S_{n}} =  
                (\pi(x))|_{S_{n}}\]
            for every large enough $n$. Noting that the set $\mathcal{G} = \{ u_j : j \in \{0,\dots,d\}\setminus \{1\} \}$ generates (as a group) $\{0\}\times \ZZ^{d-1}$, it follows that the configuration $\pi\circ x\circ\ell_{-\be_1,\be_1^\perp}$ is constant and thus
    $\Lcal_S(\pi\circ x\circ\ell_{-\be_1,\be_1^\perp})$ is a singleton
    for every finite support $S\subset\{0\}\times\ZZ^{d-1}$.
    
    Assuming 
    $S\subset\{0\}\times\ZZ^{d-1}$ is the shape of the pattern $w$, we have
    $\Lcal_S(\pi\circ x\circ\ell_{-\be_1,\be_1^\perp})=\{w\}$.
    From \Cref{lem:ddim_intersection_non_empty},
    for all finite nonempty connected subset
    $S\subset\{0\}\times\ZZ^{d-1}$,
	we have $A_S\cap B_S\neq\varnothing$ where
    \begin{align*}
        A_S &=\Lcal_S(x\circ\ell_{\bzero,\be_1^\perp})\cup 
              \Lcal_S(y\circ\ell_{\bzero,\be_1^\perp})
             =\Lcal_S(x\circ\ell_{\bzero,\be_1^\perp}),\\
        B_S &=\Lcal_S(x\circ\ell_{-\be_1,\be_1^\perp})\cup 
              \Lcal_S(y\circ\ell_{-\be_1,\be_1^\perp})
             =\Lcal_S(x\circ\ell_{-\be_1,\be_1^\perp})
    \end{align*}
    which also holds under the projection by $\pi$.
    Therefore, if $S$ is the shape of the pattern $w$, we have
    \begin{align*}
        \varnothing &\neq \pi\left(A_S\right)
                     \cap \pi\left(B_S\right)\\
        &= \Lcal_S(\pi\circ x\circ\ell_{\bzero,\be_1^\perp}) \cap 
           \Lcal_S(\pi\circ x\circ\ell_{-\be_1,\be_1^\perp})\\
        &= \Lcal_S(\pi\circ x\circ\ell_{\bzero,\be_1^\perp})
           \cap \{w\}.
    \end{align*}
    This implies that 
    $w\in \Lcal_S(\pi\circ x\circ\ell_{\bzero,\be_1^\perp})$
    which contradicts the definition of $w$.
    Thus we conclude that
    $\Lcal(\pi\circ x\circ\ell_{\bv,\be_1^\perp}) \subset
     \Lcal(\pi\circ x\circ\ell_{\bzero,\be_1^\perp})$
    for all $\bv\in\ZZd$.
\end{proof}

Given a configuration $x \in A^{\ZZd}$, we say a pattern $p\in A^S$ occurs with \define{bounded gaps} if there exists $n \in \NN$ such that for any $v \in \ZZd$, there is $u \in \llbracket -n,n \rrbracket^d$ such that $\sigma^{v+u}(x)|_S = p$. If a pattern does not occur with bounded gaps, this means that there is a sequence  $(v_i)_{i \in \NN}$ with $v_i \in \ZZd$ such that $p$ does not occur in any accumulation point of the sequence $(\sigma^{v_i}(x))_{i \in \NN}$.

\begin{lemma}\label{lem:the_restriction_of_URwFC_is_UR}
    Let $d\geq1$ be an integer.
	Let $x,y \in\alfad^{\ZZd}$ be an indistinguishable asymptotic pair satisfying the ordered flip condition. 
    If $x$ is uniformly recurrent, then $\pi\circ x\circ\ell_{0,\be_1^\perp}$ is uniformly recurrent.
\end{lemma}

\begin{proof}
	Suppose that $\pi\circ x\circ\ell_{0,\be_1^\perp}$ is not uniformly recurrent and let $p \in \Lcal(\pi\circ x\circ\ell_{0,\be_1^\perp})$ be a pattern in its language which does not occur with bounded gaps. Let $S$ be the support of $p$. Let $\psi \colon \alfa{d-1}^{\ZZd} \to \{ \smiley,\hexagon \}^{\ZZd}$ be the sliding-block code such that for any $z \in \alfa{d-1}^{\ZZd}$ and $v \in \ZZd$, \[ \psi(z)_v = \begin{cases}
		\smiley &\mbox{ if } \sigma^v(z)|_{\{0\}\times S} = p,\\
		\hexagon & \mbox{ otherwise. }
	\end{cases}   \]

    As $x$ is uniformly recurrent, the topological closure of its orbit $X=\overline{\Orb(x)}$ is a minimal subshift and it follows that both $\pi(X)$ and $\psi(\pi(X))$ are also  minimal subshifts. Let us denote $w = \psi(\pi(x))$.

	For $n \in \NN$, let $B_n = \llbracket-n,n\rrbracket^{d-1}$ and let $h_n$ denote the pattern with support $B_n$ which is identically $\hexagon$. For every $t \in \ZZ \be_1$, we define \[ N(t) = \sup\{n\in\NN \mid h_n \text{ occurs in }
	w\circ\ell_{t,\be_1^\perp} 
	\text{ with bounded gaps}\}.  \]  
	Notice that the values of $N(t)$ do not change if we replace $w$ by an accumulation point of a sequence of shifts of $w$ by vectors in $\{0\}\times \ZZ^{d-1}$. 
    Let $(t_i)_{i \geq 1}$ be an enumeration of $\ZZ \be_1$. 
    We construct a sequence $(w_i)_{i \geq 0}$ of configurations in $\psi(\pi(X))$ as follows. 
    Let $w_{0} =w$. For every $i \geq 1$,
    we construct the configuration $w_i$ from $w_{i-1}$ according to one of the following three cases:
	
	\textbf{Case 1:} $N(t_i) = -\infty$. In this case, the symbol $\hexagon$ does not occur in $w\circ\ell_{t,\be_1^\perp}$ with bounded gaps and thus there is a sequence $(u_n)_{n \in \NN}$ with $u_n \in \{0\}\times \ZZ^{d-1}$ for which $\hexagon$ does not appear in $(\sigma^{u_n}(w_{i-1}) \circ\ell_{t_i,\be_1^\perp})|_{B_n}$. We let $w_i$ be any accumulation point of this sequence.
	
	\textbf{Case 2:} $N(t_i) \in \NN$. This means that there is a largest $n \in \NN$ for which $h_{n}$ occurs in $w_{i-1}\circ\ell_{t_i,\be_1^\perp}$ with bounded gaps. In this case, there is a sequence $(u_n)_{n \in \NN}$ with $u_n \in \{0\}\times \ZZ^{d-1}$ for which $h_{n+1}$ does not appear in $(\sigma^{u_n}(w_{i-1}) \circ\ell_{t_i,\be_1^\perp})|_{B_n}$. We let $w_i$ be any accumulation point of this sequence.
	
	\textbf{Case 3:} $N(t_i) = \infty$. Here for every $n \in \NN$ the pattern $h_n$ occurs in $w_{i-1}\circ\ell_{t_i,\be_1^\perp}$ with bounded gaps. In this case there is a sequence $(u_n)_{n \in \NN}$ with $u_n \in \{0\}\times \ZZ^{d-1}$ such that $(\sigma^{u_n}(w_{i-1}) \circ\ell_{t_i,\be_1^\perp})|_{B_n}$ is identically $\hexagon$. We let $w_i$ be any accumulation point of this sequence.
	
	By construction, $w_i \in \psi(\pi(X))$ for every $i \in \NN$. Let $\bar{w}$ be an accumulation point of the sequence $(w_i)_{i \in \NN}$. It follows that $\bar{w} \in \psi(\pi(X))$. This sequence has the following properties:
	
	\begin{enumerate}
		\item $N(t)=-\infty$ if and only if $\bar{w}\circ \ell_{t,\be_1^\perp}$ is identically $\smiley$.
		\item $N(t)= n \in \NN$ if and only if $h_n$ occurs with bounded gaps in $\bar{w}\circ \ell_{t,\be_1^\perp}$ and $h_{n+1}$ does not occur.
		\item $N(t) = \infty$ if and only if $\bar{w}\circ \ell_{t,\be_1^\perp}$ is identically $\hexagon$.
	\end{enumerate}

    Let 
    \[\mathcal{N}= \{ N(t): t \in \ZZ \be_1\}\cap \NN.
    \] 

	Suppose that the collection $\mathcal{N}$ is finite.
    This contradicts the minimality of $\psi(\pi(X))$. Indeed, by the assumption on $p$, we have that both the symbol $\smiley$ and the patterns $h_n$ for every $n \in \NN$ occur in $w \circ\ell_{0,\be_1^\perp}$. In particular for every $n \in \NN$ there is a pattern $q_n$ with support $B_{n+1}$ which occurs in $w \circ\ell_{0,\be_1^\perp}$, such that $q_n|_{B_n} = h_n$ and $q_{n}|_{B_{n+1} \setminus B_n}$ is not identically $\hexagon$. For any $n > \max(\mathcal{N})$, it follows that $q_n$ does not occur in $\bar{w}$, and thus $w \notin \overline{\Orb(\bar{w})}$, contradicting minimality.
	
    Suppose now that the collection $\mathcal{N}$ is infinite. For any $\kappa \in \NN$ we can then find $(k_i)_{i=1,\dots,\kappa}$ with $k_i \in \ZZ e_1$ such that \[ 0 \leq N(k_1) < N(k_2) < \dots < N(k_{\kappa}).  \]
	For every $i \in \{1,\dots, \kappa\}$, let $G(k_i)$ be the smallest integer such that every pattern in $\bar{w}_{k_i}$ with support a translate of $B_{G(k_i)}$ contains $h_{N(k_i)}$ as a subpattern. This value exists due to the fact that $h_{N(k_i)}$ occurs in $\bar{w}_{k_i}$ with bounded gaps.
	
	Let $\ell\in \NN$ be such that the support of $p$ is contained in $B_{\ell}$, let $g = 1+2\max_{i=1,\dots,\kappa}G(k_i)$ and let $m \in \NN$ be an arbitrary number which we shall later on take sufficiently large to find a contradiction. Let us consider any pattern $r_i$ with support $B_{g+m}$ in $\Lcal(\bar{w}_{k_i})$. By definition of $\psi$, there is a pattern $r'_i$ with support $B_{g+m+\ell}$ in $\Lcal(\pi \circ x \circ \ell_{k_i,\be_1^\perp})$ whose image under $\psi$ contains $r_i$ as a subpattern. 
	
	By~\Cref{lem:pi_perp_ei_perp_is_sturmian} it follows that $\Lcal(\pi\circ x\circ\ell_{k_i,\be_1^\perp}) \subset 
	\Lcal(\pi\circ x\circ\ell_{\bzero,\be_1^\perp})$. Also, by~\Cref{prop:d-1dim-flip-condition}, the configurations $\pi\circ x\circ\ell_{0,\be_1^\perp}, \pi\circ y\circ\ell_{0,\be_1^\perp}$ are indistinguishable with the ordered flip condition and thus by~\Cref{lem:exists-occ-intersecting-F} every pattern $r'_i$ must occur in $\pi\circ x\circ\ell_{\bzero,\be_1^\perp}$ intersecting its difference set. Applying the map $\psi$, a simple estimate shows that an occurrence of $r_i$ must appear in $w \circ \ell_{\bzero,\be_1^\perp} $ such that its support $B_{g+m}$ is contained in the set $B_{2(g+m+{\ell})+1}$.
	
    Now let $i,j\in \{1,\dots,\kappa\}$ be distinct. By construction, the patterns $r_i$ and $r_{j}$ can overlap at most in their borders. More explicitly, if $r_i$ were to occur at position $v_i \in \ZZ^{d-1}$ and $r_{j}$ at position $v_{j}\in \ZZ^{d-1}$, then $v_i + B_m \cap v_{j}+ B_m= \varnothing$. This is due to the definition of $g$, because if the intersection were to contain a block $B_{g}$, then it would contain a block of $\hexagon$ larger than the maximum allowed size for one of both patterns, see~\Cref{fig:overlap}.

	\begin{figure}[ht!]
		\begin{tikzpicture}
		\draw[thick] (0,0) rectangle (6,6);
		\draw[thick, pattern=north east lines] (1,1) rectangle (5,5);
		\draw[thick, <->] (7, 1) to (7,5);
		\node at (8,3) {$2m+1$};
		\draw[thick, <->] (6.5, 0) to (6.5,1);
		\node at (7.5,0.5) {$g$};
		\node at (3,0.5) {$\texttt{Possible overlap zone}$};
		\draw[white, fill = white] (1.5,1.3) rectangle (4.5,1.8);
		\node at (3,1.5) {$\texttt{No overlap zone}$};
		\end{tikzpicture}
		\caption{Structure of the patterns $r_i$}
		\label{fig:overlap}
	\end{figure}
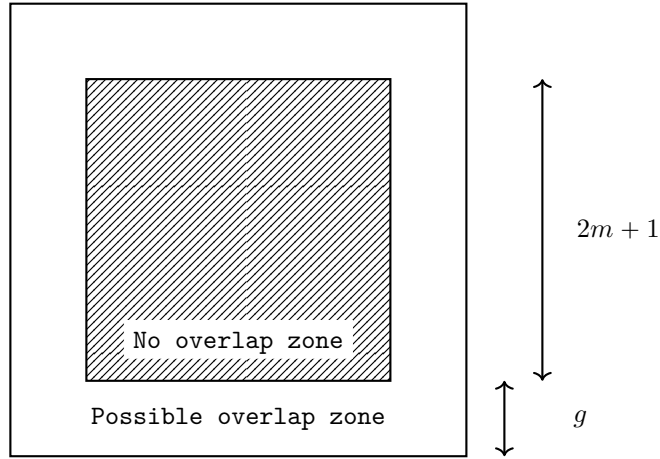
	
	We conclude that within the support $B_{2(g+m+{\ell})+1}$ we must be able to fit $\kappa$ blocks of size $B_m$ with no intersection. In particular we must have that 
    \[ (4(g+m+{\ell})+3)^d= |B_{2(g+m+{\ell})+1}| \geq \kappa|B_m| = \kappa(2m+1)^d. \]
	Let us fix $\kappa = 4^{d}$. This fixes in turn the constant $g$, thus let $K = g+\ell+1$. 
    Using the previous inequality, we obtain 
    \[ 4^d(K+m)^d \geq (4(g+m+{\ell})+3)^d \geq \kappa(2m+1)^d \geq 4^{d}(2m)^d.  \]
	From where we deduce that \[ (K+m)^d \geq (2m)^d  \mbox{ for every } m \in \NN. \]
	The previous inequality is clearly false for large enough $m$, thus we have that $\mathcal{N}$ cannot be infinite either. We conclude that $\pi\circ x\circ\ell_{0,\be_1^\perp}$ is uniformly recurrent.
\end{proof}

The next proposition shows, under some hypothesis, the existence of a factor map
$g\colon \overline{\Orb(x)}\to\RR/\ZZ$.

\begin{proposition}\label{prop:unique_rho}
        Let $d\geq2$ be an integer.
	Let $x,y \in\alfad^{\ZZd}$ be an indistinguishable asymptotic pair satisfying the ordered flip condition
	and assume $x$ is uniformly recurrent.
	Assume there exists a factor map 
	$f\colon \overline{\Orb(\pi\circ x\circ\ell_{0,\be_1^\perp})}\to\RR/\ZZ$
    commuting the actions
	$\dynsys{\overline{\Orb(\pi\circ x\circ\ell_{0,\be_1^\perp})}}{\ZZ^{d-1}}{\sigma}$
	and $\dynsys{\RR/\ZZ}{\ZZ^{d-1}}{T}$.
    Then there is $\rho \in [0,1)$ such that the map $g\colon \overline{\Orb(x)}\to\RR/\ZZ$
	defined by $g(z)=f(\pi\circ z\circ\ell_{0,\be_1^\perp})$
	is a factor map between the actions $\dynsys{\overline{\Orb(x)}}{\ZZ^{d}}{\sigma}$
	and $\dynsys{\RR/\ZZ}{\ZZ^{d}}{R_{\rho}\times T}$ where $R_{\rho}$ is the rotation by $\rho$.
\end{proposition}

\begin{proof}
    The map $g$ is continuous and onto since $f$ is continuous and onto.
    Also since $f$ is a factor map commuting the $\ZZ^{d-1}$-actions, 
    for every $z\in\overline{\Orb(x)}$ and $(0,r)\in\{0\}\times\ZZ^{d-1}$, we have
    \[
        g(\sigma^{(0,r)}z)
    =f(\pi\circ \sigma^{(0,r)}z\circ\ell_{0,\be_1^\perp})
    =f(\sigma^{r}(\pi\circ z\circ\ell_{0,\be_1^\perp}))
    =T^r(f(\pi\circ z\circ\ell_{0,\be_1^\perp}))
    =T^r(g(z)).
    \]
    Thus it remains to show that
    for every $z\in\overline{\Orb(x)}$ and $k\in\ZZ$, we have
    $g(\sigma^{k\be_1}z)=R_\rho^{k}(g(z))$ for some $\rho\in\RR/\ZZ$.
    Since $\overline{\Orb(x)}$ is minimal and $g$ is continuous,
    it is sufficient to prove it for $z=x$ or $z=y$.

    From \Cref {prop:d-1dim-flip-condition},
    the configurations
    $\pi\circ x\circ\ell_{0,\be_1^\perp}$ and 
    $\pi\circ y\circ\ell_{0,\be_1^\perp}\in\alfa{d-1}^{\ZZ^{d-1}}$
    are asymptotic with difference set $F\setminus\{-\be_1\}$.
    Therefore $\pi\circ x\circ\ell_{k,\be_1^\perp}=\pi\circ y\circ\ell_{k,\be_1^\perp}$
    for every $k\in\ZZ\setminus\{0\}$. So we have
    \begin{align*}
    g(\sigma^{k\be_1}x)
        &=f(\pi\circ \sigma^{k\be_1}x\circ\ell_{0,\be_1^\perp})
    =f(\pi\circ x\circ\ell_{k,\be_1^\perp})\\
        &=f(\pi\circ y\circ\ell_{k,\be_1^\perp})
    =f(\pi\circ \sigma^{k\be_1}y\circ\ell_{0,\be_1^\perp})
    =g(\sigma^{k\be_1}y)
    \end{align*}
    for every $k\in\ZZ\setminus\{0\}$.
    Moreover, for every $r\in\ZZ^{d-1}$, we have
    \begin{align*}
     f(\pi\circ x\circ\ell_{0,\be_1^\perp})
    -f(\pi\circ y\circ\ell_{0,\be_1^\perp})
        &=f(\pi\circ x\circ\ell_{0,\be_1^\perp})+T^r(0)
    -f(\pi\circ y\circ\ell_{0,\be_1^\perp})-T^r(0)\\
        &=f(\sigma^r\pi\circ x\circ\ell_{0,\be_1^\perp})
    -f(\sigma^r\pi\circ y\circ\ell_{0,\be_1^\perp})
    \end{align*}
    which goes to 0 when $\Vert r\Vert\to\infty$
    since
    $\pi\circ x\circ\ell_{0,\be_1^\perp}$
    and $\pi\circ y\circ\ell_{0,\be_1^\perp}$
    are asymptotic.
    We conclude that 
    \begin{equation}\label{eq:gkx=gky}
        g(\sigma^{k\be_1}x)=g(\sigma^{k\be_1}y)
    \end{equation}
    for every $k\in\ZZ$.

    The remaining of the proof is based on the following observation which we use several times.
\begin{observation}\label{obs:trick-to-prove-equalities-of-the-rho}
    Let $d\geq2$ be an integer.
    Let $z,z'\in\overline{\Orb(x)}$.
    If for all $m\in\NN$
    there exist two patterns $u$ and $v$ of support
    $\{0\}\times\llbracket 0,m-1\rrbracket^{d-1}$
    and a vector $t\in\{0\}\times\ZZ^{d-1}$ such that
    \[
        \pi(z),\pi(z') \in[u]\cap\sigma^{t+\be_1}[v]
    \]
    then
    \[
          g(z) -g(\sigma^{-\be_1}z)
        = g(z') -g(\sigma^{-\be_1}z').
    \]
\end{observation}

\begin{proof} %
    The domain of the factor map $f$ is compact so $f$ is uniformly continuous.
    Therefore, for all $\varepsilon>0$, there exists
    $m\in\NN$ such that 
    for all patterns $w$ of shape 
    $\llbracket -\lfloor\frac{m}{2}\rfloor,-\lfloor\frac{m}{2}\rfloor+m-1\rrbracket^{d-1}$,
    the Lebesgue measure of the interval $f([w])$ is less than $\varepsilon/2$.
    Since the Lebesgue measure of the interval
    $f(\sigma^k[w])=T^k(f([w]))$ is equal to the Lebesgue measure
    of $f([w])$ for every $k\in\ZZ^{d-1}$, we
    also have that for all patterns $w$ of shape 
    $B=\llbracket 0,m-1\rrbracket^{d-1}$,
    the Lebesgue measure of the interval $f([w])$ is less than $\varepsilon/2$.
    From the hypothesis, let $u$ and $v$ be two patterns of support
    $\{0\}\times\llbracket 0,m-1\rrbracket^{d-1}$
    and $t\in\{0\}\times\ZZ^{d-1}$ be a vector such that
    $\pi(z),\pi(z') \in[u]\cap\sigma^{t+\be_1}[v]$.
    We obtain
    \begin{align*}
        g(z) - g(\sigma^{-\be_1}z) 
        &= f(\pi\circ z\circ\ell_{0,\be_1^\perp}) 
        -  f(\pi\circ\sigma^{-\be_1}z\circ\ell_{0,\be_1^\perp}) \\
        &= f(\pi(z)\circ\ell_{0,\be_1^\perp}) 
        -  f(\sigma^{-\be_1}\pi(z)\circ\ell_{0,\be_1^\perp}) \\
        &\in f([u]\circ\ell_{0,\be_1^\perp}) - f(\sigma^{t} [v]\circ\ell_{0,\be_1^\perp})
    \end{align*}
    which is an interval in $\RR/\ZZ$ of size at most
        $\frac{\varepsilon}{2}
        +\frac{\varepsilon}{2}=\varepsilon$.
    Similarly, 
    \[
        g(z') - g(\sigma^{-\be_1}z') 
        \in f([u]\circ\ell_{0,\be_1^\perp}) - f(\sigma^{t} [v]\circ\ell_{0,\be_1^\perp}).
    \]
    Thus we have
    \[
        \left\vert
        \left(
        g(z) - g(\sigma^{-\be_1}z) 
        \right)
        -
        \left(
        g(z') - g(\sigma^{-\be_1}z') 
        \right)
        \right\vert
        \leq \varepsilon.
    \]
    Since this holds for all $\varepsilon>0$, 
    it concludes the proof of the observation.
\end{proof} %

    Our first goal is to show using \Cref{obs:trick-to-prove-equalities-of-the-rho}
    that for all $k\in\ZZ$ 
    we have
    \begin{equation}\label{eq:diff-in-3-levels}
        g(\sigma^{k\be_1}x) - g(\sigma^{(k-1)\be_1}x)
        \in\left\{ 
        g(\sigma^{\be_1}x) -g(x),
        g(x) -g(\sigma^{-\be_1}x),
        g(\sigma^{-\be_1}x) -g(\sigma^{-2\be_1}x)\right\}.
    \end{equation}
    Let $m\in\NN$.
    Let $B=\{0\}\times\llbracket 0,m-1\rrbracket^{d-1}\subset\ZZd$ be a
    $d$-dimensional box in $\ZZ^d$ of size $m$ in all directions except the direction $\be_1$.
    Let $u$ and $v$ be two patterns of support
    $\{0\}\times\llbracket 0,m-1\rrbracket^{d-1}$
    appearing in the configuration $x$ such that
    $\sigma^{k\be_1}x \in[u]\cap\sigma^{\be_1}[v]$.
    Thus
    \[
        \pi(\sigma^{k\be_1}x)
        \in[\pi(u)]\cap\sigma^{\be_1}[\pi(v)].
    \]
    The fact that the pair $(x,y)$ is indistinguishable implies that 
    the pattern $[u]\cap\sigma^{\be_1}[v]$
    of support $B\cup(B-\be_1)$ must
    appear in $x$ (and $y$) intersecting the difference set. Therefore,
    there exists $r\in\{0\}\times\ZZ^{d-1}$
    and $j\in\{-1,0,1\}$ such that
    $\sigma^{j\be_1+r}x \in[u]\cap\sigma^{\be_1}[v]$.
    Thus
    \[
        \pi(\sigma^{j\be_1+r}x)
        \in[\pi(u)]\cap\sigma^{\be_1}[\pi(v)].
    \]
    From \Cref{obs:trick-to-prove-equalities-of-the-rho} (here $t=0$),
    we obtain
    \begin{align*}
          g(\sigma^{k\be_1}x) -g(\sigma^{-\be_1}\sigma^{k\be_1}x)
        &= g(\sigma^{j\be_1+r}x) -g(\sigma^{-\be_1}\sigma^{j\be_1+r}x)\\
        &= g(\sigma^{j\be_1}x) + T^r(0) -g(\sigma^{(j-1)\be_1}x) - T^r(0)\\
        &= g(\sigma^{j\be_1}x) -g(\sigma^{(j-1)\be_1}x)
    \end{align*}
    which shows that \eqref{eq:diff-in-3-levels} holds.

    Our next goal is to show the existence 
    of some $\rho\in\RR/\ZZ$ such that
    \begin{equation}\label{eq:rho-for-3-levels}
        g(\sigma^{\be_1}x) - g(x)
        = g(x) - g(\sigma^{-\be_1}x)
        = g(\sigma^{-\be_1}x) -g(\sigma^{-2\be_1}x)
        = \rho.
    \end{equation}
    The strategy is to find patterns satisfying
    \Cref{obs:trick-to-prove-equalities-of-the-rho}.
    Let $m\in\NN$.
    Let $B=\{0\}\times\llbracket 0,m-1\rrbracket^{d-1}\subset\ZZd$ be a
    $d$-dimensional box in $\ZZ^d$ of size $m$ in all directions except the direction $\be_1$.
    Let $S\subset\{0\}\times\ZZ^{d-1}\subset\ZZd$ be the union of all
    translates of $B$ that intersect the difference set $F\setminus\{-\be_1\}$
    of the pair $(\pi(x), \pi(y))$,
    that is,
    \[
        S = \bigcup_{k\in K}
            B+k 
            \quad
            \text{ where }
            \quad
        K = \left\{k\in\ZZd \colon (B+k)\cap (F\setminus \{-\be_1\})\neq\varnothing
            \right\}.
    \]
    The restrictions of the configurations 
    $\pi(x)$ and 
    $\pi(y)$ 
    to the support $S$ have the nice property of containing their language of
    patterns of shape $B$ in the most optimal way, i.e., they contain exactly
    one occurrence of each pattern of shape $B$.
    Indeed, from \Cref {prop:d-1dim-flip-condition},
    the configurations
    $\pi\circ x\circ\ell_{0,\be_1^\perp}$ and 
    $\pi\circ y\circ\ell_{0,\be_1^\perp}\in\alfa{d-1}^{\ZZ^{d-1}}$
    satisfy the $(d-1)$-dimensional ordered flip condition. 
    Therefore, from \Cref{cor:language-is-the-one-intersecting-F}, it follows that
    the patterns 
    $H_x=\pi(x)|_S$ and 
    $H_y=\pi(y)|_S$ 
    contain exactly one occurrence of every
    pattern of shape $B$ that are in 
    $\Lcal_B(\pi(x))=\Lcal_B(\pi(y))$.

    We write $\pi(x)\in[H_x]$
    and $\pi(y)\in[H_y]$ where the cylinders are within $\overline{\Orb(\pi(x))}$.
    Since $\pi(x)_0\neq\pi(y)_0$, we have $[H_x]\cap[H_y]=\varnothing$.
    Observe also that,
    from \Cref{cor:language-is-the-one-intersecting-F}, 
    the pattern $H_x$ has only one occurrence in $\pi(x)$ 
    whose support intersects the difference set $F\setminus\{-\be_1\}$.
    More formally, if 
    $v\in F\setminus\{-\be_1\}-S$ and 
    $[H_x]\cap\sigma^v[H_x]\neq\varnothing$, then $v=0$.

    From \Cref{prop:d-1dim-flip-condition}, we also
    have $\Lcal(\pi\circ x\circ\ell_{\bzero,\be_1^\perp}) =
        \Lcal(\pi\circ y\circ\ell_{\bzero,\be_1^\perp})$.
    From \Cref{lem:pi_perp_ei_perp_is_sturmian},
    for every $\bv\in\ZZd$, we have 
    \[
        \Lcal(\pi\circ x\circ\ell_{\bv,\be_1^\perp}) \subseteq
        \Lcal(\pi\circ x\circ\ell_{\bzero,\be_1^\perp}) =
        \Lcal(\pi\circ y\circ\ell_{\bzero,\be_1^\perp}) \supseteq
        \Lcal(\pi\circ y\circ\ell_{\bv,\be_1^\perp}).
    \]
Since $x$ is uniformly recurrent,
we have that 
$\pi\circ x\circ\ell_{\bzero,\be_1^\perp}$
and
$\pi\circ y\circ\ell_{\bzero,\be_1^\perp}$ 
    are uniformly recurrent by~\Cref{lem:the_restriction_of_URwFC_is_UR}. 
    Thus 
    $\overline{\Orb(\pi\circ x\circ\ell_{\bzero,\be_1^\perp})}
    =\overline{\Orb(\pi\circ y\circ\ell_{\bzero,\be_1^\perp})}$ is a
    minimal subshift.
    We deduce 
    $\pi\circ x\circ\ell_{\bv,\be_1^\perp}\in
    \overline{\Orb(\pi\circ x\circ\ell_{\bzero,\be_1^\perp})}$,
    $\pi\circ y\circ\ell_{\bv,\be_1^\perp}\in
    \overline{\Orb(\pi\circ y\circ\ell_{\bzero,\be_1^\perp})}$
    and the equality of the languages:
    \[
        \Lcal(\pi\circ x\circ\ell_{\bv,\be_1^\perp}) =
        \Lcal(\pi\circ x\circ\ell_{\bzero,\be_1^\perp}) =
        \Lcal(\pi\circ y\circ\ell_{\bzero,\be_1^\perp}) =
        \Lcal(\pi\circ y\circ\ell_{\bv,\be_1^\perp})
    \]
    for every $\bv\in\ZZd$.

    Therefore, the pattern $H_x$
    must occur in $\pi\circ y\circ\ell_{\be_1,\be_1^\perp}$.
    Let $t\in \{0\}\times\ZZ^{d-1}$ be such that
    $\pi(y)\in[\sigma^{-\be_1-t}H_x]$.
    Since 
    $\pi\circ x\circ\ell_{\be_1,\be_1^\perp}
    =\pi\circ y\circ\ell_{\be_1,\be_1^\perp}$, we also have
    $\pi(x)\in[\sigma^{-\be_1-t}H_x]$.
	Recall that by~\Cref{prop:invariance_sliding_block_code}, we have that 
	$(\pi(x), \pi(y))$ is an indistinguishable asymptotic pair.
    Since the pattern $\pi(x)|_{S\cup(S-t-\be_1)}$ appears in $\pi(x)$ 
    intersecting the difference set $F\setminus\{-\be_1\}$, 
    it must appear in $\pi(y)$
    intersecting the difference set $F\setminus\{-\be_1\}$.
    Formally, there exists 
    $v\in F\setminus\{-\be_1\}-\left(S\cup(S-t-\be_1)\right)$ such that
    $\sigma^{-v}(\pi(x))\in[\sigma^{-\be_1-t}H_x]\cap[H_y]$.
    There are two cases to consider:
    \begin{itemize}
        \item If $v\in F\setminus\{-\be_1\}-S$, then
            $\sigma^{\be_1+t}(\pi(x))\in[H_x]\cap \sigma^v[H_x]$. Therefore $v=0$ and
            $\pi(x)=\sigma^{-v}(\pi(x))\in[H_y]$. 
            But $\pi(x)\in[H_x]$, which contradicts $[H_x]\cap[H_y]=\varnothing$.
        \item If $v\in F\setminus\{-\be_1\}-\left(S-\be_1-t\right)$, then
            $\pi(x)\in[H_x]\cap \sigma^{v-\be_1-t}[H_x]$. 
            Also $v-\be_1-t\in F\setminus\{-\be_1\}-S$, thus we must have
            $v-\be_1-t=0$. Therefore
            $\pi(x)\in[\sigma^{v}H_y]=[\sigma^{\be_1+t}H_y]$.
    \end{itemize}
    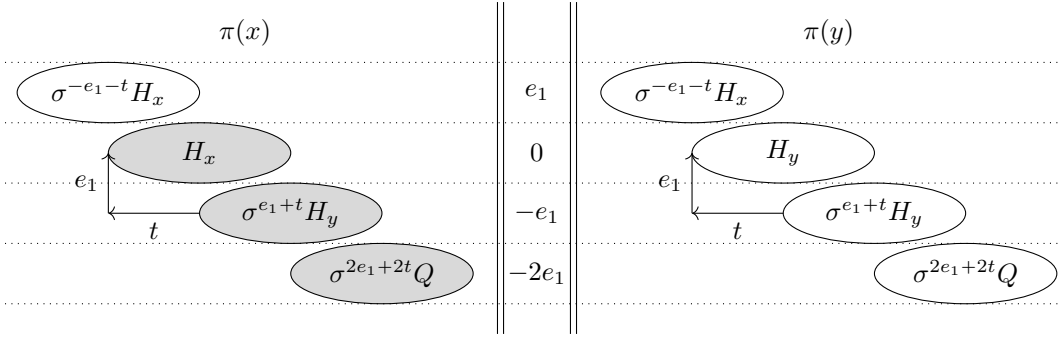
\begin{figure}[ht]
    \begin{center}
    \begin{tikzpicture}[scale=.8]
        \def\t{-1.5}
    \begin{scope}[xshift=-4.8cm]
        \node at (0.75,2) {$\pi(x)$};
        \draw (\t,1)     ellipse (15mm and 5mm) node {$\sigma^{-\be_1-t}H_x$};
        \draw[fill=black!15] (0,0)      ellipse (15mm and 5mm) node {$H_x$};
        \draw[fill=black!15] (-\t,-1)   ellipse (15mm and 5mm) node {$\sigma^{\be_1+t}H_y$};
        \draw[fill=black!15] (-2*\t,-2) ellipse (15mm and 5mm) node {$\sigma^{2\be_1+2t}Q$};
        \foreach \x in {-1}
        {
            \draw[->] (-1.5*\x-1.5,\x) -- node[below] {$t$} ++ (-1.5, 0); 
            \draw[->] (-1.5*\x-3,\x) -- node[left] {$\be_1$} ++ (0,1);
        }
    \end{scope}

    \begin{scope}[xshift=4.8cm]
        \node at (0.75,2) {$\pi(y)$};
        \draw (\t,1)     ellipse (15mm and 5mm) node {$\sigma^{-\be_1-t}H_x$};
        \draw (0,0)      ellipse (15mm and 5mm) node {$H_y$};
        \draw (-\t,-1)   ellipse (15mm and 5mm) node {$\sigma^{\be_1+t}H_y$};
        \draw (-2*\t,-2) ellipse (15mm and 5mm) node {$\sigma^{2\be_1+2t}Q$};
        \foreach \x in {-1}
        {
            \draw[->] (-1.5*\x-1.5,\x) -- node[below] {$t$} ++ (-1.5, 0); 
            \draw[->] (-1.5*\x-3,\x) -- node[left] {$\be_1$} ++ (0,1);
        }
    \end{scope}

    \foreach \y in {-2,...,2}
        \draw[dotted] (-8, \y-.5) -- (9.5,\y-.5);

    \draw (0.1, -3) -- (0.1, 2.5);
    \draw (0.2, -3) -- (0.2, 2.5);
    \draw (1.3, -3) -- (1.3, 2.5);
    \draw (1.4, -3) -- (1.4, 2.5);

    \node at (0.75, 1) {$\be_1$};
    \node at (0.75, 0) {0};
    \node at (0.75,-1) {$-\be_1$};
    \node at (0.75,-2) {$-2\be_1$};

    \end{tikzpicture}
    \end{center}
        \caption{The patterns $H_x$ and $H_y$ appearing in configurations
        $\pi(x)$ and $\pi(y)$.
        The pattern $\pi(x)|_{S\cup(S-t-\be_1)\cup(S-2t-2\be_1)}$ 
        is shown in gray background.
        Since it appears in $\pi(x)$ intersecting the difference set
        $F\setminus\{-\be_1\}$, it must appear in $\pi(y)$ intersecting the
        difference set $F\setminus\{-\be_1\}$.}
        \label{fig:Hx-Hy}
    \end{figure}
    Let $Q=(\sigma^{-2\be_1-2t}\pi(x))|_S$. 
    We have that
    \begin{align*}
        \pi(x)&\in [\sigma^{-\be_1-t}H_x] \cap [H_x] \cap [\sigma^{\be_1+t}H_y] \cap [\sigma^{2\be_1+2t}Q],\\
        \pi(y)&\in [\sigma^{-\be_1-t}H_x] \cap [H_y] \cap [\sigma^{\be_1+t}H_y] \cap [\sigma^{2\be_1+2t}Q].
    \end{align*}
    The current situation is depicted in \Cref{fig:Hx-Hy}.

    Since the pattern 
    $\pi(x)|_{S\cup(S-t-\be_1)\cup(S-2t-2\be_1)}$ appears in $\pi(x)$ 
    intersecting the difference set $F\setminus\{-\be_1\}$, 
    it must appear in $\pi(y)$
    intersecting the difference set $F\setminus\{-\be_1\}$.
    Formally, there exists 
    $v\in F\setminus\{-\be_1\}-\left(S\cup(S-t-\be_1)\cup(S-2t-2\be_1)\right)$ such that
    $\sigma^{-v}(\pi(x))\in[\sigma^{-\be_1-t}H_x]\cap[H_y]\cap[\sigma^{\be_1+t}H_y]$.
    We consider three cases, see Figure~\ref{fig:overlap_3_cases}:
    \begin{itemize}
        \item If $v\in F\setminus\{-\be_1\}-S$, then
            $\sigma^{\be_1+t}(\pi(x))\in[H_x]\cap \sigma^v[H_x]$. Therefore $v=0$ and
            $\pi(x)=\sigma^{-v}(\pi(x))\in[H_y]$. 
            But $\pi(x)\in[H_x]$, which contradicts $[H_x]\cap[H_y]=\varnothing$.
        \item If $v\in F\setminus\{-\be_1\}-\left(S-t-\be_1\right)$, then
            $\sigma^{-\be_1-t}(\pi(x))\in[H_y]\cap \sigma^{v-\be_1-t}[H_y]$. 
            Also $v-\be_1-t\in F\setminus\{-\be_1\}-S$, thus we must have
            $v-\be_1-t=0$. 
            On the one hand, we have $\sigma^{-v}\pi(x)\in[\sigma^{\be_1+t}H_y]$.
            On the other hand, we have
            $\sigma^{-v}\pi(x)=\sigma^{-\be_1-t}\pi(x)\in[\sigma^{\be_1+t}Q]$.
            We deduce the equality $Q=H_y$.
            We may now use \Cref{obs:trick-to-prove-equalities-of-the-rho}.
            Let $u=\pi(y)|_B$ be the pattern of support $B$ within $H_y$.
            We have 
            \[
                \pi(y)\in[H_y]\subset[u]
                \quad
                \text{ and }
                \quad
                \pi(y)\in[\sigma^{\be_1+t}H_y]\subset[\sigma^{\be_1+t}u].
            \]
            Also
            \[
                \sigma^{-\be_1-t}\pi(y)\in[H_y]\subset[u]
                \quad
                \text{ and }
                \quad
                \sigma^{-\be_1-t}\pi(y)\in[\sigma^{\be_1+t}Q]=[\sigma^{\be_1+t}H_y]\subset[\sigma^{\be_1+t}u].
            \]
            The pattern $u$ also appears in $H_x$
            Let $r\in F\setminus\{-\be_1\}-B$ such that $[\sigma^rH_x]\subset[u]$.
            We have
            \[
                \sigma^{\be_1+t+r}\pi(x)\in[\sigma^rH_x]\subset[u]
                \quad
                \text{ and }
                \quad
                \sigma^{\be_1+t+r}\pi(x)\in[\sigma^{\be_1+t+r}H_x]\subset[\sigma^{\be_1+t}u].
            \]
            Since the above holds for pattern $u$ of arbitrarily large size,
            from \Cref{obs:trick-to-prove-equalities-of-the-rho}, we conclude
            that
            \[
                g(\sigma^{\be_1}x) - g(x)
                = g(y) - g(\sigma^{-\be_1}y)
                = g(\sigma^{-\be_1}y) -g(\sigma^{-2\be_1}y)
                = \rho
            \]
            for some $\rho\in\RR/\ZZ$.

        \item If $v\in F\setminus\{-\be_1\}-\left(S-2t-2\be_1\right)$, then
            $\sigma^{-v+\be_1+t}(\pi(x))\in[H_x]\cap \sigma^{-v+2\be_1+2t}[H_y]$. 
            Since $v-2\be_1-2t\in F\setminus\{-\be_1\}-S$, the support 
            $S+v-2\be_1-2t$ of the
            translated pattern $\sigma^{-v+2\be_1+2t}[H_y]$ intersects the
            difference set $F\setminus\{-\be_1\}$.
            Let $\gamma\in\{0\}\times\ZZ^{d-1}$ be such that
            $B+\gamma\subset S+v-2\be_1-2t$
            and $B+\gamma\cap F\setminus\{-\be_1\}\neq\varnothing$.
            By definition of $S$, we also have $B+\gamma\subset S$.
            In other words, $u_x=\pi(x)|_{B+\gamma}$ is a subpattern of 
            $H_x$ and of $\sigma^{-v+2\be_1+2t}H_y$,
            thus satisfying 
            \[
                [H_x]\cap\sigma^{-v+2\be_1+2t}[H_y]\subset[u_x].
            \]
            Similarly, we have
            $\sigma^{-v}(\pi(x))\in[H_y]\cap \sigma^{-v+2\be_1+2t}[Q]$. 
            We obtain that $u_y=\pi(y)|_{B+\gamma}$ is a subpattern of 
            $H_y$ and of $\sigma^{-v+2\be_1+2t}Q$,
            thus satisfying 
            \[
                [H_y]\cap\sigma^{-v+2\be_1+2t}[Q]\subset[u_y].
            \]
            In summary, we have 
            \[
                \sigma^{\be_1+t}\pi(y)\in[H_x]\subset[u_x]
                \quad
                \text{ and }
                \quad
                \sigma^{\be_1+t}\pi(y)\in[\sigma^{\be_1+t}H_y]\subset[\sigma^{\be_1+t}u_y].
            \]
            Also
            \[
                \pi(x)\in[H_x]\subset[u_x]
                \quad
                \text{ and }
                \quad
                \pi(x)\in[\sigma^{\be_1+t}H_y]\subset[\sigma^{\be_1+t}u_y].
            \]
            Moreover,
            \begin{align*}
                &\sigma^{-v+\be_1+t}\pi(x)\in\sigma^{-v+\be_1+t}[\sigma^{\be_1+t}H_y]
                \subset[u_x]\\
                &\sigma^{-v+\be_1+t}\pi(x)\in\sigma^{-v+\be_1+t}[\sigma^{2\be_1+2t}Q]
                =\sigma^{\be_1+t}[\sigma^{-v+2\be_1+2t}Q]
                \subset[\sigma^{\be_1+t}u_y].
            \end{align*}
            Since the above holds for patterns $u_x$ and $u_y$ of arbitrarily large size,
            from \Cref{obs:trick-to-prove-equalities-of-the-rho}, we conclude
            that
            \[
                g(\sigma^{\be_1}y) - g(y)
                = g(x) - g(\sigma^{-\be_1}x)
                = g(\sigma^{-\be_1}x) -g(\sigma^{-2\be_1}x)
                = \rho
            \]
            for some $\rho\in\RR/\ZZ$.
    \end{itemize}
    \begin{figure}[ht]
        \begin{tikzpicture}[scale=.7]
\foreach \y in {-2,...,2}
\draw[dotted] (-3, \y-.5) -- (18.5,\y-.5);

\def\t{-1.5}

\begin{scope}
    \draw (\t,1)     ellipse (15mm and 5mm) node[left=3mm] {$H_x$};
    \draw (0,0)      ellipse (15mm and 5mm) node[left=3mm] {$H_y$};
    \draw (-\t,-1)   ellipse (15mm and 5mm) node[left=3mm] {$H_y$};
    \draw (-2*\t,-2) ellipse (15mm and 5mm) node[left=3mm] {$Q$};

    \draw[xshift=13mm,fill=black!15] (0,0)      ellipse (15mm and 5mm) node[right=3mm] {$H_x$};
    \draw[xshift=13mm,fill=black!15] (-\t,-1)   ellipse (15mm and 5mm) node[right=3mm] {$H_y$};
    \draw[xshift=13mm,fill=black!15] (-2*\t,-2) ellipse (15mm and 5mm) node[right=3mm] {$Q$};

    \draw[dashed] (\t,1)     ellipse (15mm and 5mm);
    \draw[dashed] (0,0)      ellipse (15mm and 5mm);
    \draw[dashed] (-\t,-1)   ellipse (15mm and 5mm);
    \draw[dashed] (-2*\t,-2) ellipse (15mm and 5mm);
\end{scope}

\begin{scope}[xshift=7cm]
    \draw (\t,1)     ellipse (15mm and 5mm) node[left=3mm] {$H_x$};
    \draw (0,0)      ellipse (15mm and 5mm) node[left=3mm] {$H_y$};
    \draw (-\t,-1)   ellipse (15mm and 5mm) node[left=3mm] {$H_y$};
    \draw (-2*\t,-2) ellipse (15mm and 5mm) node[left=3mm] {$Q$};

    \draw[xshift=13mm,fill=black!15] (\t,1)    ellipse (15mm and 5mm) node[right=3mm] {$H_x$};
    \draw[xshift=13mm,fill=black!15] (0,0)     ellipse (15mm and 5mm) node[right=3mm] {$H_y$};
    \draw[xshift=13mm,fill=black!15] (-\t,-1)  ellipse (15mm and 5mm) node[right=3mm] {$Q$};

    \draw[dashed] (\t,1)     ellipse (15mm and 5mm);
    \draw[dashed] (0,0)      ellipse (15mm and 5mm);
    \draw[dashed] (-\t,-1)   ellipse (15mm and 5mm);
    \draw[dashed] (-2*\t,-2) ellipse (15mm and 5mm);
\end{scope}

\begin{scope}[xshift=14cm]
    \draw (\t,1)     ellipse (15mm and 5mm) node[left=3mm] {$H_x$};
    \draw (0,0)      ellipse (15mm and 5mm) node[left=3mm] {$H_y$};
    \draw (-\t,-1)   ellipse (15mm and 5mm) node[left=3mm] {$H_y$};
    \draw (-2*\t,-2) ellipse (15mm and 5mm) node[left=3mm] {$Q$};

    \draw[xshift=13mm,fill=black!15] (2*\t,2)  ellipse (15mm and 5mm) node[right=3mm] {$H_x$};
    \draw[xshift=13mm,fill=black!15] (\t,1)    ellipse (15mm and 5mm) node[right=3mm] {$H_y$};
    \draw[xshift=13mm,fill=black!15] (0,0)     ellipse (15mm and 5mm) node[right=3mm] {$Q$};

    \draw[dashed] (\t,1)     ellipse (15mm and 5mm);
    \draw[dashed] (0,0)      ellipse (15mm and 5mm);
    \draw[dashed] (-\t,-1)   ellipse (15mm and 5mm);
    \draw[dashed] (-2*\t,-2) ellipse (15mm and 5mm);

    \def\d{.25}
    \draw[very thick,rounded corners] (-1.3,1-\d)    rectangle node {$u_x$} (-.4,1+\d);
    \draw[very thick,rounded corners] (-\t-1.3,0-\d) rectangle node {$u_y$} (-\t-.4,0+\d);
\end{scope}

\end{tikzpicture}
        \caption{The pattern $\pi(x)|_{S\cup(S-t-\be_1)\cup(S-2t-2\be_1)}$, shown in gray,
                 appears in $\pi(y)$ intersecting the difference set
                 $F\setminus\{-\be_1\}$ in one of three ways.
                 To lighten the figure, we omit the shifts 
                 $\sigma^{\be_1+t}$ in the ellipses. 
                 In the first case, the nontrivial overlap of $H_y$ with itself
                 is impossible, which implies $H_x=H_y$, a contradiction.
                 In the second case, the nontrivial overlap of $H_y$ with itself
                 is impossible, which implies that $H_y=Q$.
                 In the third case, the proof uses the existence of patterns
                 $u_x$ and $u_y$.
                 }
        \label{fig:overlap_3_cases}
    \end{figure}
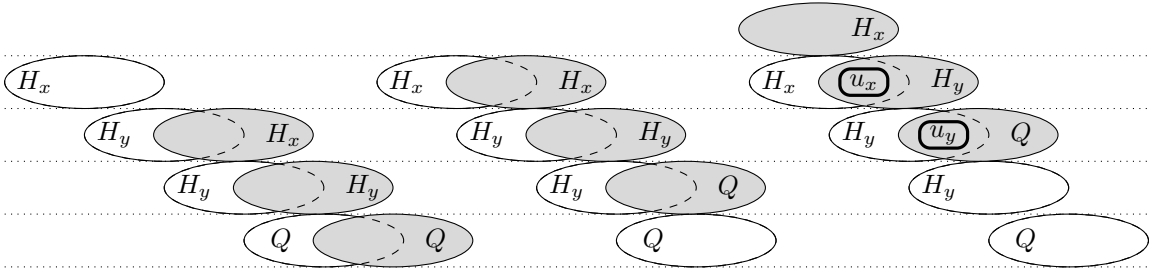
    Using \Cref{eq:gkx=gky}, we obtain that \Cref{eq:rho-for-3-levels} holds.
    From \Cref{eq:diff-in-3-levels} and
    \Cref{eq:rho-for-3-levels},
    we conclude that there exists $\rho\in\RR/\ZZ$ such that for all $k\in\ZZ$, we have
    \[
        g(\sigma^{k\be_1}x) = g(x) + k\rho = R_\rho^{k}(g(x)).
    \]
    Since $\overline{\Orb(x)}$ is minimal and $g$ is continuous,
    we conclude that
    for every $z\in\overline{\Orb(x)}$ and $k\in\ZZ$, we have
    $g(\sigma^{k\be_1}z)=R_\rho^{k}(g(z))$.
\end{proof}

\subsection{Proof of \Cref{thm:multidim_sturmian_characterization}}

In this subsection, we prove \Cref{thm:multidim_sturmian_characterization}.
The proof is essentially done in \Cref{prop:totally_irrational_case}
which assumes the ordered flip condition.
The proof is done by induction on the dimension using results proved in the
previous subsection.
First, we need the next lemma which is used thereafter.

\begin{lemma}\label{lem:singleton-implies-intervals}
    Let $A,B\subset\RR$ be two closed sets such that $A\cup B$ is a compact interval $I$. If $A\cap B$ is a singleton, then $A$ and $B$ are intervals.
\end{lemma}

\begin{proof}
    Let $x$ be the element in $A \cap B$ and let $A' = \operatorname{Int}(I) \setminus A$ and $B' = \operatorname{Int}(I)  \setminus B$. Then $A' \cup B' = \operatorname{Int}(I) \setminus \{x\}$. 
    
    If $x$ is in the boundary of $I$, we conclude that either $A'$ or $B'$ is empty, as both are open disjoint sets and $\operatorname{Int}(I)$ is connected. Therefore one of $A,B$ is equal to $I$ and the other is equal to the singleton $\{x\}$.
    
    If $x$ is in the interior of $I$, we may write $\operatorname{Int}(I) \setminus \{x\}$ as the union of two open intervals. Again as they are connected it follows that one must be $A'$ and the other $B'$, from where it follows that $A$ and $B$ are intervals.\end{proof}

    \begin{proposition}\label{prop:totally_irrational_case}
        Let $d\geq1$ be an integer.
	    Let $x,y \in\alfad^{\ZZd}$ be an indistinguishable asymptotic pair satisfying the ordered flip condition
        and assume $x$ is uniformly recurrent.
        There exists a totally irrational vector $\alpha = (\alpha_1,\dots,\alpha_d)\in [0,1)^d$ such that $ 1>\alpha_1
        >\alpha_2
        >\dots
        >\alpha_d
        >0$, and 
            $x=c_{\alpha}$, $y=c'_{\alpha}$ are the $d$-dimensional characteristic Sturmian configurations with slope $\alpha$.
	\end{proposition}

\begin{proof}
    Our proof proceeds by induction on the dimension $d$.
    The base case $d=1$ is given in~\Cref{thm:sturmian_characterization}. Let $d >1$ and assume that~\Cref{prop:totally_irrational_case} holds for $d-1$. As the configurations
    $\pi\circ x\circ\ell_{0,\be_1^\perp}$ and 
    $\pi\circ y\circ\ell_{0,\be_1^\perp}\in\alfa{d-1}^{\ZZ^{d-1}}$
    satisfy the $(d-1)$-dimensional ordered flip condition
    and
    $\pi\circ x\circ\ell_{0,\be_1^\perp}$ is uniformly recurrent (\Cref{lem:the_restriction_of_URwFC_is_UR} and \Cref{prop:d-1dim-flip-condition})
    it follows by the induction hypothesis that
    $\pi\circ x\circ\ell_{0,\be_1^\perp}$ and 
    $\pi\circ y\circ\ell_{0,\be_1^\perp}$
    are $(d-1)$-dimensional characteristic Sturmian configurations
    associated to a totally irrational slope
    $(\alpha_2,\dots,\alpha_d)\in[0,1)^{d-1}$,
    satisfying 
        $ 1>\alpha_2
           >\dots
           >\alpha_d
           >0$, that is,
    \[
    \pi\circ x\circ\ell_{0,\be_1^\perp}=c_{(\alpha_2,\dots,\alpha_d)} 
    \quad \mbox{ and }\quad 
    \pi\circ y\circ\ell_{0,\be_1^\perp}=c'_{(\alpha_2,\dots,\alpha_d)}.
    \]
    Let $f$ be the factor map 
    $f\colon \overline{\Orb(\pi\circ
    x\circ\ell_{0,\be_1^\perp})}\to\RR/\ZZ$
    obtained in \Cref{lem:ddim-sturmian-factor-map}
    which commutes
    the $(d-1)$-dimensional shift on $\overline{\Orb(\pi\circ
    x\circ\ell_{0,\be_1^\perp})}$ with the $(d-1)$-dimensional rotation 
    by $(\alpha_2,\dots,\alpha_d)$ on the circle $\RR/\ZZ$.
    From \Cref{prop:unique_rho}, 
    there exist $\rho_1\in\RR/\ZZ$ and
    a topological factor map $g\colon \overline{\Orb(x)}\to\RR/\ZZ$
    that 
    commutes
    the $d$-dimensional shift 
    with the $d$-dimensional rotation 
    by $(\rho_1,\alpha_2,\dots,\alpha_d)$ on the circle $\RR/\ZZ$.

	The map $g$ is explicitly given by $g(w)=f(\pi\circ w\circ\ell_{0,\be_1^\perp})$.
    In particular, $g([\symb{0}]\cup [\symb{1}])$, $g([\symb{2}])$, $\dots$,
    $g([d])$ are consecutive intervals from left to right on the unit
    interval.
    
    Next we show that $g([\symb{0}])$ and $g([\symb{1}])$ are also intervals 
    using~\Cref{lem:singleton-implies-intervals}. As both $[0]\cap \overline{\Orb(x)}$ and $[1]\cap \overline{\Orb(x)}$ are compact and $g$ is continuous, it follows that their images are closed, thus it suffices to show that their intersection is a singleton.
    
    We have $g(x)=g(y)=0$.
    Thus $g(\sigma^{-\be_1}x)=g(\sigma^{-\be_1}y)=-\rho_1$.
    Also, $\sigma^{-\be_1}x\in[\symb{1}]$ and $\sigma^{-\be_1}y\in[\symb{0}]$.
    Therefore $-\rho_1\in g([\symb{0}])\cap g([\symb{1}])$.
    By contradiction, suppose that
    $g([\symb{0}])\cap g([\symb{1}])$ contains another element $\gamma\neq-\rho_1$.
    Therefore, there exist $w,z\in\overline{\Orb(x)}$ with $w\in[\symb{0}]$ and
    $z\in[\symb{1}]$
    such that $g(w)=g(z)=\gamma$.
    Since $g(w)=g(z)$, the configurations $w$ and $z$ are equal on many positions.
    More precisely, 
    if $v\in\ZZd$ is such that
    $g(\sigma^v(w))=\gamma + v\cdot
    (\rho_1,\alpha_2,\dots,\alpha_d)=g(\sigma^v(z))$ is in the interior of the interval 
    $g([\symb{i}])$
    for some $\symb{i}\in\{\symb{2},\ldots,\symb{d}\}$, then
    $w_v=\symb{i}=z_v$.
    In other words, the set
    \[
        V=\left\{v\in\ZZd\colon
            \gamma + v\cdot (\rho_1,\alpha_2,\dots,\alpha_d)
            \in\operatorname{Int}(g([\symb{i}]))
            \text{ for some }\symb{i}\in\{\symb{2},\ldots,\symb{d}\}\right\}
    \]
    satisfies $w|_V=z|_V$.

    Let $\varepsilon>0$ be such that 
    $\varepsilon<|\gamma-(-\rho_1)|$
    and $\varepsilon<|\alpha_d|$.
    Let $m\in\NN$ be such that the Lebesgue measure of $f([q])$ is less than
    $\varepsilon$ for every allowed pattern 
    $q:\llbracket 0, m-1 \rrbracket^{d-1}\to\{\symb{0},\dots,d-1\}$
    appearing in the configuration
    $\pi\circ x\circ\ell_{0,\be_1^\perp}$.
    Let $B_m = \{0\} \times \llbracket 0, m-1 \rrbracket^{d-1}$ and 
    $p' \in \mathcal{L}_{B_m}(x)$ be a pattern such that the number of sites 
    $S = \{ u \in B_m : p'(u) \in \{\symb{2},\dots,d\}\}$ is maximized. 
    Let $p\colon S\to\{\symb{2},\dots,d\}$ be the restriction of $p'$ to $S$. 
    From the maximality of the set $S\subset B_m$, 
    we know that for every $u\in\ZZd$ such that $\sigma^u(x)\in[p]$,
    we have $\sigma^u(x)|_{B_m\setminus S}$ is a pattern over symbols 
    $\symb{0}$ and $\symb{1}$ only.
    Thus the pattern $p$ occurs at position $u\in\ZZd$ in $x$
    if and only if
    $\pi(p')$ occurs at position $u$ in $\pi(x)$.
    Also the pattern $p\circ\ell_{0,\be_1^\perp}$ occurs at position $u\in\ZZ^{d-1}$ 
    in $x\circ\ell_{0,\be_1^\perp}$
    if and only if
    $\pi(p')\circ\ell_{0,\be_1^\perp}$ occurs at position $u$ in $\pi(x)\circ\ell_{0,\be_1^\perp}$.
    From \Cref{maintheorem:ddim-complexity-is-FminusS},
    the pattern $\pi(p')\circ\ell_{0,\be_1^\perp}$ 
    of connected support $\llbracket 0, m-1 \rrbracket^{d-1}$
    occurs in $\pi(x)\circ\ell_{0,\be_1^\perp}$
    with support intersecting the difference set of 
    the asymptotic pair
    $\left(\pi(x)\circ\ell_{0,\be_1^\perp},
    \pi(y)\circ\ell_{0,\be_1^\perp}\right)$
    in a unique position.
    Therefore, the pattern $p$ occurs in $x$
    in a unique position with support intersecting the set $F\setminus \{-\be_1\}$.
    We use this property multiple times below.

    From the maximality of $S$, we also have that $[\pi(p)]=[\pi(p')]$
    where the cylinders are taken within $\pi\left(\overline{\Orb(x)}\right)$.
    On the one hand, it is clear that $[\pi(p')]\subseteq[\pi(p')|_S]=[\pi(p)]$.
    On the other hand, suppose $c\in\overline{\Orb(x)}$ such that $\pi(c)\in[\pi(p)]$. 
    If $\pi(c_k)\neq\symb{0}$ for some $k\in B_m\setminus S$, then
    $c_n=\pi(c_k)+1\in\{\symb{2},\dots,d\}$. Thus the pattern
    $c|_{S\cup\{k\}}$ is strictly larger than the pattern $p$ and this
    contradicts the maximality of the set $S$.
    Thus $\pi(c)\in[\pi(p')]$, and hence $[\pi(p)]\subseteq[\pi(p')]$.
    Since the support of $[\pi(p')]$ is connected, we deduce 
    from \Cref{lem:partition_into_intervals} that
    $g([p])
    =f([\pi(p)]\circ\ell_{0,\be_1^\perp})
    =f([\pi(p')]\circ\ell_{0,\be_1^\perp})$ 
    is a nonempty interval whose length is at most $\varepsilon$.

    As $(\alpha_2,\dots,\alpha_d)$ is totally irrational, 
    the interval $g([p])$ has nonempty interior and
    there exists $u\in(\{0\}\times\ZZ^{d-1})\setminus S$ such that
    $g(\sigma^{-u}(z))=g(\sigma^{-u}(w))=\gamma - u\cdot (0,\alpha_2,\dots,\alpha_d)
                   \in\operatorname{Int}(g([p]))$.
                   Thus 
                   $\sigma^{-u}(w)\in[p]$
                   and
                   $\sigma^{-u}(z)\in[p]$.
                   Equivalently,
                   $w\in\sigma^u([p])$
                   and
                   $z\in\sigma^u([p])$.
    For each $\symb{i}\in\{\symb{0},\symb{1},d\}$, let
    $q_\symb{i}\colon \{0\}\cup (S-u)\to\alfad$ be the pattern defined by
    \[
    q_\symb{i}(\bn)=
		\begin{cases}
		p(\bn+\bu)   & \mbox{ if } \bn \in S-u,\\
		\symb{i} & \mbox{ if } \bn = 0.
		\end{cases}
    \]
    We have $w\in[\symb{0}]\cap\sigma^{u}([p])=[q_\symb{0}]$
    and     $z\in[\symb{1}]\cap\sigma^{u}([p])=[q_\symb{1}]$.
    Also $[d]\cap\sigma^{u}([p])=[q_d]$.

    By~\Cref{lem:exists-occ-intersecting-F}, the pattern $q_\symb{1}$ must occur in $x$ intersecting the difference set.
    Recall that $x_{-\be_1}=\symb{1}$. But $\sigma^{-\be_1}x\notin [q_\symb{1}]$
    since the opposite implies that
    $-\rho_1= g(\sigma^{-\be_1}x)\in g([q_\symb{1}])\subset g(\sigma^{u}([p]))$.
    Since $\gamma=g(w) \in g(\sigma^{u}([p]))$
    and $g(\sigma^{u}([p]))=g([p])+ u\cdot(\rho_1,\alpha_2,\dots,\alpha_d)\subset\RR/\ZZ$ 
    is an interval of length at most $\varepsilon$, we have
    $|\gamma-(-\rho_1)|<\varepsilon$, which is a contradiction.
    Therefore the pattern $q_\symb{1}$ must occur in $x$ intersecting the
    difference set in such a way that the subpattern $p$ intersects the difference set.
    Since the symbol $\symb{1}=x_{-\be_1}$ is not in the pattern $p$, 
    the pattern $q_\symb{1}$ must occur in $x$ in such a way that the
    subpattern $p$ intersects the set $F\setminus\{-\be_1\}$.

    The pattern $q_\symb{0}$ must also occur in $x$ intersecting the difference set. 
    Since the pattern $q_\symb{0}$ does not contain the symbol $\symb{1}=x_{-\be_1}$,
    the pattern $q_\symb{0}$ must occur in $x$ intersecting the set $F\setminus\{-\be_1\}$.
    As there is exactly one occurrence of the subpattern $p$ occurring in $x$
    intersecting the set $F\setminus\{-\be_1\}$,
    we must have $x\in[q_\symb{0}]\subset\sigma^{u}([p])$.
    This implies that $y\in[q_d]$. Thus the pattern $q_d$ is in
    the language of $x$ and must also appear in $x$ intersecting the difference set.
    Since the pattern $q_d$ does not contain the symbol $\symb{1}=x_{-\be_1}$,
    the pattern $q_d$ must occur in $x$ intersecting the set $F\setminus\{-\be_1\}$.
    Again, we recall that there is exactly one occurrence of the subpattern $p$
    occurring in $x$
    intersecting the set $F\setminus\{-\be_1\}$. Therefore, we must have 
    $\sigma^{-\be_d}x\in[q_d]\subset\sigma^{u}([p])$.
    In summary, we have
    $0=g(x)\in g(\sigma^{u}([p]))$
    and
    $-\alpha_d=g(\sigma^{-\be_d}x)\in g(\sigma^{u}([p]))$. 
    Since $g(\sigma^{u}([p]))$ is an interval of length at most $\varepsilon$, we have
    $|0-(-\alpha_d)|<\varepsilon$, which is a contradiction.
    Thus we conclude that $g([\symb{0}])\cap g([\symb{1}])$ is a singleton and thus~ from~\Cref{lem:singleton-implies-intervals},
    we deduce that $g([\symb{0}])$ and $g([\symb{1}])$ are intervals. More precisely, from the facts that $g(x)=0$, $g(\sigma^{-\be_2}y)=1$ and $-\rho_1 \in g([0])\cap g([1])$ we conclude that $g([0]) = [0,-\rho_1]$ and $g([1]) = [-\rho_1,-\alpha_2]$. 
    
    We now claim that the vector $\widetilde{\alpha} = (\rho_1,\alpha_2,\dots,\alpha_d)$ is totally irrational. Indeed, were it not the case, there would exist a non-zero $n \in \ZZd$ for which $n \cdot \widetilde{\alpha} =0 \bmod{1}$. 
    Let $\beta \in \RR/\ZZ$ be rationally independent with $\widetilde{\alpha}$, it follows that for any $m \in \ZZd$, the value $\beta+m\cdot \widetilde{\alpha} \in \RR/\ZZ$ does not lie in the boundary of the intervals $g([\symb{i}])$. 
    In particular $g^{-1}(\beta)$ is a singleton. 
    Indeed, let $\widetilde{w},\widetilde{z}\in g^{-1}(\beta)$.
    Then for every $m\in\ZZd$ we have 
    $g(\sigma^m(\widetilde{w}))=g(\sigma^m(\widetilde{z}))=\beta+m\cdot \widetilde{\alpha}\in
          \operatorname{Int}(g([\symb{i}]))$ for some $\symb{i}\in\{\symb{0},\dots,d\}$.
          Thus $\widetilde{w}_m=\symb{i}=\widetilde{z}_m$
          and globally we have the equality $\widetilde{w}=\widetilde{z}$.
          Since $n \cdot \widetilde{\alpha} =0 \bmod{1}$,
    then it follows that $\sigma^n(\widetilde{w}) = \widetilde{w}$. 
    By minimality of $\overline{\Orb(x)}$, it follows that every configuration $\widetilde{z}$ in $\overline{\Orb(x)}$ satisfies $\sigma^n(\widetilde{z}) = \widetilde{z}$. This is incompatible with $x$ and $y$ having a finite difference set as we would have \[ x|_F = \sigma^{kn}(x)|_F = \sigma^{kn}(y)|_F = y|_F \mbox{ for arbitrarily large } k \in \NN. \] 
    
    Hence $\widetilde{\alpha}$ is totally irrational, it follows that the orbit $\ZZd \cdot \widetilde{\alpha}$ lies in the boundary of the intervals exactly for values in $F\cdot \widetilde{\alpha}$. An inspection of the values on the difference set shows that 
    \[
    x=c_{\widetilde{\alpha}} 
    \quad \mbox{ and }\quad 
    y=c'_{\widetilde{\alpha}}.\qedhere
    \]
\end{proof}

\begin{remark}
\Cref{prop:totally_irrational_case} is proven by induction starting with the base case $d=1$
which is dealt with in~\Cref{thm:sturmian_characterization}. 
Technically it might be possible to perform the induction using as base the case of dimension $d=0$. In this way \Cref{prop:totally_irrational_case} 
would contain an independent proof of the case $d=1$ about Sturmian configurations in $\ZZ$. In the current state of our proof this would demand significant changes to the previous lemmas so we do not do it here.
\end{remark}

We now present the proof of~\Cref{thm:multidim_sturmian_characterization}.

\begin{proof}[Proof of \Cref{thm:multidim_sturmian_characterization}]
	Let $\balpha \in [0,1)^d$ be totally irrational. 
    From \Cref{lem:sturmian_is_unirec},
    the characteristic $d$-dimensional Sturmian configurations
    $c_{\balpha}$ and $c'_{\balpha}$ are uniformly recurrent.
    From \Cref{thm:ddim-sturmian-is-indistinguishable}, 
    $(c_{\balpha},c'_{\balpha})$ is a non-trivial indistinguishable asymptotic pair.
    From \Cref{prop:sturmian_is_flip}, 
    the pair $(c_{\balpha},c'_{\balpha})$ satisfies the flip condition. 

    Let $x,y\in\alfad^\ZZd$ be an indistinguishable asymptotic pair satisfying 
    the flip condition
    and assume $x$ is uniformly recurrent. Using~\Cref{lem:flip-reduces-to-ordered-flip} we obtain a permutation matrix $A \in \operatorname{GL}_d(\ZZ)$ such that $(x\circ A,y\circ A)$ is an indistinguishable asymptotic pair which satisfies the ordered flip condition. A straightforward computation shows that $x\circ A$ is uniformly recurrent. From \Cref{prop:totally_irrational_case},
        there exists a totally irrational vector $\alpha = (\alpha_1,\dots,\alpha_d)\in [0,1)^d$ such that $ 1>\alpha_1
        >\alpha_2
        >\dots
        >\alpha_d
        >0$, and 
            $x \circ A =c_{\alpha}$, $y \circ A=c'_{\alpha}$ are the $d$-dimensional characteristic Sturmian configurations with slope $\alpha$. Then $x = c_{\alpha}\circ A^{-1}$ and $y=c'_{\alpha}\circ A^{-1}$.

       Let us compute these configurations explicitly. Recall that the adjoint
       of a permutation matrix is its inverse, that is $A^T=A^{-1}$. It follows,
       using \Cref{eq:characteristic-sturmian-formula-in-intro},
       that for every $m \in \ZZd$, we have
       \begin{align*}
           x(m) = c_{\alpha}\circ (A^{-1}m) 
           & = \sum\limits_{i=1}^d \left(\lfloor\alpha_i+(A^{-1}m)\cdot\balpha\rfloor
           -\lfloor (A^{-1}m)\cdot\balpha\rfloor\right)\\
           & = \sum\limits_{i=1}^d \left(\lfloor\alpha_i+m\cdot (A\balpha)\rfloor
           -\lfloor m\cdot (A\balpha)\rfloor\right)
            = c_{A\alpha }(m).
       \end{align*} 
       Similarly, we get that $y(m) = c'_{\alpha}\circ A^{-1}(m) = c'_{A \alpha }(m)$ for every $m \in \ZZd$. Thus we obtain that $x = c_{A \alpha }$ and $y = c'_{A \alpha }$, as required.
\end{proof}

We finish this section by extending our result to~\Cref{maincorollary:affine_multidim_sturmian_characterization}. Let us briefly recall the definition of the affine flip condition.

\begin{definition}\label{def:affine-flip-condition}
	We say that an indistinguishable asymptotic pair $x,y \in \Sigma^{\ZZd}$ with difference set $\underline{F}$ satisfies the \define{affine flip condition} if:
	\begin{enumerate}
		\item there is $m \in \underline{F}$ such that 
            $(\underline{F}-m)\setminus\{0\}$ 
            is a base of $\ZZd$,
        \item the restriction $x|_{\underline{F}}$ is a bijection $\underline{F}\to\Sigma$,
        \item the map $x_n \mapsto y_n$ for all $n\in\underline{F}$ induces a
            cyclic permutation on $\Sigma$.
	\end{enumerate}
\end{definition}

Notice that the first condition of \Cref{def:affine-flip-condition} implies
that $\#\underline{F} = d+1$.

Let us also recall that~\Cref{maincorollary:affine_multidim_sturmian_characterization} states if $x,y\in\Sigma^\ZZd$ is such that $x$ is uniformly recurrent, then
the pair $(x,y)$ is an indistinguishable asymptotic pair satisfying the
affine flip condition if and only if 
there exists a bijection $\tau\colon \alfad\to\Sigma$,
there exists an invertible affine transformation $A\in \operatorname{Aff}(\ZZd)$
and there exists a totally irrational vector $\alpha\in[0,1)^d$
such that $x=\tau\circ c_{\alpha}\circ A$ and $y=\tau\circ c'_{\alpha}\circ A$.

\begin{proof}[Proof of~\Cref{maincorollary:affine_multidim_sturmian_characterization}]
	Recall first that the property that a configuration is uniformly recurrent is invariant under affine transformations of $\ZZd$ and sliding-block codes. We shall use this fact implicitly in this proof.
	
	Suppose first that $x=\tau \circ c_{\balpha} \circ A$
	and $y=\tau \circ c_{\balpha} \circ A$ for some $A \in \operatorname{Aff}(\ZZd)$ and $\tau \colon \alfad\to\Sigma$. By~\Cref{thm:multidim_sturmian_characterization}, $c_{\alpha},c'_{\alpha}$ form an indistinguishable asymptotic pair which satisfies the flip condition. By~\Cref{prop:shifted_SI} and~\Cref{prop:invariance_sliding_block_code}, we obtain that $x = \tau \circ  c_{\balpha} \circ A, y = \tau \circ c'_{\balpha} \circ A$ is an indistinguishable asymptotic pair. Let us show that they satisfy the affine flip condition. As $\tau,A$ are bijections, it is clear that if $F$ is the difference set of $x,y$, then $\underline{F} = A^{-1}(F)$. It follows that $\#\underline{F} = d+1$, that $\{ A^{-1}(n) - A^{-1}(0) : n \in \{-\be_1,\cdots,-\be_d\}\}$ is a base of $\ZZd$, that the restriction $x|_{\overline{F}}$ is a bijection, and that the map $x_n \mapsto y_n$ induces a cyclic permutation on $\Sigma$. That is, $x,y$ satisfy the affine flip condition.
	
	Conversely, if $x,y$ satisfy the affine flip condition there is $m \in \underline{F}$ such that 
            $B=(\underline{F}-m)\setminus\{0\}$ 
    is a base of $\ZZd$. Construct an integer matrix $\textbf{B} \in \operatorname{GL}_d(\ZZ)$ by putting elements of $B = \{b_1,\dots,b_d\}$ in its columns. Let $A^{-1}\in \operatorname{Aff}(\ZZd)$ be the affine transformation such that $n \mapsto m + \textbf{B}n$ and notice that it maps $F$ onto $\underline{F}$ sending $0$ to $m$. By the second and third conditions of the affine flip condition, there is a unique bijection $\tau^{-1} \colon \Sigma \to \alfad$ such that $\tau^{-1}(x_m) = 0$ and $\tau^{-1}(x_n) = \tau^{-1}(y_n)-1 \bmod{d+1}$ for every $n \in \underline{F}$. It follows directly from the choices of $A^{-1}$ and $\tau^{-1}$ that $\tau^{-1} \circ x \circ A^{-1},\tau^{-1} \circ y \circ A^{-1}$ satisfy the flip condition. Furthermore, by~\Cref{prop:shifted_SI,prop:invariance_sliding_block_code} they form an indistinguishable asymptotic pair. Hence  by~\Cref{thm:multidim_sturmian_characterization} it follows that there is a totally irrational vector $\alpha\in [0,1]^d$ such that $\tau^{-1} \circ x \circ A^{-1}= c_{\alpha}$ and $\tau^{-1} \circ y \circ A^{-1}= c'_{\alpha}$. Thus $x = \tau \circ c_{\alpha}\circ A$ and $y = \tau \circ c'_{\alpha}\circ A.$
\end{proof}

\begin{proof}[Proof of \Cref{corollary:language-vs-existence-alpha}]
It follows directly from \Cref{maintheorem:ddim-complexity-is-FminusS}
and~\Cref{thm:multidim_sturmian_characterization}.
\end{proof}

\appendix

\section{Indistinguishable pairs on countable groups}\label{sec:appendix}

As mentioned in~\Cref{sec:preliminaries}, the results in the first part of that section can be stated and proven in the context of an arbitrary countable group $\Gamma$. At this moment we do not have any interesting application in this context, but in order to avoid senseless repetition in potential future work, we provide proofs of those statements in this appendix.

Let $\Sigma$ be a finite set which we call \define{alphabet} and $\Gamma$ a countable group. An element $x \in \Sigma^{\Gamma} = \{x \colon \Gamma \to \Sigma  \}$ is called a \define{configuration}. For $g \in \Gamma$, let $x_g$ denote the value $x(g)$. The set $\Sigma^{\Gamma}$ of all configurations is endowed with the prodiscrete topology.

The (left) \define{shift action} $\Gamma \overset{\sigma}{\curvearrowright} \Sigma^{\Gamma}$ (by right multiplication) is given by the map $\sigma\colon \Gamma \times \Sigma^{\Gamma}\to \Sigma^{\Gamma}$ where
\[ \sigma^g(x)_h \isdef \sigma(g,x)_h =  x_{hg} \quad \mbox{ for every } g,h \in \Gamma, x \in \Sigma^{\Gamma}.  \]

\begin{remark}
	We may alternatively consider the left action by left multiplication given by $\sigma^g(x)_h =  x_{g^{-1}h}$ for every $g,h \in \Gamma$ and $x \in \Sigma^{\Gamma}$. Here we chose right multiplication to be consistent with the definition on $\ZZ^d$. All proofs below are also valid with this choice.
\end{remark}

Two configurations $x,y$ are \define{asymptotic} if the set $F = \{ g \in \Gamma \colon x_g \neq y_g\}$ is finite. $F$ is called the \define{difference set} of $(x,y)$. If $x=y$ we say that the asymptotic pair is \define{trivial}.

For finite $S \subset \Gamma$, an element $p \in \Sigma^S$ is called a \define{pattern} and the set $S$ is its \define{support}. Given a pattern $p \in \Sigma^S$, the \define{cylinder} centered at $p$ is $[p] = \{ x \in \Sigma^{\Gamma} \colon x|_S = p \}$. A pattern $p$ \define{appears} in $x \in \Sigma^{\Gamma}$ if there exists $g \in \Gamma$ such that $\sigma^g(x) \in [p]$. We also denote by $\occ_p(x) = \{g \in \Gamma \colon \sigma^g(x) \in [p]\}$ the set of \define{occurrences} of $p$ in $x \in \Sigma^{\Gamma}$.

For finite $S \subset \Gamma$, the \define{language with support $S$} of a configuration $x$ is the set of patterns \[\Lcal_{S}(x) = \{ p \in \Sigma^S : \mbox{ there is } g \in \Gamma \mbox{ such that } \sigma^g(x) \in [p]\}. \]
The \define{language of $x$} is the union $\Lcal(x)$ of the sets $\Lcal_S(x)$ for every finite $S \subset \Gamma$. 

\begin{definition}
	We say that two asymptotic configurations $x$ and $y$ in $\Sigma^{\Gamma}$ are \define{indistinguishable} if for every pattern $p$ we have \[ \#(\occ_p(x)\setminus \occ_p(y)) = \#(\occ_p(y)\setminus \occ_p(x)). \]
\end{definition}

For a pattern $p \in \Sigma^S$, its discrepancy in $x,y$ is given by \[ \Delta_p (x,y) \isdef \sum_{g \in S^{-1}F } \indicator{[p]}(\sigma^g(y))-\indicator{[p]}(\sigma^g(x)).\]

It is clear that the following conditions are equivalent:
	\begin{enumerate}
		\item $x$ and $y$ are indistinguishable asymptotic configurations with difference set $F$,
		\item for every pattern $p$ with finite support $S \subset \ZZd$, we have 
		\[
		\#\left(\occ_p(x)\cap S^{-1}F\right) = 
		\#\left(\occ_p(y)\cap S^{-1}F\right),
		\] 
		\item for every pattern $p$ with finite support $S \subset \ZZd$, we have $\Delta_p (x,y) = 0$.
	\end{enumerate}

\begin{proposition}\label{prop:trivialite2}
	Let $S_1 \subset S_2$ be finite subsets of $\Gamma$, and let $p \in \Sigma^{S_1}$. 
	We have \[ \Delta_p(x,y) = \sum_{q \in \Sigma^{S_2}, [q] \subset [p] } \Delta_q(x,y).\]
\end{proposition}

\begin{proof}
	Notice that $[p]$ is the disjoint union of all $[q]$ where $q \in \Sigma^{S_2}$ and $[q] \subset [p]$.  It follows that for any $z \in \Sigma^{\Gamma}$ we have $ \indicator{[p]}(z) =  1$ if and only if there is a unique $q \in \Sigma^{S_2}$ such that $[q] \subset [p]$ and $\indicator{[q]}(z)=1$. Letting $F$ be the difference set of $x,y$ we obtain,
	
	\begin{align*}
		\Delta_p(x,y) & = \sum_{g \in S_1^{-1}F}\indicator{[p]}(\sigma^g(y))-\indicator{[p]}(\sigma^g(x)) \\
		& = \sum_{g \in S_2^{-1}F}\indicator{[p]}(\sigma^g(y))-\indicator{[p]}(\sigma^g(x)) \\
		& = \sum_{g \in S_2^{-1}F}  \sum_{\substack{q \in \Sigma^{S_2} \\ [q] \subset [p]} }\indicator{[q]}(\sigma^g(y))-\indicator{[q]}(\sigma^g(x)).
	\end{align*}
	Exchanging the order of the sums yields the result.\end{proof}

Let us denote the group of automorphisms of $\Gamma$ by $\operatorname{Aut}(\Gamma)$.

\begin{proposition}\label{prop:shifted_SI2}
	Let $(x,y)$ be an indistinguishable asymptotic pair, then
	\begin{enumerate}
		\item $(\sigma^g(x),\sigma^g(y))$ is an indistinguishable asymptotic pair for every $g \in \Gamma$.
		\item $(x\circ \varphi, x \circ \varphi)$ is an indistinguishable asymptotic pair for every $\varphi \in \operatorname{Aut}(\Gamma)$.
	\end{enumerate}
\end{proposition}

\begin{proof}
	Let $F$ be the difference set of $(x,y)$. A straightforward computation shows that the difference set of $(\sigma^g(x),\sigma^g(y))$ is $F_1 = Fg^{-1}$ and the difference set of $(x\circ \varphi, x \circ \varphi)$ is $F_2 = \varphi^{-1}(F)$.
	
	Let $S \subset \Gamma$ be a finite set and $p \in \Sigma^S$. For the first claim we have
	
	\begin{align*}
		\Delta_p(\sigma^g(x),\sigma^g(y)) & = \sum_{h \in S^{-1}F_1}\indicator{[p]}(\sigma^{h}(\sigma^g(y)))-\indicator{[p]}(\sigma^h( \sigma^g(y)  ))\\
		& = \sum_{h \in S^{-1}Fg^{-1}}\indicator{[p]}(\sigma^{hg}(y))-\indicator{[p]}(\sigma^{hg}(y))\\
		& = \sum_{t \in S^{-1}F}\indicator{[p]}(\sigma^{t}(y))-\indicator{[p]}(\sigma^{t}(y)) = \Delta_p(x,y) = 0.
	\end{align*}
	Thus $ (\sigma^g(x),\sigma^g(y))$ is an indistinguishable asymptotic pair.
	
	For the second claim, let $q \in \Sigma^{\varphi(S)}$ be the pattern given by $q(\varphi(s))=p(s)$ for every $s \in S$. We note that for any $h \in \Gamma$, $\sigma^h(x)\in [q]$ if and only if $\sigma^{\varphi^{-1}(h)}(x \circ \varphi) \in [p]$. This means that $h \in \occ_q(x)$ if and only if $\varphi^{-1}(h) \in \occ_p(x \circ \varphi)$.
	
	As $(x,y)$ is an indistinguishable asymptotic pair, there is a finitely supported permutation $\pi$ of $\Gamma$ so that $\occ_q(x) = \pi(\occ_q(y))$. Then $\pi' = \varphi \circ \pi \circ \varphi^{-1}$ is a finitely supported permutation of $\Gamma$ so that $\occ_p(x\circ \varphi)= \pi'(\occ_p(y\circ \varphi))$. We conclude that $\Delta_p(x\circ\varphi,y\circ \varphi)=0$ and thus they are indistinguishable.\end{proof}

Let $\Sigma_1,\Sigma_2$ be alphabets. A map $\phi \colon \Sigma_1^{\Gamma} \to \Sigma_2^{\Gamma}$ is a \define{sliding block code} if there exists a finite set $D \subset \Gamma$ and map $\Phi \colon \Sigma_1^D \to \Sigma_2$ called the \define{block code} such that

\[ \phi(x)_{g} = \Phi( \sigma^{g}(x)|_{D})  \mbox{ for every } g \in \Gamma, x \in \Sigma_1^{\Gamma}.\]

\begin{proposition}\label{prop:invariance_sliding_block_code2}
	Let $x,y \in \Sigma_1^{\Gamma}$ be an indistinguishable asymptotic pair and $\phi \colon \Sigma_1^{\Gamma} \to \Sigma_2^{\Gamma}$ a sliding block code. The pair $\phi(x),\phi(y) \in \Sigma_2^{\Gamma}$ is also an indistinguishable asymptotic pair.
\end{proposition}

\begin{proof}
	Let $F$ be the difference set of $x,y$ and $D \subset \Gamma$, $\Phi \colon \Sigma_1^D \to \Sigma_2$ be the set and block code which define $\phi$. If $g \notin D^{-1}F$, then $\sigma^{g}(x)|_D = \sigma^{g}(y)|_D$ and thus $\phi(x)_{g} = \phi(y)_{g}$. As $D^{-1}F$ is finite, it follows that $\phi(x),\phi(y)$ are asymptotic.
	
	Let $S\subset\Gamma$ be finite and $p\colon S\to\Sigma_2$ be a pattern.
	Let $\phi^{-1}(p) \subset (\Sigma_1)^{DS}$ be the set of patterns $q$ so that for every $s \in S$, $\Phi((q_{ds})_{d \in D}) = p_s$. It follows that $\phi^{-1}([p]) = \bigcup_{q \in \phi^{-1}(p)}[q]$.
	
	Let $W \subset \Gamma$ be a finite set which is large enough such that $W \supseteq F\cup DF$. We have,
	\begin{align*}
		\#\{g \in S^{-1}W \mid \sigma^g(\phi(x))\in[p]\}
		&= \sum_{q \in \phi^{-1}(p)}\#\{g\in S^{-1}W \mid \sigma^g(x)\in[q]\}\\
		&= \sum_{q \in \phi^{-1}(p)}\#\{g\in S^{-1}W \mid \sigma^g(y)\in[q]\}\\
		&= \#\{g\in S^{-1}W \mid \sigma^g(\phi(y))\in[p]\}.
	\end{align*}
	Taking $W$ large enough such that $W \supseteq F\cup DF$, we conclude that $(\phi(x), \phi(y))$
	is an indistinguishable asymptotic pair.
\end{proof}

	Let $(x_n,y_n)_{n \in \NN}$ be a sequence of asymptotic pairs. We say that $(x_n,y_n)_{n \in \NN}$ \define{converges in the asymptotic relation} to a pair $(x,y)$ if $(x_n)_{n \in \NN}$ converges to $x$, $(y_n)_{n \in \NN}$ converges to $y$, and there exists a finite set $F \subset \Gamma$ so that $x_n|_{\Gamma \setminus F} = y_n|_{\Gamma \setminus F}$ for all large enough $n \in \NN$. We say that $(x,y)$ is the \define{\'etale limit} of $(x_n,y_n)_{n \in \NN}$.

\begin{proposition}\label{prop:limit-of-indist-is-indist2}
	Let $(x_n,y_n)_{n \in \NN}$ be a sequence of asymptotic pairs in $\Sigma^{\Gamma}$ which converges in the asymptotic relation to $(x,y)$. If for every $n \in \NN$ we have that $(x_n,y_n)$ is indistinguishable, then $(x,y)$ is indistinguishable.
\end{proposition}

\begin{proof}
	Let $p \in \Sigma^S$ be a pattern. As $(x_n,y_n)_{n \in \NN}$ converges in the asymptotic relation to $(x,y)$, there exists a finite set $F \subset \Gamma$ and $N_1 \in \NN$ so that $x_n|_{\Gamma \setminus F} = y_n|_{\Gamma \setminus F}$ for every $n \geq N_1$. In particular we have that the difference sets of $(x,y)$ and $(x_n,y_n)$ for $n \geq N_1$ are contained in $F$. It suffices thus to show that \[ \# \{ \occ_p(x) \cap S^{-1}F\} = \#\{ \occ_p(y) \cap S^{-1}F\}.   \]
	As $(x_n)_{n\in \NN}$ converges to $x$ and $(y_n)_{n} \in \NN$ converges to $y$, there exists $N_2 \in \NN$ so that $x_n|_{SS^{-1}F} = x|_{SS^{-1}F}$ and $y_n|_{SS^{-1}F} = y|_{SS^{-1}F}$ for all $n \geq N_2$. This implies that $\occ_p(x) \cap S^{-1}F = \occ_p(x_n) \cap S^{-1}F$ and $\occ_p(y) \cap S^{-1}F = \occ_p(y_n) \cap S^{-1}F$ for every $n \geq N_2$.
	
	Let $N = \max\{N_1,N_2\}$ and let $n \geq N$. As $n \geq N_1$, we have that $(x_n,y_n)$ is an indistinguishable asymptotic pair whose difference set is contained in $F$, it follows that $\# \{ \occ_p(x_n) \cap S^{-1}F\} = \#\{ \occ_p(y_n) \cap S^{-1}F\}.$ As $n \geq N_2$, we obtain $\# \{ \occ_p(x) \cap S^{-1}F\} = \#\{ \occ_p(y) \cap S^{-1}F\}.$ As this argument holds for every pattern $p$, we conclude that $(x,y)$ is indistinguishable.
\end{proof}

	A configuration $x \in \Sigma^{\Gamma}$ is \define{recurrent} if for every $p \in \Lcal(x)$ we have that $\occ_p(x)$ is infinite.

\begin{proposition}\label{prop:recurrence2}
	Let $x,y \in \Sigma^{\Gamma}$ be an indistinguishable asymptotic pair. If $x$ is not recurrent, then $x,y$ lie in the same orbit.
\end{proposition}

\begin{proof}
	If $x$ is not recurrent, there is a finite $S \subset \Gamma$ and $p \in \Lcal_S(x)$ such that $\occ_p(x)$ is finite. As $\Delta_p(x,y)=0$, it follows that $\occ_p(y)$ is also finite.
	
	Let $(S_n)_{n \in \NN}$ be an increasing sequence of finite subsets of $\Gamma$ such that $S_0 = S$ and $\bigcup_{n \in \NN}S_n = \Gamma$ and let $q_n = x|_{S_n}$. As $x \in [q_n]$ and $\Delta_{q_n}(x,y)=0$, there exists $g_n \in \Gamma$ so that $\sigma^{g_n}(y) \in [q_n]$. Furthermore, as $q_n|_S = p$, it follows that $\sigma^{g_n}(y) \in [p]$ and thus $g_n \in \occ_p(y)$. As $\occ_p(y)$ is finite, there exists $h \in \occ_p(y)$ and a subsequence such that $g_{n(k)} = h$ and thus $\sigma^{h}(y) \in [q_{n(k)}]$ for every $k \in \NN$. As  $\bigcap_{n \in \NN}[q_n] = \bigcap_{k \in \NN}[q_{n(k)}] = \{x\}$ we deduce that $\sigma^{h}(y) = x$.
\end{proof}

\bibliographystyle{abbrv}
\bibliography{bibliofj}

\end{document}